\documentclass[11pt]{amsart}
\usepackage{amssymb, latexsym, mathrsfs,   color, amsmath }
\usepackage[all]{xy}
\usepackage[colorlinks=true, pdfstartview=FitV,linkcolor=blue,citecolor=blue,urlcolor=blue]{hyperref}
\usepackage{chngcntr}
\usepackage{amsmath}
\usepackage{arydshln}
\usepackage{forest}
\counterwithin{figure}{section}

    \advance \topmargin by -\headheight
    \advance \topmargin by -\headsep

    \evensidemargin \oddsidemargin
\newtheorem {theorem}    {Theorem}[section]

\newtheorem {lemma}      [theorem]    {Lemma}
\newtheorem {corollary}  [theorem]    {Corollary}

\renewcommand{\rm}{\mathrm}

\newcommand{\GGL}{\mathrm{GL}}

\theoremstyle{definition}

\newtheorem{remark}[theorem]{Remark}

\newcommand{\R}{\mathrm{R}}

\numberwithin{equation}{section}

\begin{document}

\title{The coarse trace formula of $\rm{GL}(4)$}

\date{\today}

\author{Haoyang Wang,\ \ Xinghua Cui,\ \ Zhifeng Peng*}

\address{School of Mathematical Science, Soochow University, Suzhou 310027, Jiangsu, P.R. China}

\maketitle

\begin{abstract}
Trace formula is an important method to study the Langlands program.  Arthur obtains the existence of stable trace formula for connected reductive group. In this paper, we will give the explicit coarse trace formula of GL(4). In general case, Arthur applies the truncation operator on the two sides of trace formula, which is convergent. In our case, we will prove that the divergent terms of the two sides of the trace formula of $\rm{GL}(4)$ are equal. We also obtain the explicit formula for ramified orbits of the geometric side of trace formula of GL(4).

\end{abstract}

\section{introduction}
 Arthur-Selberg trace formula is the important tool to study the automorphic representation of the connected reductive group $G$. Selberg \cite{S1},\cite{S2} gave a formula for the trace of 
a certain operator associated with a compact quotient of a semisimple Lie group and a discrete subgroup. Assume that $\mathbb{A}$ is the adeles ring, then $G_\mathbb{A}$ is a locally compact topological group and $G_\mathbb{Q}$ is a discrete subgroup. An automorphic representation is an irreducible representation of the decomposition of the right regular representation $\R$ which is the right action of $G_\mathbb{A}$ on the space $L^2(G_\mathbb{Q} \backslash G_\mathbb{A})$. If $\phi\in L^2(G_\mathbb{Q} \backslash G_\mathbb{A})$, we have 
\[(\R(y)\phi)(x)=\phi(xy),\qquad x,y\in G_\mathbb{A}.\]
Then $\R$ is a unitary representation of $G_\mathbb{A}$. The trace formula is to consider the involution operator 
\[\R(f) = \int_{G(\mathbb{A})}f(y)\R(y)dy\] 
for $f$ in $C_c^\infty(G(\mathbb{A}))$.  
 Then 
\begin{equation*}
    \begin{split}
     (\R(f)\phi)(x)&=\int_{G_{\mathbb{A}}}f(y)\phi(xy)dy       \\
                 &=\int_{G_{\mathbb{Q}}\backslash G_\mathbb{A}} \{\sum_{\gamma\in G_{\mathbb{Q}}}f(x^{-1}\gamma y)\}\phi(y)dy . 
    \end{split}
\end{equation*}
  We denote its kernel by 
  \[K(x,y)=\sum_{\gamma\in G_{\mathbb{Q}}}f(x^{-1}\gamma y)\]
 If the quotient $G_{\mathbb{Q}}\backslash G_{\mathbb{A}}$ is compact, then we have two natural ways to expand the kernel
 \[ K(x,y)=\sum_{\mathfrak{o}\in\mathcal{O}}K_{\mathfrak{o}}(x,y)\]
   and 
   \[K(x,y)=\sum_{\chi\in \mathscr{X}}K_{\chi}(x,y).\]
  Where $\mathcal{O}$ is the set of conjugacy classes in the group $G_\mathbb{Q}$, $\mathscr{X}$ is the set of unitary equivalence classes of irreducible representations of $G_\mathbb{A}$, and the restriction of the regular representation $\R$ to the subspace $(L^{2}(G_{\mathbb{Q}}\backslash G_{\mathbb{A}}))_{\chi}$ is equivalent to a finite number of copies of $\chi$. We denote $\mathscr{B}_{\chi}$ by an orthonormal basis of $(L^{2}(G_{\mathbb{Q}}\backslash G_{\mathbb{A}}))_{\chi}$ for each $\chi\in\mathscr{X}$. Then
   \[K_{\mathfrak{o}}(x,y)=\sum_{\gamma\in\mathfrak{o}}f(x^{-1}\gamma y),\qquad \mathfrak{o}\in\mathcal{O}\]
   and 
   \[ K_{\chi}(x,y)=\sum_{\phi\in\mathscr{B}_{\chi}}(\R(f)\phi)(x)\cdot\overline{\phi(y)}.\]
The Arthur-Selberg trace formula comes 
from integrating both formulas for the kernel over the diagonal. Thus we obtain the trace formula
\[\sum_{\mathfrak{o}\in\mathcal{O}}J_{\mathfrak{o}}(f)=\sum_{\chi\in\mathscr{X}}J_{\chi}(f),\]
where $J_{\mathfrak{o}}(f)$ is the integral over $x$ in $G_{\mathbb{Q}}\backslash G_{\mathbb{A}}$ of $K_{\mathfrak{o}}$, and 
$J_{\mathfrak{\chi}}(f)$ is the integral over $x$ in $G_{\mathbb{Q}}\backslash G_{\mathbb{A}}$ of $K_{\mathfrak{\chi}}(x,x)$.

  If the quotient $G_{\mathbb{Q}}\backslash G_{\mathbb{A}}$ is non-compact, then $\R$ contains continuous representations for any parabolic subgroup $P$ of $G$ over $\mathbb{Q}$, the intertwining operators are provided by Eisenstein series, thus $\mathscr{X}$ must be defined as $\mathscr{T}(G)$ in $\S 4$, in terms of cuspidal automorphic representations of Levi components of parabolic subgroup
 of $G$. However the definition of $\mathcal{O}$ is the set of the 
equivalence classes composed of those elements in $G_\mathbb{Q}$ whose semisimple component are $G_{\mathbb{Q}}$-conjugate. Then we still have an identity
\begin{eqnarray}\label{formal formula 1}
    \sum_{\mathfrak{o}\in\mathcal{O}}K_{o}(x,y)=\sum_{\chi\in\mathscr{X}}K_{\chi}(x,y)
\end{eqnarray}
by equating two different formulas for the kernel of $\R(f)$. 

   However, if $\chi$ stands for a continuous representation, then the integration of $K_{\chi}(x,x)$ over $x$ in $G_{\mathbb{Q}}\backslash G_{\mathbb{A}}$ is divergent.  If $\mathfrak{o}$ meets a group $P_{\mathbb{Q}}$ which is a proper parabolic subgroup of $G$ over $\mathbb{Q}$, then the integration of $K_{\mathfrak{o}}(x,x)$ over $x$ in $G_{\mathbb{Q}}\backslash G_{\mathbb{A}}$ is also divergent. For such parabolic $P$, Arthur defined an Arthur truncation operator $\Lambda^T$ on $K_{P,\mathfrak{o}}(x,x)$ and $K_{P,\chi}(x,x)$, we denote the truncation functions by $K^{T}_{\mathfrak{o}}(x,f)$ and $K^{T}_{\chi}(x,f)$, then we have an identity
   \begin{equation} \label{Truncation function equation}
    \sum_{\mathfrak{o}\in\mathcal{O}} K^{T}_{\mathfrak{o}}(x,f)=\sum_{\chi\in\mathscr{X}}K^{T}_{\chi}(x,f).
   \end{equation}
   Arthur \cite{A3} showed that each side of the identity (\ref{Truncation function equation}) is integrable and the integrals can be taken inside the sums. If $J^{T}_{\mathfrak{o}}(f)$ and $J^{T}_{\chi}(f)$ stand for the integrals of the summations, we then have the coarse trace formula
   \begin{eqnarray}\label{weight trace formula}
    \sum_{\mathfrak{o}\in\mathcal{O}} J^{T}_{\mathfrak{o}}(f)=\sum_{\chi\in\mathscr{X}}J^{T}_{\chi}(f).
   \end{eqnarray}
   However a natural question is whether the truncation operator retains all the information of automorphic representations? For $\GGL(4)$ case,  we show that the remainder terms of the geometric side and the spectral side will cancel each other, the following theorem is the first main result.

\begin{theorem}(Theorem \ref{thm1})
    For any $f\in C_c^\infty(Z_\infty^+\backslash G_\mathbb{A})$,\[J_{\rm{geo}}^{\rm{d}}(f)=J_{\rm{spec}}^{\rm{d}}(f).\]
\end{theorem}
Where the left hand side is the divergent terms of the geometric side and the right hand side is the divergent terms of the spectral side.  

In general case, Arthur did not give the concrete distributions of ramified orbits. In this paper, the other main result is to obtain the explicit formula for distributions of ramified orbits of $\GGL(4)$.

\begin{theorem}(Theorem \ref{thm2})
    For ramified orbits, the integrals of the kernel over $Z_\infty^+G_\mathbb{Q}\backslash G_\mathbb{A}$ is the sum
    \begin{eqnarray*}
        &\rm{lim}_{\lambda\rightarrow 0}\int_{Z_\infty^+G_\mathbb{Q}\backslash G_\mathbb{A}}D_\lambda\{\lambda\mu_{\mathfrak{o}_{1111}^4}(\lambda,f,x)\}dx
    \end{eqnarray*}
    \begin{eqnarray*}
        &+\rm{lim}_{\lambda\rightarrow 0}\int_{Z_\infty^+G_\mathbb{Q}\backslash G_\mathbb{A}}D\lambda\{\lambda\mu_{\mathfrak{o}_{22}^2}(\lambda,f,x)\}dx
    \end{eqnarray*}
    \begin{eqnarray*}
        &+\sum_{\substack{\rm{ramified}\ \mathfrak{o},\\\mathfrak{o}\neq\mathfrak{o}_{1111}^4,\mathfrak{o}_{22}^2}}c_{P_{\{\mathfrak{o}\}}}a_{P_{\{\mathfrak{o}\}}}\sum_{\gamma\in M_{t,\{\mathfrak{o}\}}^\mathfrak{o}}\tilde{\tau}(\gamma,M)\int_K\int_{N_{\{\mathfrak{o}\},\mathbb{A}}}\int_{M_{\{\mathfrak{o}\}}(\gamma)_{\mathbb{A}}\backslash M_{\{\mathfrak{o}\},\mathbb{A}}}\\\notag
        &f(k^{-1}n^{-1}m^{-1}\gamma mnk)v_{M_{\{\mathfrak{o}\}}}(m)dm\ dn\ dk.
    \end{eqnarray*}
\end{theorem}

 The limit terms in Theorem \ref{thm2} can be expressed as the unramified orbit integral, but we do not write down in this paper.
We will give the coarse trace formula of $\rm{GL}(4)$ in Theorem \ref{result}.

In general, the cuspidal part of the spectrum terms in (\ref{Truncation function equation}) is indeed of trace class, then we have 
\begin{eqnarray}\label{1.3}
\mathrm{Tr}(\R_\mathrm{cusp}(f))=\sum_{\mathfrak{o}}J_\mathfrak{o}^T(f)-\sum_{\chi\in\mathscr{X}-\mathscr{X}(G)}J_\chi^T(f),
\end{eqnarray}
(see \cite{A4}), where the index of the terms of spectrum means that are not cuspidal part. We called this formula for Arthur's coarse trace formula. 

The following is to introduce the sketch of each section:
In the section 3, we classify the orbits of $G$, according to the eigenpolynomial of a given element in $G_\mathbb{Q}$.  We then show that the orbits are correspond to the standard parabolic subgroups.
In the section 4, we recall the theory of Eisenstein series, which is  developed by Harish-Chandra, Langlands and so on. Then we can  decompose the spectrum of $G$. 

In the section 5, we prove that the discrete series is of trace class, associated with the test function. We shall find a correspondence between $\mathfrak{o}$ and $P$, then we shall give a formula of $K_\mathfrak{o}(x,x)$ associated to $P$, if an orbit $\mathfrak{o}$ is given. So that the integrals of geometric terms can connect with spectral terms. For ramified orbits, Arthur \cite{A15} gave the existence of the formula of them. But we shall give an explicit formula of the ramified orbits. In general, the integrals of both sides of (\ref{formal formula 1}) are divergent, we can write them into the terms which are convergent and not convergent. The key is how we can cancel the divergent terms. In this article, we shall explain how to cancel the divergent terms by the difference of geometric terms and spectral terms. We also introduce Arthur's truncation operator which is associated to a sum of characteristic functions, to control the convergence of the integrals of $K_\mathfrak{o}(x,x)$, and the integrals associated to truncation operator are the convergent terms which we shall obtain. In fact, the two sides of (\ref{weight trace formula}) is a polynomial associated to $T$. The equality (\ref{1.3}) is true for infinite number of $T$, thus we can obtain that the coefficients of $T^k$ vanish except $k=0$. So, we can take $T=0$ to obtain the trace formula.

In the section 6, we shall obtain an explicit formula of the integral of ramified orbits, which is one of our main results. In the section 7, we prove the convergence of some special cases, and give some lemmas which are the basis of all convergence of integral. In the section 8, we prove the convergence of integrals of ramified orbits. In Sections 9 to 12, we calculate the difference between the integrals of geometric and spectral terms associated to $P_{31}$,$P_{22}$,$P_{211}$ and $P_{1111}$ respectively. Then we can find that all the terms left are without the parameter $T$, that is, the terms divergent can all be canceled. Then we can obtain the coarse trace formula of $\rm{GL}(4)$.

Acknowledgement. I am deeply grateful to Arthur for his assistance and encouragement. We acknowledge generous support provided by National natural Science
Foundation of PR China (No. 12071326).

\section{Preliminaries}
\subsection{Some general definition}
For any place $v$ of $\mathbb{Q}$, we write $G_v$ for $G_{\mathbb{Q}_v}$, the group of $\mathbb{Q}_v$-rational points of $G$. Denote the adele of $\mathbb{Q}$ by $\mathbb{A}$, and we denote $G_{\mathbb{A}}$ for the corresponding adele group.

Let $C_c^\infty(G_\mathbb{A})$ represents the space of linear combination of functions $f=\prod_vf_v$ satisfing :
\begin{enumerate}
    \item If $v$ is infinite, $f_v\in C_c^\infty(G_v)$
    \item If $v$ is finite, $f_v$ is locally constant and has compact support
    \item For almost all finite places $v$, $f_v$ is the characteristic function of $G_{O_v}$, where $O_v$ is the algebraic integer ring of $\mathbb{Q}_v$.
\end{enumerate}
For connected reductive algebraic group $G$, denote $X(G)_\mathbb{Q}$ by the group of $\mathbb{Q}$-rational characters of $G$ and $A_G$ by the split component of $G$. $X(G)_\mathbb{Q}$ is a free abelian group, then we have the vector space \[\mathfrak{a}_G=\rm{Hom}_\mathbb{Z}(X(G)_\mathbb{Q}, \mathbb{R})\] and \[\mathfrak{a}^*=X(G)_\mathbb{Q}\otimes\mathbb{R}. \]
Then we define a map \[H_G:G_\mathbb{A}\rightarrow\mathfrak{a}_G\] by \[e^{<\chi, H_G(x)>}=\left| \chi(x) \right|,\quad x\in G_\mathbb{A},\, \chi\in X(G)_\mathbb{Q}. \]The kernel of $H_G$ will be denoted by $G_\mathbb{A}^1$. We can then decompose $G_\mathbb{A}$ into \[Z_{\infty}^+ \times G_\mathbb{A}^1,\] the group $Z_{\infty}^+$ is independent of the basis of $X(G)_\mathbb{Q}$.

We fix a minimal parabolic subgroup $P_0$ which equals the Borel subgroup of $G$ with a decomposition $P_0=M_0N_0$. We call a parabolic subgroup $P$ is standard, if $P\supset P_0$. For example, for $\rm{GL}(n)$, we take $P_0$ to be the upper triangular matrix. For any parabolic subgroup $P\supset P_0$, there exists a decomposition $P=M_PN_P$ such that $M_P\supset M_0, A_P\subset A_0$, $A_P$ is the split component of $P$. Unless otherwise specified, we only consider the standard parabolic subgroups.

Then by the Iwasawa decomposition,\[G=PK=M^1_\mathbb{A}N_\mathbb{A}AK,\]$P$ is a parabolic subgroup, $M$ is the Levi associated to $P$, $N$ is the unipotent radical of $P$.

Then function $H_P=H_M$ can be defined similarly as $H_G$.

We have to adopt some conventions for choices of Haar measures.  Write the Haar measure $dx$ on $G_\mathbb{A}$. We take the Haar measure $dk$ of $K$ such that \[\int_Kdk=1.\] Fix Haar measures on each of the vector spaces $\mathfrak{a}_P$.  We take the dual Haar measures on the spaces $\mathfrak{a}_P^*$ .

Any basis $\chi_1, . . . , \chi_r$ of $X(G)_\mathbb{Q}$ defines an isomorphism between $Z_{\infty}^+$ and $(\mathbb{R}_+^*)^r$. Take the measure on $Z_{\infty}^+$ which corresponds to the Euclidean measure on $(\mathbb{R}_+^*)^r$, this is independent of the choice of the basis $\chi_1, . . . , \chi_r$. Then we define a measure on $G_\mathbb{A}^1$ which we also denote by $dx$. In fact, the number \[\tau(G)=\int_{G_\mathbb{Q}\backslash G_\mathbb{A}^1}dx=\int_{Z_\infty^+G_\mathbb{Q}\backslash G_\mathbb{A}}dx\] is finite, and $\tau(G)$ is the Tamagama number of $G$.

We take $\mathfrak{a}_P=\mathfrak{a}_P^*$. Write $\Sigma_P$ the roots of $P$,  $\Phi_{P_0}=\Phi_0$ the set of all simple roots of $G$, $\hat{\Phi}_P=\{\hat{\alpha}_i:\alpha_i\in \Sigma_P,<\hat{\alpha}_i,\alpha_j>=\delta_{ij}\}$, $\Phi_P$ the projection of the simple roots of $P$ onto $\mathfrak{a}_P$.

Fix a parabolic subgroup $P$, then take the Iwasawa decomposition \[G=PK.\]There exists a constant $c_P$, such that for all $f\in C_c^\infty(G_\mathbb{A})$, 
\begin{eqnarray*}
    \int_{G_\mathbb{A}}f(x)dx=c_P\int_{K}\int_{P_\mathbb{A}}f(kp)d_lpdk=c_P\int_K\int_{P_\mathbb{A}}f(kp)\delta_P(p)d_rpdk. 
\end{eqnarray*}
For any $\gamma\in G_\mathbb{Q}$ and $H$ is a connected subgroup of $G$. We write $H^+(\gamma)$ the center of $\gamma$ in $H$. We write $H(\gamma)$ the identity component of $H^+(\gamma)$. It is a normal subgroup of finite index in $H^+(\gamma)$ by the properties of reductive group, we denote $n_{\gamma, H}$ for the finite index.

If $H$ is reductive and $\gamma$ is semisimple in $H_\mathbb{Q}$. $H(\gamma)$ \cite{B1} is reductive.

For any function $\phi\in C_c^\infty(Z_\infty^+G_\mathbb{Q}\backslash G_\mathbb{A})$, 
\begin{align*}
    \rm{R}(f)\phi(x)&=\int_{Z_\infty^+\backslash G_\mathbb{A}}f(y)\phi(xy)dy\\
    &=\int_{Z_\infty^+\backslash G_\mathbb{A}}f(x^{-1}y)\phi(y)dy\\
    &=\int_{Z_\infty^+G_\mathbb{Q}\backslash G_\mathbb{A}}\sum_{\gamma\in G_\mathbb{Q}}f(x^{-1}\gamma y)\phi(y)dy. 
\end{align*}
Then the kernel of $\rm{R}(f)$ is \[K(x, y)=\sum_{\gamma\in G_\mathbb{Q}}f(x^{-1}\gamma y). \]

And \[\rm{Tr(R(f))}=\int_{Z_\infty^+G_\mathbb{Q}\backslash G_\mathbb{A}}K(x,x)dx.\]

\subsection{Associated class}
For two standard parabolic subgroups $P_1$, $P_2$, write \[\mathfrak{a}_i=\mathfrak{a}_{P_i},\quad i=1,2.\] Let $\Omega(\mathfrak{a}_1, \mathfrak{a}_2)$ denote the set of distinct isomorphisms from $\mathfrak{a}_1$ to $\mathfrak{a}_2$ by restricting the elements of Weyl group $(G,A_0)$. We say $P_1$ and $P_2$ are associated if $\Omega(\mathfrak{a}_1, \mathfrak{a}_2)\neq\emptyset$. For a fixed parabolic subgroup $P$, we define an associated class which is associated to $P$ by
\[\mathfrak{P}=\{P_1:\Omega(\mathfrak{a},\mathfrak{a}_1)\neq \emptyset\}.\] 
For $G=\rm{GL}(4)$, we write the associated classes respectively:
\begin{eqnarray*}
\mathfrak{P}_G=\{G\},\, \mathfrak{P}_{31}=\{P_{31},P_{13}\}, \mathfrak{P}_{22}=\{P_{22}\},\, \mathfrak{P}_{211}=\{P_{211},P_{121},P_{112}\}, \,\mathfrak{P}_{1111}=\{P_{1111}\}.
\end{eqnarray*}
Where the subscript indicates the structure of the Levi component. Explicitly, we have

\begin{eqnarray}\label{equc1}
 P_{1111}=\begin{pmatrix}
     a_{11} & * &* &* \\
     &a_{22}&*&* \\
     &&a_{33}&* \\
     &&&a_{44}
\end{pmatrix} ,\quad P_{211}=\begin{pmatrix}
    A_{22}&*&* \\
    & a_1&* \\
    &&a_2
\end{pmatrix},\quad P_{31}=\begin{pmatrix}
    A_{33}&* \\
    &a 
\end{pmatrix}, \\\notag
P_{22}=\begin{pmatrix}
    A_{22}&* \\
    &A_{22}'
\end{pmatrix},\quad P_G=\begin{pmatrix}
    A_{44} \end{pmatrix},\quad A_{ii},A_{ii}'\in \rm{GL}(i, \mathbb{Q}).    
\end{eqnarray}

We also have
\begin{eqnarray*}
    &|\Omega(\mathfrak{a}_{G}, \mathfrak{a}_{G})|=|\Omega(\mathfrak{a}_{31}, \mathfrak{a}_{31})|=|\Omega(\mathfrak{a}_{31}, \mathfrak{a}_{13})|=1,\\
    &|\Omega(\mathfrak{a}_{22},\mathfrak{a}_{22})|=|\Omega(\mathfrak{a}_{211}, \mathfrak{a}_{211})|=|\Omega(\mathfrak{a}_{211}, \mathfrak{a}_{121})|=|\Omega(\mathfrak{a}_{211}, \mathfrak{a}_{112})|=2,\\
    &|\Omega(\mathfrak{a}_{1111}, \mathfrak{a}_{1111})|=|S_4|=24.
\end{eqnarray*}

\section{The orbit}
In this section, we give a way to classify the elements in $G_\mathbb{Q}$.

Recall that the kernel function is given by \[K(x, y)=\sum_{\gamma\in G_\mathbb{Q}}f(x^{-1}\gamma y), \]we generally consider the case that $y=x$. Thus we need to consider the conjugacy class of $G_\mathbb{Q}$.

We now classify the orbits by the eigenvalues of the elements in $G_\mathbb{Q}$.

The eigenpolynomial of any element in $\rm{GL}(4,\mathbb{Q})$ is a degree-four polynomial over $\mathbb{Q}$.

There are eleven orbits: \[\mathfrak{o}_{G}=\mathfrak{o}_4^0, \mathfrak{o}_{31}^0, \mathfrak{o}^0_{211}, \mathfrak{o}_{211}^{2}, \mathfrak{o}_{22}^0, \mathfrak{o}_{22}^{2}, \mathfrak{o}_{1111}^{0}, \mathfrak{o}_{1111}^{211}, \mathfrak{o}_{1111}^{22}, \mathfrak{o}_{1111}^{31}, \mathfrak{o}_{1111}^{4}, \]
where the subscript indicates the degrees of the polynomials that the quartic polynomial factors into over $\mathbb{Q}$. If the quartic polynomial can not be  reducible completely, the superscript denotes the multiplicity of a repeated root while it is zero if it has no repeat roots. Also, if the quartic polynomial is reducible completely, the superscript except $0$ indicates the number of the irreducible polynomials which are the same and the superscript $0$ means it has no repeat roots.

For example, the subscript of  $\mathfrak{o}_{211}^0$ or $\mathfrak{o}_{211}^2$ corresponds to an eigenpolynomial that factors as a product of one quadratic irreducible factor and two distinct linear irreducible factors over $\mathbb{Q}$. The superscript $0$ or $2$ indicates the multiplicity of eigenvalues. Two typical forms of such eigenpolynomial are  \[(a_1x^2+b_1x+c_1)(d_1x+e_1)(f_1x+g_1),\] and \[(a_2x^2+b_2x+c_2)(d_2x+e_2)^2. \]

We observe that the classification aligns with the standard parabolic subgroups of $G_\mathbb{Q}$. The subscript of each orbit corresponds to the type of a standard parabolic subgroup in (\ref{equc1}).

For example, $\mathfrak{o}_{211}$ corresponds to the parabolic subgroup $P_{211}$. That is, the elements in $G_\mathbb{Q}$ within the orbit $\mathfrak{o}_{211}$ are $G_\mathbb{Q}$-conjugate to elements in $P_{211}$.

Let $\mathfrak{o}$ be an orbit. Choose a parabolic subgroup $P$ and a semisimple element $\gamma\in M_\mathbb{Q}\cap \mathfrak{o}$ such that no $M_\mathbb{Q}$-conjugate of $\gamma$ lies in $P_{1,\mathbb{Q}}$, where $P_1$ is a parabolic subgroup properly contained in $P$. In this case, we say that $\gamma$ is $M$-elliptic.

Let $P\cap\mathfrak{o}=P^{\mathfrak{o}}$, and let $P_\mathfrak{o}$ denote the minimal parabolic subgroup: 
\[P_{i_1i_2. . . i_n},\quad i_1\ge i_2\ge. . . \ge i_n, \]
where elements in $\mathfrak{o}$ could lie.

For fixed orbit $\mathfrak{o}$ and parabolic subgroup $P_\mathfrak{o}$, the superscript $0$ of $\mathfrak{o}$ corresponds to the condition that the centralizer satisfies: \[G(\gamma)=M_\mathfrak{o}(\gamma),\quad \gamma\in M^{\mathfrak{o}}_\mathfrak{o}.\] Equivalently, this is the condition that the unipotent part is trivial: $N_\mathfrak{o}(\gamma)=\{e\}$.

Call the orbit satisfies $G(\gamma)=M_\mathfrak{o}(\gamma), \,\gamma\in M_\mathfrak{o}^{\mathfrak{o}}$ unramified, others ramified. If one orbit is unramified, the elements in it are totally semisimple.

If $H$ is a subgroup of $G$, then $H(\gamma)\subset H(\gamma_s)$, where \[\gamma=\gamma_s\cdot\gamma_u\] is the Jordan decomposition.

As a result, we write\[K_{\mathfrak{o}}(x, x)=\sum_{\gamma\in \mathfrak{o}}{f(x^{-1}\gamma x)}.\]Then \[K(x, x)=\sum_{\mathfrak{o}}{K_\mathfrak{o}}(x, x). \]

\section{The Eisenstein series}\label{section4}
In this section, we recall the theory of Eisenstein series. The spectral theory of Eisenstein series was begun by Selberg and completed by Langlands \cite{L1}. We will mainly state the key results without detailed proofs.

For any parabolic subgroup $P$ of $G$, denote $\mathfrak{a}=\mathfrak{a}_P=\mathfrak{a}_M$, and \[\mathfrak{a}^+=\{H\in\mathfrak{a}:\, <\alpha,H>>0,\,\alpha\in\Phi_P\}.\]We call $\mathfrak{a}^+$ the chamber of $P$ in $\mathfrak{a}$.

Fix $P$, let $\mathscr{H}_P^0$ denote the space of functions \[\Phi:A_\infty^+\cdot N_\mathbb{A}\cdot M_\mathbb{Q}\backslash G_\mathbb{A}\rightarrow \mathbb{C}\] satisfing
\begin{enumerate}
    \item for any $ x\in G_\mathbb{A}$, the function $m\rightarrow \Phi(mx), m\in M_\mathbb{A}$, is $\mathscr{Z}_{M_\mathbb{R}}$-finite, where $\mathscr{Z}_{M_\mathbb{R}}$ is the center of the universal enveloping algebra of $\mathfrak{m}_{\mathbb{C}}$. 
    \item the function $\Phi_k:x\rightarrow \Phi(xk),\, x\in G_\mathbb{A},\, k\in K$ is $K$-finite. 
    \item \[||\Phi||^2=\int_K\int_{A_\infty^+\cdot M_\mathbb{Q}\backslash M_\mathbb{A}}|\Phi(mk)|^2dm\ dk\le\infty. \]
\end{enumerate}
We complete the space $\mathscr{H}_P^0$ and denote it by $\mathscr{H}_P$. Define a representation on $G_\mathbb{A}$ by \[\pi_P(\lambda, y)\Phi(x)=\Phi(xy)\rm{exp}(<\lambda+\rho_P, H_P(xy)>)\rm{exp}(-<\lambda+\rho_P, H_P(x)>), \]
where $\lambda\in\mathfrak{a}_\mathbb{C}=\mathfrak{a}\bigotimes \mathbb{C}$, $\Phi\in \mathscr{H}_P, \, x, y\in G_\mathbb{A}$. 

It is induced from a representation of $P_\mathbb{A}$, which in turn is the pull-back of a certain representation  $\pi_M^M(\lambda)$ of $M_\mathbb{A}$. $I_M^M(0)$ is the subrepresentation of the regular representation of $M_\mathbb{A}$ on $L^2(A_\infty^+M_\mathbb{Q}\backslash M_\mathbb{A})$ which decomposes discretely. Write \[I_M^M(0)=\oplus_l\sigma^l,\]where $\sigma^l=\oplus_v\sigma_v^l$ is an irreducible representation of $M_\mathbb{A}$. Define \[\sigma_{v,\lambda}=\sigma_v(m)\rm{exp}(<\lambda,H_M(m)>),\quad \lambda\in\mathfrak{a}_\mathbb{C},\, m\in M_{\mathbb{Q}_v},\] if $v$ is prime and $\sigma_v$ is an irreducible unitary representation of $M_{\mathbb{Q}_v}$.

If $\sigma_{v, \lambda}$ is lifted to $P_{\mathbb{Q}_v}$ and then induced to $G_{\mathbb{Q}_v}$, we obtain a representation $\pi_P(\sigma_{v, \lambda})$ of $G_{\mathbb{Q}_v}$, acting on a Hilbert space $\mathscr{H}_P(\sigma_v)$. With this notation, we write
\[
\pi_P(\lambda) = \oplus_l \otimes_v \pi_P(\sigma^l_{v, \lambda}).
\]
$\rho_P$ is a vector in $\mathfrak{a}$, satisfing \[|\rm{det}(\rm{Ad}\ m)_{\mathscr{n}_\mathbb{A}}|=\rm{exp}(<2\rho_P, H_P(m)>), \quad m\in M_\mathbb{A}. \]

The following identities hold: \[\pi_P(\lambda, y)^*=\pi_P(-\overline{\lambda}, y^{-1}), \quad y\in G_\mathbb{A}, \]
and 
\[\pi_P(\lambda, f)^*=\pi_P(-\overline{\lambda}, f^*),\quad f\in C_c^\infty(G_\mathbb{A}), \]
where $f^*(y)=f(y^{-1}). $ In particular, $\pi_P(\lambda)$ is unitary if $\lambda$ is purely imaginary.

Let $P$, $P_1$ be associated parabolic subgroups, with $s\in\Omega(\mathfrak{a}, \mathfrak{a}_1)$. Fix $w_s$ a representative of $s$ in the intersection of $K\cap G_\mathbb{Q}$ with $N_G(A_0)$, the normalizer of $A_{0}$ in $G$. For $\Phi\in\mathscr{H}_P^0, \,\lambda\in\mathfrak{a}_\mathbb{C},$ and $x\in G_\mathbb{A}$, define the intertwining operator $M_P(s, \lambda)\Phi(x)$ to be
\[\int_{N_{1, \mathbb{A}}\cap w_sN_{2, \mathbb{A}}w_s^{-1}\backslash N_{1, \mathbb{A}}}\Phi(w_s^{-1}nx)\rm{exp}(<\lambda+\rho_{P_1}, H_P(w_s^{-1}nx)>)dn\ \rm{exp}(-<s\lambda+\rho_{P_2}, H_{P_2}(x)>). \]
This integral is absolutely convergant if $<\alpha, \rm{Re}\ \lambda-\rho_{P_1}>$ is positive, for each $\alpha\in\Sigma_{P_1}$, such that $s\alpha\in -\Sigma_{P_2}$(in \cite{A5}).  
It defines a linear operator from $\mathscr{H}_{P_1}^0$ to $\mathscr{H}_{P_2}^0$.

This operator satisfys:
\[M_P(s, \lambda)^*=M_P(s^{-1}, -s\overline{\lambda}). \]If $f\in C_c^\infty(G_\mathbb{A})^K$, the $K$-conjugate invariant functions in $C_c^\infty(G_\mathbb{A})$, we have \[M_P(s, \lambda)\pi_P(\lambda, f)=\pi_P(s\lambda, f)M_P(s, \lambda). \]

If $\Phi\in\mathscr{H}_P^0, \,x\in G_\mathbb{A}, \,\lambda\in\mathfrak{a}_\mathbb{C}$, then the Eisenstein series is \[E_P(\Phi, \lambda, x)=\sum_{\delta\in P_\mathbb{Q}\backslash G_\mathbb{Q}}\Phi(\delta x)\rm{exp}(<\lambda+\rho_P, H_P(\delta x)>).\]
This series is absolutely convergent if $\rm{Re}\ \lambda\in \rho_P+\mathfrak{a}^+$.

We now state Langlands' fundamental theorem on Eisenstein series: 
\begin{theorem}[Langlands\cite{L1}]
    \begin{enumerate}
        \item Suppose $\Phi\in\mathscr{H}_P^0$, $E_P(\Phi, \lambda, x)$ and $M_P(s, \lambda)\Phi$ can be analytically continued as meromorphic functions to $\mathfrak{a}_\mathbb{C}$. On $i\mathfrak{a}$, $E_p(\Phi, \lambda, x)$ is regular, and $M_P(s, \lambda)$ is unitary. For $f\in C_c^\infty(G_\mathbb{A})^K$ and $t\in \Omega(\mathfrak{a}_1, \mathfrak{a}_2)$, the following functional equation hold:
        \begin{enumerate}
            \item $E_P(\pi_P(\lambda, f)\Phi, \lambda, x)=\int_{G_\mathbb{A}}f(y)E_P(\Phi, \lambda, xy)dy, $
            \item $E_P(M_P(s, \lambda)\Phi, s\lambda, x)=E_P(\Phi, \lambda, x)$, 
            \item $M_P(ts, \lambda)\Phi=M_P(t, s\lambda)M_P(s, \lambda)\Phi. $
        \end{enumerate}
        \item Let $\mathfrak{P}$ be an associated class of parabolic subgroups. Let $\hat{L}_{\mathfrak{P}}$ be the set of collections \[F=\{F_P:P\in\mathfrak{P}\}\] of measurable functions $F_P:i\mathfrak{a}\rightarrow \mathscr{H}_P$ such that 
        \begin{enumerate}
            \item If $s\in \Omega(\mathfrak{a}, \mathfrak{a}_1)$, \[F_{P_1}(s\lambda)=M_P(s, \lambda)F_P(\lambda), \]
            \item \[||F||^2=\sum_{P\in\mathfrak{P}}n(A)^{-1}(\frac{1}{2\pi i})^{\rm{dim}\ A}\int_{i\mathfrak{a}}||F_P(\lambda)||^2\le\infty, \]where $n(A)$ is the number of chambers in $\mathfrak{a}$. Then the map which sends $F$ to the function \[\sum_{P\in\mathfrak{P}}n(A)^{-1}(\frac{1}{2\pi i})^{\rm{dim}\ A}\int_{i\mathfrak{a}}E_P(F_P(\lambda), \lambda, x)d\lambda, \]
            defined for $F$ in a dense subspace of $\hat{L}_{\mathfrak{P}}$, extends to a unitary map from $\hat{L}_{\mathfrak{P}}$ onto a closed $G_\mathbb{A}-$invariant subspace $L^2_\mathfrak{P}(G_\mathbb{Q}\backslash G_\mathbb{A})$ of $L^2(G_\mathbb{Q}\backslash G_\mathbb{A})$. Moreover, we have an orthogonal decomposition \[L^2(G_\mathbb{Q}\backslash G_\mathbb{A})=\oplus_{\mathfrak{P}}L^2_{\mathfrak{P}}(G_\mathbb{Q}\backslash G_\mathbb{A}). \]
        \end{enumerate}
    \end{enumerate}
\end{theorem}
If parabolic subgroup $P\neq G$, let $\mathscr{H}_{P, \rm{cusp}}$ denote the Hilbert space of the measurable functions $\Phi$ on $A_\infty^+N_\mathbb{A}M_\mathbb{Q}\backslash G_\mathbb{A}$ satisfing
\begin{enumerate}
    \item $||\Phi||^2=\int_K\int_{M_\mathbb{Q}A_\infty^+\backslash M_\mathbb{A}}|\Phi(mk)|^2dm\ dk<\infty$, 
    \item for the parabolic subgroup $Q$, such that $G\supsetneqq Q\supsetneqq P$, and for $x\in G_\mathbb{A}$, the integral \[\int_{N_{Q, \mathbb{Q}}\backslash N_{Q, \mathbb{A}}}\Phi(nx)dn=0. \]
\end{enumerate}
And if $P=G$, let $\mathscr{H}_{G, \rm{cusp}}$ denote the Hilbert space of the measurable functions $\Phi$ on $Z_\infty^+G_\mathbb{Q}\backslash G_\mathbb{A}$ satisfing
\begin{enumerate}
    \item $||\Phi||^2=\int_{Z_\infty^+G_\mathbb{Q}\backslash G_\mathbb{A}}|\Phi(x)|^2dx<\infty$, 
    \item for the parabolic subgroup $Q$, such that $Q\supsetneq G$, and for $x\in G_\mathbb{A}$, the integral \[\int_{N_{Q, \mathbb{Q}}\backslash N_{Q, \mathbb{A}}}\Phi(nx)dn=0. \]
\end{enumerate}
This space is invariant under right $G_\mathbb{A}$-action.

\begin{lemma}\cite{GGPS}\label{hilbert}
         If $f\in C_c^N(G_\mathbb{A})$,  for $N$ large enough, then the map $\Phi\rightarrow \Phi*f,\, \Phi\in\mathscr{H}_{G,\rm{cusp}}$ is a Hilbert-Schmidt operator on $\mathscr{H}_{G,\rm{cusp}}$. 
\end{lemma}
\begin{corollary}
    $\mathscr{H}_{G, \rm{cusp}}$ decomposes into a direct sum of irreducble representations of $G_\mathbb{A}$, each occuring with finite multiplicity. 
\end{corollary}
This Corollary can be followed by Lemma \ref{hilbert} combined with the spectral theorem for compact operators.

The space $\mathscr{H}_{G, \rm{cusp}}$ is defined as  the space of cusp forms on $G_\mathbb{A}$. By the above corollary, any function in $\mathscr{H}_{G, \rm{cusp}}$ can be approximated by a limit of functions in $\mathscr{H}_P^0$. Consequently, $\mathscr{H}_{P, \rm{cusp}}$ is a subspace of $\mathscr{H}_P$.

Denote $\mathscr{T}(G)$ by the set of all triplets $\chi=(\mathfrak{P}, \mathfrak{V}, W)$, where $W$ is an irreducble representation of $K$, $\mathfrak{P}$ is an associated class of parabolic subgroups, $\mathfrak{V}$ is a family of subspaces
\[\{V_P\subset \mathscr{H}_{M, \rm{cusp}}^M, \rm{the\  space\  of\  cusp\  forms\  on}\  M_\mathbb{A}\}_{P\in\mathfrak{P}}, \] satisfing
\begin{enumerate}
    \item for $P\in\mathfrak{P}$, $V_P$ is the eignspace of $\mathscr{H}_{M, \rm{cusp}}^M$ associated to a complex homomorphism of $\mathscr{Z}_{M_\mathbb{R}}$, 
    \item for $P_1, P_2\in\mathfrak{P}$, $s\in\Omega(\mathfrak{a}_1, \mathfrak{a}_2)$, the space $V_{P_2}$ can be obtained by conjugating functions in $V_{P_1}$ by $w_s$. 
\end{enumerate}
For $P\in\mathfrak{P}$, the space $\mathscr{H}_{P, \chi}$ consists of the functions $\Phi\in \mathscr{H}^0_{P, \rm{cusp}}$ satisfing that for every $x\in G_\mathbb{A}$, 
\begin{enumerate}
    \item the function takes $k$ to $\Phi(xk),\, k\in K$, is a matrix coefficient of $W$, 
    \item the function takes $m$ to $\Phi(mx),\, m\in M_\mathbb{A}$, is contained in $V_P$. 
\end{enumerate}
The dimensional of the space $\mathscr{H}_{P, \chi}$ is finite and it is invariant under $\pi_P(\lambda, f), $ for any $f\in C_c^\infty(G_\mathbb{A})^K$.

We have the decomposition
\[\mathscr{H}_{P, \rm{cusp}}=\oplus_{\chi}\mathscr{H}_{P, \chi}, \]where $\chi=(\mathfrak{P}, \mathfrak{V}, W),\, P\in\mathfrak{P}$.

For any $\chi$, and any $ P\in\mathfrak{P}$, suppose that the analytic function \[\lambda\rightarrow \Phi(\lambda)=\Phi(\lambda,x),\quad \lambda\in\mathfrak{a}_\mathbb{C},\, x\in A_\infty^+N_\mathbb{A}M_\mathbb{Q}\backslash G_\mathbb{A}, \]is of Paley-Wiener type, it maps $\mathfrak{a}_\mathbb{C}$ to $\mathscr{H}_{P, \chi}$. Then
\[\phi(x)=(\frac{1}{2\pi i})^{\rm{dim}\ A}\int_{\rm{Re}\ \lambda=\lambda_0}\rm{exp}(<\lambda+\rho_P, H_P(x)>)\Phi(\lambda, x)d\lambda,\quad x\in A_\infty^+N_\mathbb{A}M_\mathbb{Q}\backslash G_\mathbb{A}, \] is a function on $N_\mathbb{A}M_\mathbb{Q}\backslash G_\mathbb{A}$, which is independent of $\lambda_0\in\mathfrak{a}$.

The series \[\hat{\phi}(x)=\sum_{\delta\in P_\mathbb{Q}\backslash G_\mathbb{Q}}\phi(\delta x)\] converges absolutely and it belongs to $L^2(G_\mathbb{Q}\backslash G_\mathbb{A})$. Denote the closed space generated by such $\hat{\phi}$ by $L^2_\chi(G_\mathbb{Q}\backslash G_\mathbb{A})$. We then have the orthogonal decomposition \[L^2(G_\mathbb{Q}\backslash G_\mathbb{A})=\oplus_{\chi}L^2_\chi(G_\mathbb{Q}\backslash G_\mathbb{A}). \]
The constant term of  Eisenstein series is denoted by $E_P^{c_1}(\Phi, \lambda, x)$, where $c_1$ indicates that it is the constant term associated to another parabolic subgroup $P_1$. $ E_P^{c_1}(\Phi, \lambda, x)$ is defined by
\begin{eqnarray*}
    &\int_{N_{1, \mathbb{Q}}\backslash N_{1, \mathbb{A}}}E_P(\Phi, \lambda, nx)dn, 
\end{eqnarray*}
which equals
\begin{eqnarray*}
    &\sum_{s\in\Omega(\mathfrak{a}, \mathfrak{a}_1)}(M_P(s, \lambda)\Phi)(x)\rm{exp}(<s\lambda+\rho_{P_1}, H_{P_1}(x)>). 
\end{eqnarray*}
For $\lambda_0\in \rho_P+\mathfrak{a}^+$, we have \[\hat{\phi}(x)=
(\frac{1}{2\pi i})^{\rm{dim}\ A}\int_{\rm{Re}\ \lambda=\lambda_0}E_P(\Phi(\lambda), \lambda, x)d\lambda. \]
For another $\Phi_1(\lambda_1, x)$ associated to parabolic subgroup  $P_1\in\mathfrak{P}$, the inner product \[\int_{G_\mathbb{Q}\backslash G_\mathbb{A}}\hat{\phi}(x)\overline{\hat{\phi}_1(x)}dx\] equals
\begin{eqnarray*}
    (\frac{1}{2\pi i})^{\rm{dim}\ A}\int_{\lambda_0+i\mathfrak{a}}\sum_{s\in\Omega(\mathfrak{a}, \mathfrak{a}_1)}(M_P(s, \lambda)\Phi(\lambda), \Phi_1(-s\overline{\lambda}))d\lambda,\quad \lambda_1\in\rho_P+\mathfrak{a}^+. 
\end{eqnarray*}
Fix $\chi=(\mathfrak{P}_\chi, \mathfrak{V}, W)$. For $\Phi\in\mathscr{H}_{P, \chi}, \,P\in\mathfrak{P}_{\chi}$, one shows that the singularities of the two functions $E_P(\Phi, \lambda, x)$ and $M_P(s, \lambda)\Phi$ are hyperplanes of the form \[\mathcal{\tau}=\{\lambda\in\mathfrak{a}_\mathbb{C}:<\alpha, \lambda>=\mu, \mu\in\mathbb{C}, \alpha\in\Sigma_{P}\},\] and only finitely many of them meet $\mathfrak{a}^++i\mathfrak{a}$, which equals to the set \[\{\lambda\in\mathfrak{a}_{\mathbb{C}}:<\alpha, \rm{Re}\ \lambda>>0, \alpha\in\Phi_P\}.\]

Write the space generated by function $\phi(x)$ as $L^2_{\mathfrak{P}_\chi, \chi}(G_\mathbb{Q}\backslash G_\mathbb{A})$, it is closed in $L^2_\chi(G_\mathbb{Q}\backslash G_\mathbb{A})$.

If $\phi_1(x)$ comes from $\Phi_1(\lambda_1)$, $P_1\in\mathfrak{P}_\chi$, then the inner product \[\int_{G_\mathbb{Q}\backslash G_\mathbb{A}}\hat{\phi}(x)\overline{\hat{\phi}_1(x)}dx\]equals
\[\sum_{P_2\in\mathfrak{P}_\chi}n(A)^{-1}(\frac{1}{2\pi i})^{\rm{dim}\ A}\int_{i\mathfrak{a}_2}(F_{P_2}(\lambda), F_{1, P_2}(\lambda))d\lambda, \]where $F_{1, P_2}(\lambda)=\sum_{r\in\Omega(\mathfrak{a}_2, \mathfrak{a}_1)}M_P(r, \lambda)^{-1}\Phi_1(r\lambda)$, and $F_{P_2}$ similarly.

Define $\hat{L}_{\mathfrak{P}_\chi, \chi}$ as the space consisting of the functions $\{F_{P_2}:P_2\in\mathfrak{P}_\chi, F_{P_2}$ takes values in $\mathscr{H}_{P_2, \chi}\}$. In fact, it is an isometric isomorphic from a dense subspace of $\hat{L}_{\mathfrak{P}_\chi, \chi}$ to a dense subspace of $L^2_{\mathfrak{P}_\chi, \chi}(G_\mathbb{Q}\backslash G_\mathbb{A})$.

Write $Q$ as the projection of $L^2_\chi(G_\mathbb{Q}\backslash G_\mathbb{A})$ onto the orthogonal complement of $L^2_{\mathfrak{P}_\chi, \chi}(G_\mathbb{Q}\backslash G_\mathbb{A})$, we denote this space by $L^2_{\chi, \rm{res}}(G_\mathbb{Q}\backslash G_\mathbb{A})$. Then for the functions $\hat{\phi}(x),\, \hat{\phi}_1(x)$ corresponding to $\Phi(\lambda), \Phi_1(\lambda_1)$, the inner product $(Q\hat{\phi}, \hat{\phi}_1)$, is given by the difference 
\begin{eqnarray*}
    (\frac{1}{2\pi i})^{\rm{dim}\ A}(\int_{\lambda_0+i\mathfrak{a}}\sum_{s\in\Omega(\mathfrak{a}, \mathfrak{a}_1)}(M_P(s, \lambda)\Phi(\lambda), \Phi_1(-s\overline{\lambda}))d\lambda\\
    -\int_{i\mathfrak{a}}\sum_{s\in\Omega(\mathfrak{a}, \mathfrak{a}_1)}(M_P(s, \lambda)\Phi(\lambda), \Phi_1(-s\overline{\lambda})d\lambda)). 
\end{eqnarray*}
Consider choosing a path in $\mathfrak{a}^+$ from $\lambda_1$ to $0$ whose intersection with any singular hyperplane $\tau$ of $\{M_P(s, \lambda):s\in\Omega(\mathfrak{a}, \mathfrak{a}_1)\}$ is at most one point, denote the set by $Z(\mathscr{\tau})$. We can write $\tau$ into the sum $X(\tau)+\check{\tau}_\mathbb{C}$, where $\check{\tau}$ is a subspace of $\mathfrak{a}$ of codimension one, $X(\tau)$ is a vector in $\mathfrak{a}$ orthogonal to $\check{\tau}$, and $Z(\tau)\in X(\tau)+\check{\tau}$. Then, by the residue theorem, the inner product $(Q\hat{\phi}, \hat{\phi}_1)$ becomes \[(\frac{1}{2\pi i})^{(\rm{dim}\ A)-1}\sum_\tau\int_{Z(\tau)+i\check{\tau}}\sum_{s\in\Omega(\mathfrak{a}, \mathfrak{a}_1)}\rm{Res}_\tau(M_P(s, \lambda)\Phi(\lambda), \Phi_1(-s\overline{\lambda}))d\lambda. \]
We have the following decompositions
\[L^2_\chi(G_\mathbb{Q}\backslash G_\mathbb{A})=\oplus_{\mathfrak{P}}L^2_{\mathfrak{P}, \chi}(G_\mathbb{Q}\backslash G_\mathbb{A}), \]
\[L^2_{\mathfrak{P}}(G_\mathbb{Q}\backslash G_\mathbb{A})=\oplus_\chi L^2_{\mathfrak{P}, \chi}(G_\mathbb{Q}\backslash G_\mathbb{A}), \]
and 
\[L^2(G_\mathbb{Q}\backslash G_\mathbb{A})=\oplus_{\mathfrak{P}, \chi}L^2_{\mathfrak{P}, \chi}(G_\mathbb{Q}\backslash G_\mathbb{A}). \]
We now replace $G_\mathbb{Q}\backslash G_\mathbb{A}$ by $Z_\infty^+G_\mathbb{Q}\backslash G_\mathbb{A}$. 
For any fixed parabolic subgroup $P$ and function $f\in C_c^\infty(Z_\infty^+\backslash G_\mathbb{A}),\, \lambda\in \mathfrak{a}_\mathbb{C}$, we define the function $P_P(\lambda, f, x, y)$ by the product of 
\begin{eqnarray*}
\rm{exp}(<\lambda+\rho_P, H_P(y)>)\rm{exp}(<-\lambda-\rho_P, H_P(x)>)
\end{eqnarray*}
and  
\begin{eqnarray*}
    \sum_{\gamma\in M_{\mathbb{Q}}}\int_{N_\mathbb{A}}\int_{\mathfrak{a}_G\backslash\mathfrak{a}}f(x^{-1}nh_a\gamma y)\rm{exp}(<-\lambda-\rho_P, a>)da\ dn. 
\end{eqnarray*}
This function is continuous on $N_\mathbb{A}M_\mathbb{Q}A_\infty^+\backslash G_\mathbb{A}\times N_\mathbb{A}M_\mathbb{Q}A_\infty^+\backslash G_\mathbb{A},$ and is a Schwartz function of $\lambda\in\mathfrak{a}$. 
\begin{lemma}\label{lemma4.4}
    Given $f\in C_c^\infty(Z_\infty^+\backslash G_\mathbb{A}), \lambda\in\mathfrak{a}_\mathbb{C}, \phi\in \mathscr{H}_P, \pi_P(\lambda, f)\phi(x)$ equals \[c_P\int_K\int_{A_\infty^+M_\mathbb{Q}\backslash M_\mathbb{A}}P_P(\lambda, f, x, mk)dm\ dk. \]
\end{lemma}
\begin{proof}
    \begin{align}
        \pi_P(\lambda, f)\phi(x)&=\int_{Z_\infty^+\backslash G_\mathbb{A}}f(y)\pi_P(\lambda, y)\phi(x)dy\notag\\
        &=\int_{Z_\infty^+\backslash G_\mathbb{A}}f(y)\phi(xy)\rm{exp}(<\lambda+\rho_P, H_P(xy)>)\rm{exp}(-<\lambda+\rho_P, H_P(x)>)dy\notag\\
        &=\int_{Z_\infty^+\backslash G_\mathbb{A}}f(x^{-1}y)\phi(y)\rm{exp}(<\lambda+\rho_P, H_P(y)>)\rm{exp}(-<\lambda+\rho_P, H_P(x)>)dy.\label{234}
    \end{align}
For $m\in A_\infty^+M_\mathbb{Q}\backslash M_\mathbb{A},\,k\in K$, define the function $O(m,k)$ equals 
\[\sum_{\gamma\in M_\mathbb{Q}}\int_{\mathfrak{a}_G\backslash \mathfrak{a}}\int_{N_\mathbb{A}}f(x^{-1}nh_a\gamma mk)\cdot \rm{exp}(<\lambda+\rho_P, a>)dn\ da\cdot\rm{exp}(-<\lambda+\rho_P, H_P(x)>).\]
By the Iwasawa decomposition, the term (\ref{234}) equals
    \begin{eqnarray*}
        &c_P\int_K\int_{A_\infty^+M_\mathbb{Q}\backslash M_\mathbb{A}}O(m,k)\phi(mk)dm\ dk,
    \end{eqnarray*}
    since in this case, \[\rm{exp}(<\lambda+\rho_P,H_P(mk)>)=1,\] (\ref{234}) is
    \begin{eqnarray*}
        &c_P\int_K\int_{A_\infty^+M_\mathbb{Q}\backslash M_\mathbb{A}}P_P(\lambda, f, x, mk)\phi(mk)dm\ dk. 
    \end{eqnarray*}
\end{proof}
We denote \[L_0^2(Z_\infty^+G_\mathbb{Q}\backslash G_\mathbb{A})=L^2_{\rm{cusp}}(Z_\infty^+G_\mathbb{Q}\backslash G_\mathbb{A})\oplus\oplus_{\chi}L^2_{\chi,\rm{res}}(Z_\infty^+G_\mathbb{Q}\backslash G_\mathbb{A}),\]
and \[L^2(Z_\infty^+G_\mathbb{Q}\backslash G_\mathbb{A})=L_0^2(Z_\infty^+G_\mathbb{Q}\backslash G_\mathbb{A})\oplus L_1^2(Z_\infty^+G_\mathbb{Q}\backslash G_\mathbb{A}).\]

Denote the restriction of $\rm{R}(f)$ on  the former space $L_0^2(Z_\infty^+G_\mathbb{Q}\backslash G_\mathbb{A})$ by $\rm{R}_0(f)$, and on the latter space $L_1^2(Z_\infty^+G_\mathbb{Q}\backslash G_\mathbb{A})$ by $\rm{R}_1(f)$.

Thus, \[\rm{R}(f)=\rm{R}_0(f)\oplus\rm{R}_1(f). \]

\section{The operator $\rm{R}_0(f)$}\label{sc5}
In this section, we shall show that the operator $\rm{R}_0(f)$ is of trace class.

By the decomposition in the last section, we have a similar decomposition \[L^2(Z_\infty^+G_\mathbb{Q}\backslash G_\mathbb{A})=\oplus_{\mathfrak{P}, \chi}L^2_{\mathfrak{P}, \chi}(Z_\infty^+G_\mathbb{Q}\backslash G_\mathbb{A}). \]
Thanks to Duflo and Labesse(\cite{D1}), we have the following Lemma
\begin{lemma}
    For every $N\ge0$, suppose $f$ belongs to $C_c^\infty(G_\mathbb{A})$, then $f$ equals a finite sum of functions of the form \[f^1*f^2, \]
    where $f^1, f^2\in C_c^N(Z_\infty^+\backslash G_\mathbb{A})^K$, the superscript $K$ indicates the function is $K$-finite. 
\end{lemma}
By this Lemma, we can assume that \[f=f^1*f^2.\]
For any parabolic subgroup $P$ and $\chi$, write $\mathscr{B}_{P, \chi}$ for the set of indices $\alpha$ corrssponding to an element of an orthonormal basis of the finite-dimensional space $\mathscr{H}_{P, \chi}$. 
Then, define \[I_P=\cup_{\chi} \mathscr{B}_{P, \chi}. \]
Fix an orthonormal basis \[\{\Phi_\beta:\beta\in I_P\}.\] Denote \[\Phi_\alpha=\pi_P(\lambda, f)\Phi_\beta.\]

Recall that $\mathscr{T}(G)$, the collection of $\chi$, can be considered as  a set of unitary equivalence classes of irreducible representations of $G_\mathbb{A}$(see \cite{A4}). For any representation $\sigma$ of $M_\mathbb{A}$, define the action of $s\in\Omega(\mathfrak{a},\mathfrak{a}')$ on another Levi subgroup $M'_\mathbb{A}$ by \[(s\sigma')(m')=\sigma(w_s^{-1}m'w_s),\quad m'\in M'_\mathbb{A}.\] We call a class $\chi$ is unramified if for every pair $(P,  \pi_P)$ in $\chi$,  the stabilizer of $\pi_P$ in $\Omega(\mathfrak{a}, \mathfrak{a})$ is the identity, otherwise $\chi$ is ramified.

For unramified $\chi$, suppose $\chi=(\mathfrak{P},\mathfrak{V},W),\,P\in\mathfrak{P},\,\Phi,\,\Phi'\in\mathscr{H}_{P,\chi},\,s,\,s'\in \Omega(\mathfrak{a},\mathfrak{a})$, then \[(M_P(s,\lambda)\Phi,M'_P(s',\lambda')\Phi')=0,\]unless $s=s'$(see\cite{A4}).

Recall that \[\pi_P(\lambda)=\oplus_l\otimes_v\pi_P(\sigma_{v,\lambda}^l),\]we can assume that 
\begin{eqnarray}\label{class}
    \pi_{P,\chi}(\lambda,f^1)=\pi_{P,\chi}(\lambda,f^2)=0,
\end{eqnarray}
for almost all ramified $\chi$.

The residual discrete spectrum associated to unramified $\chi$ is zero\cite{L1}.

For a parabolic subgroup $P$, denote the restriction of the operator $\rm{R}_1(f)$ on the space $L_{\mathfrak{P}}^2(Z_\infty^+G_\mathbb{Q}\backslash G_\mathbb{A})$ by $\rm{R}_{P,1}(f)$.
\begin{lemma}
Given $P\in\mathfrak{P}$. $\rm{R}_{P, 1}(f)$ is an integral operator with kernel $K_{P}(x, y)$, which is
\begin{eqnarray*}
\sum_{\chi}n(A)^{-1}(\frac{1}{2\pi i})^{\rm{dim}\ Z\backslash A}\int_{i\mathfrak{a}_G\backslash i\mathfrak{a}}\sum_{\alpha, \beta\in \mathscr{B}_{P, \chi}}E_{P}(\Phi_\alpha, \lambda, x)\overline{E_{P}(\Phi_\beta, \lambda, y)}d\lambda. 
\end{eqnarray*}
\end{lemma}
\begin{proof}
    The definition of the kernel $K_{P}(x,y)$ follows from the spectral decomposition.

    We now only  need to prove the convergence of the integral in $K_{P}(x,y)$ and the sum over $\chi$ converged and show that they are locally bounded.

    Write $f=f_1*f_2$.
    Define $K_{P, \chi}(f, x, y)$ to be
    \begin{eqnarray*}
       \sum_{\beta\in \mathscr{B}_{P, \chi}}E_{P}(\pi_{P}(\lambda, f)\Phi_\beta, \lambda, x)\overline{E_{P}(\Phi_\beta, \lambda, x)} 
    \end{eqnarray*} 
    Applying the Cauchy-Schwartz inequality, the absolute value of the function $K_{P, \chi}(f, x, y)$ is bounded by \[K_{{P}, \chi}(f^1*(f^1)^*, x, x)^{\frac{1}{2}}\cdot K_{P, \chi}((f^2)^**f^2, y, y)^{\frac{1}{2}}. \]
    However, the operetor $\rm{R}_{{P},1}(f)$ is the restriction of the positive semidefinite operator $\rm{R}(f)$ to an invariant subspace. The integrand in the expression for $K_{P}(x,x)$ is non-negative, and the integral is bounded by $K(x, x)$. 

    By \cite{H1}, $K(x,x)$ is bounded.
\end{proof}
The proof also shows that the kernel $K_{P}(x, y)$ is continuous in $x, y$.
\begin{theorem}
    Given function $f\in C_c^\infty(Z_\infty^+\backslash G_\mathbb{A})$, the operator $\rm{R}_0(f)$ is of trace class. 
\end{theorem}
\begin{proof}
    The operator $\rm{R}_0(f)$ is the sum of $\rm{R}_{0, \rm{cusp}}(f)$ and $\rm{R}_{0, \rm{res}}(f)$, these two are the restriction of $\rm{R}_0(f)$ to the space of cusp forms and the space $\oplus_\chi L^2_{\chi, \rm{res}}(G_\mathbb{Q}\backslash G_\mathbb{A})$. Now, \[\rm{R}_{0, \rm{cusp}}(f)=\rm{R}_{0, \rm{cusp}}(f^1)\cdot \rm{R}_{0, \rm{cusp}}(f^2),\] Harish-Chandra (\cite{H1}) has proved that these two operators $\rm{R}_{0, \rm{cusp}}(f^1), \rm{R}_{0, \rm{cusp}}(f^2)$ are of Hilbert-Schmidt class. Then $\rm{R}_{0, \rm{cusp}}(f)$ is of trace class. 

    By (\ref{class}) and the fact that the residual discrete spectrum associated to unramified $\chi$ is zero\cite{L1},  both $\rm{R}_{0, \rm{res}}(f^1)$ and $\rm{R}_{0, \rm{res}}(f^2)$ are of finite rank. Therefore $\rm{R}_{0, \rm{res}}(f)$ is of trace class. 
\end{proof}
We now express the trace of $\rm{R}_{0}(f)$ as \[\int_{Z_\infty^+G_\mathbb{Q}\backslash G_\mathbb{A}}K_0(x, x)dx. \]
\section{The calculation of the kernel}\label{sc6}
In this section, we shall give a explicit way to calculate the kernel $K_{\mathfrak{o}}(x,y)$.

Consider the geometric side of the trace formula of $\rm{GL(4)}$. Recall that \[K(x,x)=\sum_{\gamma\in G_\mathbb{Q}}f(x^{-1}\gamma x). \]
Now fix a parabolic subgroup $P$ and $\gamma\in M_\mathbb{Q}$. Recall $G^+(\gamma), P^+(\gamma), M^+(\gamma), N^+(\gamma)$ are the centralizer of $\gamma$ in $G, P, M, N$ respectively.

\begin{lemma}
    For any $\gamma\in M_\mathbb{Q}$,  \[P^+(\gamma)=M^+(\gamma)N^+(\gamma). \]
\end{lemma}
\begin{proof}
    Since \[M^+(\gamma)\subset P^+(\gamma),\quad N^+(\gamma)\subset P^+(\gamma),\]
    we have \[M^+(\gamma)N^+(\gamma)\subset P^+(\gamma).\]
    Suppose $p\in P^+(\gamma)\subset P$, we can write  \[p=mn, \quad m\in M_\mathbb{Q}, \,n\in N_\mathbb{Q}. \]
    Then, \[p=\gamma^{-1}p\gamma=\gamma^{-1}m\gamma\cdot \gamma^{-1}n\gamma=mn. \]Since  $\gamma$ normalize $M$ and $N$, it follows that  
    \[m=\gamma^{-1}m\gamma, \quad n=\gamma^{-1}n\gamma. \]
    Therefore, \[m\in M^+(\gamma),\quad n\in N^+(\gamma).\]Then $p\in M^+(\gamma)N^+(\gamma),$ the lemma follows.
\end{proof}
Since $N^+(\gamma)$ is connected, the centralizer of $\gamma$ in $N$ is $N(\gamma)$. Write $\gamma=\gamma_s\gamma_u$ as the Jordan decomposition over $\mathbb{Q}$, where $\gamma_s$ is semisimple, $\gamma_u$ is unipotent.

The following Lemma is from \cite{A3}. For convenience, we present the proof.
\begin{lemma}\label{l1}
    Suppose that $P=NM$ is a parabolic subgroup, and $\gamma\in M_\mathbb{Q}$. Then for any $\phi\in C_c(N_\mathbb{A})$, 
    \begin{eqnarray}\label{Lemma6.2}
        \sum_{\delta\in N(\gamma_s)_\mathbb{Q}\backslash N_\mathbb{Q}}\sum_{\eta\in N(\gamma_s)_\mathbb{Q}}\phi(\gamma^{-1}\delta^{-1}\gamma\eta\delta)=\sum_{\eta\in N_\mathbb{Q}}\phi(\eta). 
    \end{eqnarray}
\end{lemma}

\begin{proof}
    If $\gamma$ is replaced by an $M$-conjugate element,  $\gamma = \mu^{-1}m\mu$,  with $\mu,m \in M$. Then the term $\gamma^{-1}\delta^{-1}\gamma\eta\delta$ becomes
    $\mu^{-1} m^{-1} \mu \delta^{-1} \mu^{-1} m \mu \eta \delta$,  
    which is
    \begin{eqnarray*}
            \mu^{-1} m^{-1} \mu \delta^{-1} \mu^{-1} m \mu \mu^{-1} \mu \eta \mu^{-1} \mu \delta. 
    \end{eqnarray*}
    We write this as
    \begin{eqnarray*}
        \mu^{-1} \cdot m^{-1} \cdot \mu \delta^{-1}  \mu^{-1} \cdot m \cdot \mu \eta \mu^{-1} \cdot \mu \delta \mu^{-1} \cdot \mu
    \end{eqnarray*}
    where $ \mu \delta^{-1}  \mu^{-1} $,  $\mu \eta  \mu^{-1}$,  $\mu \delta  \mu^{-1}$ are belong to $N$,  since $N$ is normal in ${P}$. So the conjugate does not change the form. Then we can replace $\gamma$ by $\gamma_s$.

    Assume that there exists a parabolic subgroup $P_1 \subset P$,  such that \[\gamma_s \in M_1 ,\quad \gamma_u \in M(\gamma_s) \cap N_1.\]  
    The Lie algebra of $N$ can be decomposed into eigenspaces under the action of $A_1$, then there exists a sequence \[N=N_0\supset N_1 \supset . . . \supset N_n =\{e\}\] of normal $\gamma_s$-stable subgroups of $N$, with the properties that $N_{k+1}\setminus N_{k}$ is abelian for $ n-1\ge k\ge 0$,  and $\eta^{-1} \delta^{-1} \eta  \delta$ belongs to $N_{k+1}$ for any $\delta \in N_k$ and $\eta \in N$ or $\eta=\gamma_u$. 
       
    We claim that 
    \begin{equation}\label{61}
    \sum_{\delta \in N(\gamma_s)N_k \setminus N}  \sum_{\eta \in N(\gamma_s)N_k}\phi \left( \gamma^{-1}\delta^{-1}\gamma\eta\delta\right)
       \end{equation} 
       equals
       \begin{equation}\label{62}
           \sum_{\delta \in N(\gamma_s) \setminus N} \sum_{\eta \in N(\gamma_s)} \phi(\gamma^{-1}\delta^{-1}\gamma\eta\delta). 
       \end{equation}
It is easy to see that (\ref{Lemma6.2}) is the case of $k=0$.

The equality $(\ref{61})=(\ref{62})$ holds when $k=n-1$. Suppose
\begin{eqnarray*}
\sum_{\delta \in N(\gamma_s)N_{k+1} \setminus N} \sum_{\eta \in N(\gamma_s)N_{k+1}}\phi \left( \gamma^{-1}\delta^{-1}\gamma\eta\delta\right)
\end{eqnarray*}
is the sum over $\delta_1 \in N(\gamma_s)N_k \setminus N$ of
\begin{eqnarray*}
\sum_{\delta_2 \in N(\gamma_s)N_{k+1} \setminus N(\gamma_s)N_k} \sum_{\eta \in N(\gamma_s)N_{k+1}}\phi(\gamma^{-1}\delta_1^{-1}\delta_2^{-1}\gamma\eta\delta_2\delta_1), 
\end{eqnarray*}
which becomes
\begin{eqnarray*}
\sum_{\delta_2 \in N_k(\gamma_s)N_{k+1} \setminus N_k} \sum_{\eta} \phi(\gamma^{-1}\delta_1^{-1}\delta_2^{-1}\gamma\eta\delta_2\delta_1). 
\end{eqnarray*}
Fix $\delta_2 \in N_k$,  we change the variables in the inner sum over $\eta$.  We have
\begin{eqnarray*}
   \sum_{\eta \in N(\gamma_s)N_{k+1}} \phi(\gamma^{-1}\delta_1^{-1}\delta_2^{-1}\gamma\eta\delta_2\delta_1) = \sum_{\eta} \phi(\gamma^{-1}\delta_1^{-1}\gamma \cdot \gamma^{-1}\delta_2^{-1}\gamma\eta\delta_2 \cdot \delta_1), 
\end{eqnarray*}
by \[\gamma^{-1} \delta_{2}^{-1} \gamma = \gamma_{s}^{-1} \gamma_{u}^{-1} \delta_{2}^{-1} \gamma_u \delta_2 \cdot \delta_{2}^{-1} \gamma_s. \] 
Since $\delta_2^{-1} \in N_k(\gamma_s)N_{k+1}$, it becomes
\begin{eqnarray*}
\sum_{\eta} \phi(\gamma^{-1}\delta_1^{-1}\gamma \cdot \gamma_s^{-1}\delta_2^{-1}\gamma_s\eta\delta_2 \cdot \delta_1)=\psi(\gamma_s^{-1}\delta_2^{-1}\gamma_s\delta_2). 
\end{eqnarray*} 
where \[\psi(x) = \sum_{\eta \in N(\, \gamma_s) N_{k+1}} \phi(\gamma^{-1}\delta_1^{-1}\gamma\eta \cdot x \cdot \delta_1)\]is a compactly supported function on the discrete set
\(N_{k}(\gamma_s)N_{k+1}\setminus N_k\). 
 The map 
 \begin{eqnarray*}
     y\rightarrow N_{k}(\gamma_s)N_{k+1}\cdot \gamma_s^{-1}y^{-1}\gamma_s y,  \qquad y\in N_{k}(\gamma_s)N_{k+1}\setminus N_k
 \end{eqnarray*}
 is an isomorphism from \(N_{k}(\gamma_s)N_{k+1}\setminus N_k\) onto itself.  Therefore 
 \begin{eqnarray*}
     \sum_{\delta_2 \in N_k(\gamma_s) N_{k+1} \setminus N_{k}} \psi(\gamma_s^{-1}\delta_2^{-1}\gamma_s\delta_2)= \sum_{\eta \in N(\gamma_s) N_k} \phi(\gamma^{-1}\delta_1^{-1}\gamma\eta\delta_1).
 \end{eqnarray*}
which is the case of (\ref{61}) at $k$. 
\end{proof}
Similarly, we have
\begin{lemma}\label{l2}
    $\int_{N(\gamma_s)_\mathbb{A}\backslash N_\mathbb{A}}\int_{N(\gamma_s)_\mathbb{A}}\phi(\gamma^{-1}n_1^{-1}\gamma n_2n_1)dn_2dn_1=\int_{N_\mathbb{A}}\phi(n)dn. $
\end{lemma}
\begin{lemma}
    For any fixed parabolic subgroup $P$, if $\gamma\in P_\mathbb{Q}$, then $\gamma$ is $P_\mathbb{Q}$-conjugate to an element $\gamma v$, where $\gamma\in M_\mathbb{Q}, v\in N(\gamma_s)_\mathbb{Q}$. 
\end{lemma}
\begin{proof}
    Since every element in $P_\mathbb{Q}$ can be written as $\gamma\eta$, for $\gamma\in M_\mathbb{Q}, \eta\in N_\mathbb{Q}$, by Lemma \ref{l1}, we can find  $\delta\in N_\mathbb{Q}, v\in N(\gamma_s)$ such that \[\eta=\gamma^{-1}\delta^{-1}\gamma v\delta. \]That is, \[\gamma\eta=\delta^{-1}\gamma v\delta. \]
\end{proof}
\begin{lemma}\label{l3}
    Fix an unramified orbit $\mathfrak{o}$ and a parabolic subgroup $P=P_\mathfrak{o}$. Suppose $\delta_1, \delta_2\in G_\mathbb{Q}$ satisfy \[\delta_1^{-1}\gamma_1\delta_1=\delta_2^{-1}\gamma_2\delta_2, \]for $\gamma_1, \gamma_2\in M_{\mathbb{Q}}^\mathfrak{o}$. Then there exists $w_s\in M\backslash N_G(A)$, a representative of the Weyl group, such that $\delta_1\delta_2^{-1}\in Mw_s$ . 
\end{lemma}
\begin{proof}
Let $\epsilon=\delta_1\delta_2^{-1}$, and $a\in A$, then \[a\gamma_i=\gamma_ia, \quad i=1, 2. \]
Consider 
\begin{eqnarray}\label{equition1}
    \epsilon^{-1}\gamma_1^{-1}a\gamma_1\epsilon=\epsilon^{-1}a\epsilon.
\end{eqnarray}
The left-hand side of (\ref{equition1}) equals\[\epsilon^{-1}\gamma_1^{-1}\epsilon\cdot\epsilon^{-1}a\epsilon\cdot\epsilon^{-1}\gamma_1\epsilon, \]
which equals \[\gamma_2^{-1}\cdot\epsilon^{-1}a\epsilon\cdot\gamma_2. \]
Hence, \[\epsilon^{-1}a\epsilon\in G(\gamma_2).\] $G(\gamma_2)$ is a maximal torus.

Similarly, we have $\epsilon^{-1}a\epsilon\in G(\gamma_1)$. Therefore, \[\epsilon^{-1}a\epsilon\in G(\gamma_1)\cap G(\epsilon^{-1}\gamma_1\epsilon). \]
However, easy to check that \[G(\epsilon^{-1}\gamma_1\epsilon)=\epsilon^{-1}G(\gamma_1) \epsilon. \]
Since $\gamma_1, \gamma_2$ are both in the unramified orbit, \[G(\gamma_1)\subset M,\quad G(\gamma_2)\subset M.\] Thus, if \[G(\gamma_1)\cap \epsilon^{-1}G(\gamma_1)\epsilon=A, \]then \[\epsilon^{-1}a\epsilon\in A.\]
It indicates that $\epsilon\in N_G(A)$.

If the intersection $G(\gamma_1)\cap G(\epsilon^{-1}\gamma_1\epsilon)$ contains an elements $g$ not in $A$, then there exist elements $m_1, m_2\in G(\gamma_1)$ such that \[g=m_1=\epsilon^{-1}m_2\epsilon, \]
$\epsilon$ must in $N_G(G(\gamma_1))\subset N_G(A)$. In fact, it implies \[\epsilon\in Mw_s,\] for $w_s\in M\backslash N_G(A)$.  
\end{proof}

For any parabolic subgroup $P$, define 
\[M_t=\{\gamma\in M(\mathbb{Q}):N(\gamma)=\{e\}\},\]and 
\[M_n=\{\gamma\in M(\mathbb{Q}):N(\gamma)\neq\{e\}\}.\]

Wirte $\{M_t\}$ and $\{M_n\}$ as fixed sets of representatives of $M_\mathbb{Q}$-conjugacy classes in $M_t$ and $M_n$. 
We now describe the geometric side of the trace formula of $\rm{GL(4)}$.

We pick out a special term for which we shall prove that the associated integral is absolutely convergent. 
Define\[I_G(f, x)=\sum_{\gamma\in \{G_e\}}(n_{\gamma, G})^{-1}\sum_{\delta\in G(\gamma)_\mathbb{Q}\backslash G_\mathbb{Q}}f(x^{-1}\delta^{-1}\gamma\delta x), \]
where $G_e$ denotes the set of $G$-elliptic elements in the orbit $\mathfrak{o}_G$.

By Lemma \ref{l3}, we can write the terms associated to unramified orbits $I_{\rm{unram}}^\mathfrak{o}(f, x)$ as 
\begin{eqnarray*}
    &\sum_{\delta\in N_G(A_\mathfrak{o})_{\mathbb{Q}}\backslash G_\mathbb{Q}}\sum_{\gamma\in M_{t, \mathfrak{o}}^{\mathfrak{o}}}f(x^{-1}\delta^{-1}\gamma\delta x). 
\end{eqnarray*}
We write this term as
\begin{eqnarray}\label{s1}
    \frac{1}{|M_{\mathfrak{o}}\backslash N_G(A_\mathfrak{o})|}\sum_{\delta\in M_{\mathfrak{o},\mathbb{Q}}\backslash G_\mathbb{Q}}\sum_{\gamma\in M_{t, \mathfrak{o}}^{\mathfrak{o}}}f(x^{-1}\delta^{-1}\gamma\delta x).
\end{eqnarray}
Which equals
\begin{eqnarray*}
    \frac{1}{|M_\mathfrak{o}\backslash N_G(A_\mathfrak{o})|}\sum_{\gamma\in \{M_{t, \mathfrak{o}}^{\mathfrak{o}}\}}(n_{\gamma,M})^{-1}\sum_{\delta\in M(\gamma)_{\mathfrak{o},\mathbb{Q}}\backslash G_\mathbb{Q}}f(x^{-1}\delta^{-1}\gamma\delta x).
\end{eqnarray*}
Define $\Omega(\mathfrak{a},P_1)$ to be the set of elements $s$ in $\cup_{P'}\Omega(\mathfrak{a}, \mathfrak{a}')$ such that if $\mathfrak{a}'=s\mathfrak{a}, \mathfrak{a}'$ contains $\mathfrak{a}_1$, and $s^{-1}\alpha$ is positive for every $\alpha\in\Phi_{P'}^{P_1}=\Phi_{P'}-\Phi_{P_1}$. 
\begin{lemma}\label{l6.6}
    Suppose $\mathfrak{o}$ is unramified, $P, P_\mathfrak{o}$ are parabolic subgroups. Then the expression (\ref{s1}) can be written as \[\frac{1}{|\Omega(\mathfrak{a}_\mathfrak{o}, P)|}\sum_{\delta \in M_\mathbb{Q}\backslash G_\mathbb{Q}}\sum_{\gamma\in M_{t}^\mathfrak{o}}f(x^{-1}\delta^{-1} \gamma\delta x). \]
\end{lemma}
\begin{proof}
    We wirte the term (\ref{s1}) as 
    \begin{eqnarray*}
       \frac{1}{|M_{\mathfrak{o}}\backslash N_G(A_\mathfrak{o})|}\sum_{\delta_1\in M_\mathbb{Q}\backslash G_\mathbb{Q}}\sum_{\delta_2\in M_{\mathfrak{o}, \mathbb{Q}}\backslash M_\mathbb{Q}}\sum_{\gamma\in M_{t, \mathfrak{o}}^\mathfrak{o}}f(x^{-1}\delta_1^{-1}\delta_2^{-1}\gamma \delta_2\delta_1x). 
    \end{eqnarray*}
    Notice that the sums over $\delta_2$ and $\gamma$ range over the orbit in $M$ if we multiply by \[\frac{1}{|M_\mathfrak{o}\backslash N_{M}(A_\mathfrak{o})|}.\]

    Thus the expression becomes
    \begin{eqnarray*}
      \frac{|M_\mathfrak{o}\backslash N_M(A_\mathfrak{o})|}{|M_\mathfrak{o}\backslash N_G(A_\mathfrak{o})|}\sum_{\delta \in M_\mathbb{Q}\backslash G_\mathbb{Q}}\sum_{\gamma\in M_{t}^\mathfrak{o}}f(x^{-1}\delta^{-1} \gamma\delta x). 
    \end{eqnarray*}
    And we can obtain that \[\frac{|M_\mathfrak{o}\backslash N_M(A_\mathfrak{o})|}{|M_\mathfrak{o}\backslash N_G(A_\mathfrak{o})|}=\frac{1}{|\Omega(\mathfrak{a}_\mathfrak{o}, P)|}. \]
\end{proof}
Now we proceed to compute the terms associated to ramified orbits.

If $\mathfrak{o}$ is ramified, let $M_{\{\mathfrak{o}\}}$ be minimal the Levi subgroup such that \[G(\gamma)=M_{\{\mathfrak{o}\}}(\gamma),\quad \gamma\in M_{\mathfrak{o},n}^\mathfrak{o},\]and \[M_{\{\mathfrak{o}\},i_1i_2...i_n},\quad i_1\ge...\ge i_n.\]

Assume that $\mathfrak{o}$ is a ramified orbit, denote 
\[I^{\mathfrak{o}}(f,x)=\sum_{\gamma\in G^\mathfrak{o}}f(x^{-1}\gamma x).\]

Define $M_s$ to be the subset of $M$ consisting of semisimple elements.

Since the integrals of some orbits over $Z_\infty^+G_\mathbb{Q}\backslash G_\mathbb{A}$ are divergent, we have to  introduce a characteristic function to control them.

Let $\hat{\tau}_P$ be the characteristic function of the set \[\{Z\in \mathfrak{a}_0:\hat{\alpha}(Z)>0, \hat{\alpha}\in\hat{\Phi}_P\}. \]
Take $T\in\mathfrak{a}_0^+$. We say $T$ is large enough, the mean is that $T$ is far away from the walls.

Write the terms corresponding to the unramified orbits as the sum of
\begin{eqnarray*}
    J_{\rm{unram}}^{\mathfrak{o}}(f, x, T)=\frac{1}{|M_\mathfrak{o}\backslash N_G(A_{\mathfrak{o}})|}\sum_{\delta\in M_{\mathfrak{o}, \mathbb{Q}}\backslash G_\mathbb{Q}}\sum_{\gamma\in M_{t, \mathfrak{o}}^{\mathfrak{o}}}f(x^{-1}\delta^{-1}\gamma\delta x)\\
    (\sum_{P_1\neq G}(-1)^{\rm{dim}(Z\backslash A_1)+1}\sum_{s\in \Omega(\mathfrak{a}, P_1)}\hat{\tau}_{P_1}(H_0(w_s\delta x)-T))
\end{eqnarray*}
and 
\begin{eqnarray*}
    I_{\rm{umram}}^{\mathfrak{o}}(f, x, T)=\frac{1}{|M_\mathfrak{o}\backslash N_G(A_{\mathfrak{o}})|}\sum_{\delta\in M_{\mathfrak{o}, \mathbb{Q}}\backslash G_\mathbb{Q}}\sum_{\gamma\in M_{t, \mathfrak{o}}^{\mathfrak{o}}}f(x^{-1}\delta^{-1}\gamma\delta x)\\
    (1+\sum_{P_1\neq G}(-1)^{\rm{dim}(Z\backslash A_1)}\sum_{s\in \Omega(\mathfrak{a}, P_1)}\hat{\tau}_P(H_0(w_s\delta x)-T)). 
\end{eqnarray*}
If the orbit $\mathfrak{o}$ is ramified, we consider
\begin{eqnarray*}
    J_{\rm{ram}}^{\mathfrak{o}}(f, x, T)=\sum_{P\neq G}(-1)^{(\rm{dim}Z\backslash A)+1}\sum_{\gamma\in M^\mathfrak{o}}\sum_{\delta\in M_\mathbb{Q}N(\gamma_s)_\mathbb{Q}\backslash G_\mathbb{Q}}\sum_{v\in N(\gamma_s)_\mathbb{Q}}f(x^{-1}\delta^{-1}\gamma v\delta x)\hat{\tau}_P(H_0(\delta x)-T)
\end{eqnarray*}
and 
\begin{eqnarray*}
    I_{\rm{ram}}^{\mathfrak{o}}(f, x, T)=\sum_{P}(-1)^{\rm{dim}Z\backslash A}\sum_{\gamma\in M^\mathfrak{o}}\sum_{\delta\in M_\mathbb{Q}N(\gamma_s)_\mathbb{Q}\backslash G_\mathbb{Q}}\sum_{v\in N(\gamma_s)_\mathbb{Q}}f(x^{-1}\delta^{-1}\gamma v\delta x)\hat{\tau}_P(H_0(\delta x)-T). 
\end{eqnarray*}
By Lemma \ref{l1}, \[\sum_{\gamma\in M^\mathfrak{o}}\sum_{\delta\in M_\mathbb{Q}N(\gamma_s)_\mathbb{Q}\backslash G_\mathbb{Q}}\sum_{v\in N(\gamma_s)_\mathbb{Q}}f(x^{-1}\delta^{-1}\gamma v\delta x)\]equals
\[\sum_{\gamma\in M_n^\mathfrak{o}}\sum_{\delta\in P_\mathbb{Q}\backslash G_\mathbb{Q}}\sum_{v\in N_\mathbb{Q}}f(x^{-1}\delta^{-1}\gamma v\delta x),\]
which equals 
\[\sum_{\gamma\in \{M^\mathfrak{o}\}}(n_{\gamma,M})^{-1}\sum_{\delta_1\in M(\gamma)_\mathbb{Q}\backslash M_\mathbb{Q}}\sum_{\delta\in P_\mathbb{Q}\backslash G_\mathbb{Q}}\sum_{v\in N_\mathbb{Q}}f(x^{-1}\delta^{-1}\delta_1^{-1}\gamma\delta_1 v\delta x).\]
Since $N$ is normal in $P$, we replace $v$ by $\delta_1^{-1}v\delta_1$.

Then $J_{\rm{ram}}^\mathfrak{o}(f,x,T)$ becomes
\begin{eqnarray*}
\sum_{P\neq G}(-1)^{\rm{dim}(Z\backslash A)+1}\sum_{\gamma\in \{M^\mathfrak{o}\}}(n_{\gamma,M})^{-1}\sum_{\delta\in M(\gamma)_\mathbb{Q}N(\gamma_s)_\mathbb{Q}\backslash G_\mathbb{Q}}\sum_{v\in N(\gamma_s)_\mathbb{Q}}f(x^{-1}\delta^{-1}\gamma v\delta x)\hat{\tau}_P(H_0(\delta x)-T)
\end{eqnarray*}
and $I_{\rm{ram}}^\mathfrak{o}(f,x,T)$ becomes
\begin{eqnarray*}\label{ramified1}
    \sum_{P}(-1)^{\rm{dim}(Z\backslash A)}\sum_{\gamma\in \{M^\mathfrak{o}\}}(n_{\gamma,M})^{-1}\sum_{\substack{\delta\in M(\gamma)_\mathbb{Q}N(\gamma_s)_\mathbb{Q}\backslash G_\mathbb{Q}}}\sum_{v\in N(\gamma_s)_\mathbb{Q}}f(x^{-1}\delta^{-1}\gamma v\delta x)\hat{\tau}_P(H_0(\delta x)-T).
\end{eqnarray*}
\begin{lemma}\label{l6}
    For fixed ramified orbit $\mathfrak{o},$ and any parabolic subgroup $P$. Then
    \begin{eqnarray}
        \sum_{\gamma\in M^\mathfrak{o}}\sum_{\delta\in M_\mathbb{Q}N(\gamma_s)_\mathbb{Q}\backslash G_\mathbb{Q}}\sum_{v\in N(\gamma_s)_\mathbb{Q}}f(x^{-1}\delta^{-1}\gamma v\delta x)
    \end{eqnarray}
    can be written as 
    \begin{eqnarray*}
    &\sum_{\gamma\in M_{n}^\mathfrak{o}}\sum_{\delta\in M_\mathbb{Q}N(\gamma_s)_{\mathbb{Q}}\backslash G_\mathbb{Q}}\sum_{v\in N(\gamma_s)_\mathbb{Q}}f(x^{-1}\delta^{-1}\gamma v\delta x)\\
    &+\sum_{\gamma\in M_t^{\mathfrak{o}}}\sum_{\delta\in M_\mathbb{Q}\backslash G_\mathbb{Q}}f(x^{-1}\delta^{-1}\gamma \delta x). 
    \end{eqnarray*}
\end{lemma}
This lemma is clear, as $\gamma v\in MN(\gamma_s)$ has two types: \[\gamma v\in M_{n}^\mathfrak{o}N(\gamma_s) \text{ and } \gamma\in M^{\mathfrak{o}}_t.\] For example, 
\begin{eqnarray*}
    \gamma=\left(\begin{matrix}
        a&x&&\\
        &a&&\\
        &&b&\\
        &&&c\end{matrix}\right)\in M_{t,211}^{\mathfrak{o}_{1111}^{211}},\quad v=e
\end{eqnarray*}and 
\begin{eqnarray*}
    \gamma=\left(\begin{matrix}
        a&&&\\
        &b&&\\
        &&a&\\
        &&&c\end{matrix}\right)\in M_{n, 211}^{\mathfrak{o}_{1111}^{211}},\quad v=\left(\begin{matrix}
            1&&x&\\
            &1&&\\
            &&1&\\
            &&&1\end{matrix}\right)
        \in N(\gamma_s)_{211}. 
\end{eqnarray*}

Then, we consider the spectral side of the trace formula of $\rm{GL(4)}$.

Suppose $\Phi$ belongs to some $\mathscr{H}_P$. Recall that $E_P^{c_{P_1}}(\Phi, \lambda, x)$, the constant term of $E_P(\Phi, \lambda, x)$ associated to $P_1\in \mathfrak{P}$, is given by \[\sum_{s\in\Omega(\mathfrak{a}, \mathfrak{a}_1)}(M_P(s, \lambda)\Phi)(x)\rm{exp}(<s\lambda+\rho_{P_1}, H_{P_1}(x)>). \]
For $T\in\mathfrak{a}_0^+$, define $E^{'T}_P(\Phi, \lambda, x)$ to be
\begin{eqnarray*}
(-1)^{(\rm{dim}\ Z\backslash A)+1}\sum_{P_1\in\mathfrak{P}}\sum_{\delta\in P_{1, \mathbb{Q}}\backslash G_\mathbb{Q}}E_P^{c_{P_1}}(\Phi, \lambda, \delta x)\hat{\tau}_{P_1}(H_{P_1}(\delta x)-T). 
\end{eqnarray*}
For any function $\phi$ on $Z_\infty^+G_\mathbb{Q}\backslash G_\mathbb{A}$, denote the truncation operator $(\Lambda^T\phi)(x)$ by \[\sum_P(-1)^{\rm{dim}(Z\backslash A)}\sum_{\delta\in P_\mathbb{Q}\backslash G_\mathbb{Q}}\hat{\tau}_P(H(\delta x)-T)\int_{N_\mathbb{Q}\backslash N_\mathbb{A}}\phi(n\delta x)dn.\]
Let 
\[E_P^{''T}(\Phi, \lambda, x)=E_P(\Phi, \lambda, x)-E_P^{'T}(\Phi, \lambda, x). \]
Since \[\int_{N_{1,\mathbb{Q}}\backslash N_{1,\mathbb{A}}}E_P(\Phi,\lambda,nx)dn=0,\]if $\Omega(\mathfrak{a},\mathfrak{a}_1)$ is empty, thus \[E^{'T}_P(\Phi,\lambda,x)=\Lambda^T E_P(\Phi,\lambda,x).\] Also, \[\Lambda^T\phi=\phi,\]if $\phi$ is the cusp form. And \[\Lambda^T\circ \Lambda^T=\Lambda^T. \]
(see \cite{L1}).

Define 
\begin{eqnarray*}
K'_P(f, x, T)=\sum_{P_1, P_2\in\mathfrak{P}}\sum_\chi n(A)^{-1}\cdot(-1)^{(\rm{dim}\ Z\backslash A)+1}(\frac{1}{2\pi i})^{\rm{dim}\ Z\backslash A}\sum_{\delta\in P_{\mathbb{Q}}\backslash G_\mathbb{Q}} \\
\cdot\int_{i\mathfrak{a}_G\backslash i\mathfrak{a}}\sum_{\alpha, \beta\in \mathscr{B}_{P, \chi}}E^{c_{P_1}}_P(\Phi_\alpha, \lambda, \delta x)\overline{E^{c_{P_2}}_P(\Phi_\beta, \lambda, \delta x)}\hat{\tau}_{P}(H_P(\delta x)-T)d\lambda. 
\end{eqnarray*}

\begin{lemma}
    For any parabolic subgroups $P, P_1\in\mathfrak{P}$, $y\in G_\mathbb{A}$, and fixed $s, s'\in\Omega(\mathfrak{a}, \mathfrak{a}_1)$, then the expression \[\int_{i\mathfrak{a}_G\backslash i\mathfrak{a}}\sum_{\chi}|\sum_{\beta\in \mathscr{B}_{P, \chi}}( M_P(s, \lambda)\pi_P(\lambda, f)\Phi_\beta)(y)(\overline{M_P(s', \lambda)\Phi_\beta(y)})|d\lambda\] is finite. 
\end{lemma}
\begin{proof}
    Put \[\rm{R}_{P, \chi}(\lambda, f, y, x)=\sum_{\beta\in \mathscr{B}_{P, \chi}}(\pi_P(\lambda, f)\Phi_\beta)(y)\overline{\Phi_\beta(x)}, \] which is continuous in $y, x\in G_\mathbb{A}$.

    We write the expression \[( M_P(s, \lambda)\pi_P(\lambda, f)\Phi_\beta)(y)(\overline{M_P(s', \lambda)\Phi_\beta(y)})\] as \[( M_P(s, \lambda)\pi_P(\lambda, f)M_P(s'^{-1}, s'\lambda)M_P(s', \lambda)\Phi_\beta)(y)(\overline{M_P(s', \lambda)\Phi_\beta(y)}), \]which is \[( M_P(s, \lambda)M_P(s'^{-1}, s'\lambda)\pi_P(s'\lambda, f)M_P(s', \lambda)\Phi_\beta)(y)(\overline{M_P(s', \lambda)\Phi_\beta(y)}), \]
    by the properties of intertwining operator \[M_P(s, \lambda)^*=M_P(s^{-1}, -s\overline{\lambda})\]and \[M_P(s, \lambda)\pi_P(\lambda, f)=\pi_P(s\lambda, f)M_P(s, \lambda). \]
    Since $\{M_P(s, \lambda)\Phi_\beta\}$ also forms an orthonormal basis for $\mathscr{H}_{P, \chi}$, we see $\rm{R}_{P, \chi}(ss'^{-1}\lambda, f, y, x)$ is the kernel of the restriction of \[M_P(s, \lambda)M_P(s'^{-1}, s'\lambda)\pi(\lambda, f)\] to $\mathscr{H}_{P, \chi}$.

    By \[M_P(s, \lambda)\pi_P(\lambda, f)=\pi_P(s\lambda, f)M_P(s, \lambda)\] and \[\pi_P(\lambda, f)=\pi_P(\lambda, f^1)\pi_P(\lambda, f^2), \]we have 
    \begin{equation}\label{*}
    \begin{aligned}
    &M_P(s, \lambda)M_P(s'^{-1}, s'\lambda)\pi_P(\lambda, f^1)(M_P(s, \lambda)M_P(s'^{-1}, s\lambda)\pi_P(\lambda, f^1))^*\\
    &=M_P(s, \lambda)M_P(s'^{-1}, s'\lambda)\pi_P(\lambda, f^1)\pi_P(\lambda, (f^1)^*)M_P(s', \lambda)M_P(s^{-1}, s\lambda)\\
    &=M_P(s, \lambda)M_P(s'^{-1}, s'\lambda)\pi_P(\lambda, ^1f)M_P(s', \lambda)M_P(s^{-1}, s\lambda),
    \end{aligned}
    \end{equation}
    where $^1f=f^1*(f^1)^*,\, ^2f=f^2*(f^2)^*$.

    Since $M_P(s, \lambda)M_P(s^{-1}, s\lambda)=\rm{Id}$, (\ref{*}) equals
    \[M_P(s, \lambda)M_P(s'^{-1}, s'\lambda)\pi_P(\lambda, ^1f)M_P(s'^{-1}, s'\lambda)^{-1}M_P(s, \lambda)^{-1}. \]
    The above expression is \[\pi_P(ss'^{-1}\lambda, ^1f). \]Similar is $\pi_P(ss'^{-1}\lambda, ^2f).$
    Therefore, by the Cauchy-Schwartz inequality, the absolute value of \[\sum_{\beta\in\mathscr{B}_{P, \chi}}( M_P(s, \lambda)\pi_P(\lambda, f)\Phi_\beta)(y)(\overline{M_P(s', \lambda)\Phi_\beta(y)})\] is bounded by \[|\rm{R}_{P, \chi}(ss'^{-1}\lambda, ^1f, y, y)|^{\frac{1}{2}}|\rm{R}_{P, \chi}(ss'^{-1}\lambda, ^2f, y, y)|^{\frac{1}{2}}. \]
    Also, we have shown that for every finite set $S$ of $\chi$, \[\sum_{\chi\in S}|\rm{R}_{P, \chi}(\lambda, ^1f, y, y)|\] is bounded by a function $P(\lambda, ^1f, y, y) $, which is independent of $S$. Hence we can conclude \[\sum_{\chi} \rm{R}_{P, \chi}(ss'^{-1}z, ^1f, y, y)\] is bounded by $P_P(ss'^{-1}\lambda, ^1f, y, y)$. Since \[\int_{i\mathfrak{a}_G\backslash i\mathfrak{a}}P_P(\lambda, ^1f, y, y)d\lambda\]is convergant, this lemma follows.
\end{proof}
We shall decompose the integral over $\mathfrak{a}_G\backslash\mathfrak{a}$ into lines along which the dual simple roots lie such that \[\mathfrak{a}^G_{P}=\mathfrak{a}_{P}^{P_{i_1}}\oplus\mathfrak{a}_{P_{i_1}}^{P_{i_2}}\oplus...\oplus\mathfrak{a}_{P_{i_j}}^G,\]where $\mathfrak{a}_{P_i}^{P_{t}}=\mathfrak{a}_{P_t}\backslash \mathfrak{a}_{P_i}$, $j$ is the rank of $A_P$. Write $a=\sum_{k=1}^{j}a_{k}\hat{\alpha}_{i_k}$, and $\alpha_{i_k}\in\Phi_P$.

Define $a_P=\rm{det}(<\hat{\alpha}_{i_n}, \hat{\alpha}_{i_m}>_{m, n})^{\frac{1}{2}}. $

If $P,P_1$ and $s,s_1\in\Omega(\mathfrak{a},\mathfrak{a}_1)$ are given, denote the function $L(s_1,s_2,f,x)$ by 
\[\int_{i\mathfrak{a}_G\backslash i\mathfrak{a}}\sum_\chi\sum_{\beta\in \mathscr{B}_{P, \chi}}((M_P(s_1, \lambda)\pi_P(\lambda, f)\Phi_\beta)(\delta x))(\overline{M_P(s_2, \lambda)\Phi_\beta(\delta x)})d\lambda,\]
and $L'(s_1,s_2,f,kp,a)$ by 
\[\int_K\int_{P_\mathbb{Q}\backslash P_\mathbb{A}}\int_{i\mathfrak{a}_G\backslash i\mathfrak{a}}\sum_\chi\sum_{\alpha,\beta\in \mathscr{B}_{P, \chi}}((M_P(s_1, \lambda)\Phi_\alpha)(\delta kpa))(\overline{M_P(s_2, \lambda)\Phi_\beta(\delta kpa)})\rm{exp}(<2\lambda,a>)d\lambda\ dp\ dk.\]
\begin{lemma}\label{l7}
    Fix $s_1, s_2\in\Omega(\mathfrak{a}, \mathfrak{a}_1)$, such that $M_P(s_1, \lambda)\neq M_P(s_2, \lambda)$, then the function 
    \begin{eqnarray}\label{eq2220}
    \frac{1}{(2\pi i)^{\rm{dim}(Z\backslash A)}\cdot n(A)}\sum_{\delta\in P_\mathbb{Q}\backslash G_\mathbb{Q}} L(s_1,s_2,f,x)\rm{exp}(<-2\rho_P, H_P(\delta x)>)\hat{\tau}_P(H_P(\delta x)-T)
    \end{eqnarray}
    is locally integrable over $Z_\infty^{+}G_\mathbb{Q}\backslash G_\mathbb{A}$ and its integral tends $0$ as $T\rightarrow \infty$. 
\end{lemma}
\begin{proof}
    Form the previous lemma, the  function inside is locally integrable. The integral of the absolutely value of (\ref{eq2220}) over $Z_\infty^{+}G_\mathbb{Q}\backslash G_\mathbb{A}$ is 
    \begin{eqnarray*}
    \frac{1}{(2\pi i)^{\rm{dim}(Z\backslash A)}}\frac{1}{n(A)}\int_{<T, \hat{\alpha}_{i_1}>}^\infty. . . \int_{<T, \hat{\alpha}_{i_k}>}^\infty\left|L'(s_1,s_2,f,kp,a)\right|da_{i_1}. . . da_{i_k}, 
    \end{eqnarray*}
    The function \[\pi_P(\lambda, f)\Phi_\beta\]vanishes for all but finitely many $\beta$.

    So, the integral over $\lambda$ in $L'(s_1,s_2,f,kp,a)$, we can change the contour to \[\{\lambda:<\rm{Re}\ \lambda, \alpha_k>=\delta,\quad  \alpha_k\in\Phi_P\},\] for $\delta<0$, such that the integral of exponential function can be finite. The integral approaches $0$ as $T\rightarrow \infty$. 
\end{proof}
By the property of the truncation operator\[\Lambda^{T}\circ\Lambda^{T}=\Lambda^T, \]
we have \[\Lambda^{T}E_P^{''}(\Phi,\lambda,x)=E_P^{''}(\Phi,\lambda,x),\]and \[\Lambda^{T}E_P^{'}(\Phi,\lambda,x)=0.\]
Then
\begin{eqnarray*}
&(E_P(\Phi_1, \lambda_1, x), E_P(\Phi_2, \lambda, x))\\
&=(E^{''T}_P(\Phi_1, \lambda_1, x)+E^{'T}_P(\Phi_1, \lambda_1, x), E^{''T}_P(\Phi_2, \lambda, x)+E^{'T}_P(\Phi_2, \lambda, x))
\end{eqnarray*}is
\begin{eqnarray*}
    &(E^{''T}_P(\Phi_1, \lambda_1, x), E^{''T}_P(\Phi_2, \lambda, x))+(\Lambda^TE^{''T}_P(\Phi_1, \lambda_1, x), \Lambda^TE^{'T}_P(\Phi_2, \lambda, x))\\
    &+(\Lambda^TE^{'T}_P(\Phi_1,\lambda_1, x), \Lambda^TE^{''T}_P(\Phi_2, \lambda, x))+(E^{'T}_P(\Phi_1, \lambda_1, x), E^{'T}_P(\Phi_2, \lambda, x)),
\end{eqnarray*}
which equals
\begin{eqnarray*}
    &(E^{''T}_P(\Phi_1, \lambda_1, x), E^{''T}_P(\Phi_2, \lambda, x))+(E^{'T}_P(\Phi_1, \lambda_1, x), E^{'T}_P(\Phi_2, \lambda, x)). 
\end{eqnarray*}
Thus, we define
\[K''_P(f, x, T)=K_{P}(f, x)-K^{'}_P(f, x, T). \]
We can now write $K(x, x)-K_1(x, x)$ as the sum of following six terms:
\begin{flalign}
I_G(f, x)\label{66}
\end{flalign}
and
\begin{flalign}
&&+J_{\rm{unram}}^{\mathfrak{o}^0_{31}}(f, x, T)-K'_{P_{31}}(f, x, T)-K'_{P_{13}}(f, x, T)+I_{\rm{unram}}^{\mathfrak{o}^0_{31}}(f, x, T)\label{67}\\
&&+J_{\rm{unram}}^{\mathfrak{o}_{22}^0}(f, x, T)-K'_{P_{22}}(f, x, T)+I_r^{\mathfrak{o}_{22}^0}(f, x, T)+I_{\rm{ram}}^{\mathfrak{o}_{22}^2}(f, x, T)\\
&&+J_{\rm{unfrm}}^{\mathfrak{o}_{211}^0}(f, x, T)+J_{\rm{ram}}^{\mathfrak{o}_{211}^2}(f, x, T)-\sum_{P\in P_{211}}K'_{P}(f, x, T)+I_{\rm{unram}}^{\mathfrak{o}_{211}^0}(f, x, T)+I_{\rm{ram}}^{\mathfrak{o}_{211}^2}(f, x, T)\\
&&+J_{\rm{unram}}^{\mathfrak{o}_{1111}^0}(f, x, T)+\sum_{i}J_{\rm{ram}}^{\mathfrak{o}_{1111}^i}(f, x, T)-K_{P_{1111}}^{'}(f, x, T)+I_{\rm{unram}}^{\mathfrak{o}_{1111}^0}(f, x, T)+\sum_{i}I_{\rm{ram}}^{\mathfrak{o}_{1111}^i}(f, x, T)\\
&&-\sum_{\mathfrak{P}}K_P^{''}(f, x, T)
\end{flalign}

We aim to prove the integrals of (\ref{66}) over $Z_\infty^+G_\mathbb{Q}\backslash G_\mathbb{A}$ is absolutely convergent.

For any orbit $\mathfrak{o}$ and parabolic subgroup $P_\mathfrak{o}$, we shall refer to these terms respectively as

the $G$-elliptic term \[I_G(f, x),\]the first parabolic term \[J^\mathfrak{o}_{\rm{unram}}(f, x, T)+J_{\rm{ram}}^\mathfrak{o}(f, x, T)-K'(f, x, T),\] the second parabolic term \[I_{\rm{unram}}^\mathfrak{o}(f, x, T)+I_{\rm{ram}}^\mathfrak{o}(f, x, T),\] the third parabolic term \[-\sum_{P\in\mathfrak{P}_\mathfrak{o}}K_P^{''}(f, x, T).\]

We claim that (\ref{66}) is integrable, the first parabolic terms are locally integrable and the values of them approach $0$ when $T\rightarrow \infty$. And the sum of the second parabolic term and third parabolic term is integrable and its value is independent of the parameter $T$. 

\section{The G-elliptic term}
In this section, we shall prove that the integral of $G$-elliptic term is absolutely convergant.

\begin{lemma}\label{l9}
    Suppose $C$ is any subset of $G_\mathbb{A}$ compact modulo $Z_\infty^+$. For fixed parabolic $P$, the number of elements $\gamma\in \{M_t\}\cup\{M_n\}$ such that there exists $x\in G_\mathbb{A},\, n\in N_\mathbb{A}$ with $x^{-1}\gamma nx\in C$ is finite. 
\end{lemma}
\begin{proof}
    Consider \[C_1=\{k^{-1}ck:c\in C,\,k\in K\}. \]Since $P$ is closed, the intersect of $C_1$ and $P_\mathbb{A}$ is compact modulo $Z_\infty^+$. Then we choose a subset $C_M\subset M_\mathbb{A}$ compact modulo $Z_\infty^+$ and satisfing \[C_1\cap P_\mathbb{A}\subset C_MN_\mathbb{A}. \]
    For $x^{-1}\gamma nx\in C$, write \[x=kp,\quad k\in K,\, p\in P_\mathbb{A},\] then \[p^{-1}\gamma n p\in C_MN_\mathbb{A}. \]
    By the definition of Siegel domain, we can choose $\omega$ a relatively compact set of representatives in $P_\mathbb{A}$, and write \[p=av\pi,\quad a\in A_\infty^+, \,v\in\omega, \,\pi\in P_\mathbb{Q}. \]
    Thus \[v^{-1}\pi^{-1}\cdot\gamma n\cdot \pi v\in a\cdot C_MN_\mathbb{A}\cdot a^{-1}=C_MN_\mathbb{A}. \]
    Choose a subset $C_M'\subset M_\mathbb{A}$ compact modulo $Z_\infty^+$, satisfing \[\omega\cdot C_MN_\mathbb{A}\cdot \omega^{-1}\subset C_M'N_\mathbb{A}. \]Then \[\pi^{-1}\gamma n\pi\in C_M'N_\mathbb{A}. \]
    Therefore $\gamma$ can be conjugated by $M_\mathbb{Q}$ into  $C_M'$.

    Since number of the elements in the intersection of a compact set and a set of finite elements is finite, we conclude that only finitely many $M_\mathbb{Q}-$conjugacy classes in $M_\mathbb{Q}$ meet $C_M'$. 
\end{proof}
Recall that $N(\gamma)$ is a subgroup of $N(\gamma_s)$. 
\begin{lemma}\label{l10}
Given parabolic subgroup $P$ and $\gamma\in\{M_t\}\cup\{M_n\}$. Suppose $C$ is a compact subset in $P_\mathbb{A}^1$. If $p\in P(\gamma)_\mathbb{A}^1\backslash P_\mathbb{A}^1$ satisfing \[(p^{-1}\cdot \gamma N(\gamma_s)_\mathbb{A}\cdot p)\cap C\neq \emptyset, \]then there exists a compact subset $C_1\subset P(\gamma)_\mathbb{A}^1\backslash P_\mathbb{A}^1$ such that $p\in C_1$. 
\end{lemma}
\begin{proof}
    Suppose the positive roots of $P$ are $\alpha_1, . . . , \alpha_n$. Denote the restriction of these elements to $P(\gamma)$ by $\alpha_i(\gamma),1\le i\le n$.
    Write \[\mathfrak{n}_i(\gamma)=\{X\in\mathfrak{n}(\gamma):\rm{Ad}(a)X=a^{\alpha_i(\gamma)}X, \,a\in A\}. \]$\mathfrak{n}_i(\gamma)$ is a subspace of $\mathfrak{n}(\gamma)$. Denote $\mathfrak{n}_i$ by $\mathfrak{n}_i(e)$. Write $\tilde{\mathfrak{n}}_i(\gamma)$ the compementary subspace of $\mathfrak{n}_i(\gamma)$ in $\mathfrak{n}_i$. Let $N_i(\gamma), \tilde{N}_i(\gamma)$ be the image of $\rm{exp}\ \mathfrak{n}_i(\gamma),\, \rm{exp}\ \tilde{\mathfrak{n}}_i(\gamma)$.

    It is clear that $\tilde{N}_i(\gamma)$ is the set of representatives for $N_i(\gamma)\backslash N_i$.

    Let $\omega$ be the relatively compact fundamental set in $P_\mathbb{A}^1$ for $P_\mathbb{Q}\backslash P_\mathbb{A}^1$, and define $C'$ the closure of $\omega\cdot C\cdot \omega^{-1}$ in $P_\mathbb{A}^1$.

    Set \[S=\{\delta\in P(\gamma)_\mathbb{Q}\backslash P_\mathbb{Q}:(\delta^{-1}\cdot \gamma N(\gamma)_\mathbb{A}\cdot \delta)\cap C'\neq\emptyset\}. \]
    If we can prove that the set is finite, let \[C_1'=\overline{\{\cup_{\delta\in S}\delta\omega\}},\] it is the closure in $P(\gamma)_\mathbb{Q}\backslash P_\mathbb{A}^1$. The $p$ which satisfies the condition of the lemma must lie in $C_1'$. If we write $C_1$ the projection of $C_1'$ onto $P(\gamma)_\mathbb{A}^1\backslash P_\mathbb{A}^1$, we can conclude that $C_1$ is the set we desired.

    Let $\{M\}_\gamma$ be the set of representatives of $M(\gamma)_\mathbb{Q}\backslash M_\mathbb{Q}$, then \[\{M\}_\gamma\Pi_{i=1}^n\tilde{N}_i(\gamma)\] is the set of representatives for $P(\gamma)_\mathbb{Q}\backslash P_\mathbb{Q}$ in $P_\mathbb{Q}$.

    Thus there exists a compact subset $C_M\subset M_\mathbb{A}^1$ satisfing $C'\subset C_MN_\mathbb{A}$.

    Write \[S_1=\{\delta\in \{M\}_\gamma:\delta\gamma\delta^{-1}\in C_M\}. \]Since $M_\mathbb{Q}$ is discrete in $C_M$, $S_1$ is finite.

    Thus $\cup_{\delta\in S_1}\delta C'\delta^{-1}$ is compact in $P_\mathbb{A}^1$.

    Also, \[\cup_{\delta\in S_1}\delta C'\delta^{-1}\subset M_\mathbb{A}^1\cdot \Pi_{i=1}^n\tilde{C}_{N_i}\cdot N_i(\gamma)_\mathbb{A}, \]
    where $\tilde{C}_{N_i}$ is some compact subset in $\tilde{N}_i(\gamma)_\mathbb{A}$.

    Set \[S_{N_i}=\{n_{i}\in\tilde{N}_i(\gamma)_\mathbb{Q}:\gamma^{-1}n_{i}^{-1}\gamma n_i\in \tilde{C}_{N_i}\cdot N_i(\gamma)_\mathbb{A}\}. \]
    Then $S_i$ is finite.

    The set \[\{\cup_{\delta\in S_1}\cup_{n_{i}\in S_i}n_{i}\delta C'\delta^{-1} n_{i}^{-1}\}\] is compact and is contained in \[M_\mathbb{A}^1\cdot N_{i, \mathbb{A}}.\]

    Since the product over $i$ is finite and the finite set \[S_1\cdot S_{N_i}\] contains a set of representatives of $S$ of cosets, the lemma follows.
\end{proof}
We now take an example.

For the ramified orbits $\mathfrak{o}_{22}^2$ and $\mathfrak{o}_{1111}^4$, we consider the semisimple elements in these orbits. Define \[I_{s}^\mathfrak{o}(f,x)=\sum_{\gamma\in\{G^\mathfrak{o}_{s}\}}(n_{\gamma,G})^{-1}\sum_{\delta\in G(\gamma)_\mathbb{Q}\backslash G_\mathbb{Q}}f(x^{-1}\delta^{-1} \gamma \delta x).\]
The integral 
\[\int_{Z_\infty^+G_\mathbb{Q}\backslash G_\mathbb{A}}|I_{s}^{\mathfrak{o}}(f, x)|dx\] is bounded by 
\[\sum_{\gamma\in\{G^\mathfrak{o}_{s}\}}(n_{\gamma, G})^{-1}\int_{Z_\infty^+G(\gamma)_\mathbb{Q}\backslash G_\mathbb{A}}|f(x^{-1}\gamma x)|dx,\] which is 
\[\sum_{\gamma\in\{G^\mathfrak{o}_{s}\}}(n_{\gamma, G})^{-1}\int_{Z_\infty^+G(\gamma)_\mathbb{Q}\backslash G(\gamma)_\mathbb{A}}dx_1\int_{G(\gamma)_\mathbb{A}\backslash G_\mathbb{A}}|f(x^{-1}\gamma x)|dx.\]

Since $f\in C_c^\infty(Z_\infty^+\backslash G_\mathbb{A})$, we can use Lemma \ref{l9} to see the sum over $\gamma$ is finite. 

Easy to see that the split component of this kind $G(\gamma)$ is $A_G=Z_G$. The first integral resemble the Tamagama number, but the measure on $Z_\infty^+$ can not be used directly.

Define \[\Gamma_{\gamma, G}=[X(G(\gamma))_\mathbb{Q}:X(G)_\mathbb{Q}|_{G(\gamma)}], \]and write \[\tilde{\tau}(\gamma, G)=(n_{\gamma, G})^{-1}(\Gamma_{\gamma, G})^{-1}\tau(G(\gamma)). \]
The integral of two kinds of orbits $\mathfrak{o}_{22}^2, \mathfrak{o}_{1111}^{4}$, equals  
\[\sum_{\gamma\in\{M^\mathfrak{o}_{n, \mathfrak{o}}\}}\tilde{\tau}(\gamma, G)\int_{G(\gamma)_\mathbb{A}\backslash G_\mathbb{A}}|f(x^{-1}\gamma x)|dx. \]

Then we decompose the integral into the integral over $K$ and $P(\gamma)^1_\mathbb{A}\backslash P^1_\mathbb{A}$.

By Lemma \ref{l10}, for fixed $\gamma$, the function on $P(\gamma)^1_\mathbb{A}\backslash P^1_\mathbb{A}$ which sends $p$ to \[\int_K|f(k^{-1}p^{-1}\gamma pk)|dk, \]is of compact support. Thus the integral of  singular term is absolutely integrable.

The integral of the absolute value of elliptic term is \[\int_{Z_\infty^+G_\mathbb{Q}\backslash G_\mathbb{A}}|\sum_{\gamma\in G_e}f(x^{-1}\gamma x)|dx,\]
which is bounded by the integral of \[\sum_{\gamma\in G_e}|f(x^{-1}\gamma x)|. \]
Define $\tau_1^P$ to be the characteristic function of \[\{H\in\mathfrak{a}_0:\alpha(H)>0,\, \alpha\in\Phi_{P_1}^P\}.\]
Suppose $\omega$ is a compact subset of $N_{0,\mathbb{Q}}M_{0,\mathbb{A}}^1$ and let $T_0\in -\mathfrak{a}_0^+$. For any parabolic subgroup $P_1$, define 
\[\mathfrak{s}^{P_1}(T_0,\omega)=\{ph_ak,\, p\in\omega,\, h_a\in A_{0,\infty}^+,\, k\in K:\alpha(H_0(h_a)-T_0)>0,\, \text{for each } \alpha\in \Phi_0^1\}.\]
Then
\[G_\mathbb{A}=P_{1,\mathbb{Q}}\cdot\mathfrak{s}^{P_1}(T_0,\omega).\]
We also define $\mathfrak{s}^{P_1}(T_0,T,\omega)$ to be the set 
\[\{x\in\mathfrak{s}^{P_1}(T_0,\omega):\hat{\alpha}(H_0(x)-T)\le 0,\text{for each } \hat{\alpha}\in\hat{\Phi}_0^1\}.\]
Let \[F^{P_1}(x,T)=F^1(x,T)\] be the characteristic function of
\[\{x\in G_\mathbb{A}:\delta x\in \mathfrak{s}^{P_1}(T_0,T,\omega),\,\text{for some } \delta\in P_{1,\mathbb{Q}}\}.\]

We have an equality.
\begin{eqnarray}\label{l71}
    \sum_{P_1:P_0\subset P_1\subset P}\sum_{\delta\in P_{1,\mathbb{Q}}\backslash G_\mathbb{Q}}F^1(\delta x,T)\tau_1^P(H_0(\delta x)-T)=1
\end{eqnarray}
for all $x\in G_\mathbb{A}$.

For example, if $G=\rm{GL}_3,P=P_{21}$ . Then $P_1=P_0$ or $P_{21}$.

For $x\in G_\mathbb{A}$, choose $\delta\in P_\mathbb{Q}$ such that $\delta\in \mathfrak{s}^P(T_0,\omega)$. It is then easy to see that $P_1=P_0$ satisfies the condition $\hat{\alpha}(H_0(\delta x)-T)\ge 0,\,\hat{\alpha}\in\hat{\Phi}_0^1$ and $\alpha(H_0(\delta x)-T)>0,\,\alpha\in\Phi_0^1$. Thus
\[F^1(\delta x,T)\tau_1^P(H_0(\delta x)-T)=1.\]
This shows that the sum is at least one.

Suppose there exists $\delta_1,\delta_2\in G_\mathbb{Q}$, and $P_1=P_0,P_2=P_{21}$ such that 
\[F^1(\delta_1 x,T)\tau_1^P(H_0(\delta_1 x)-T)=1=F^1(\delta_2 x,T)\tau_2^P(H_0(\delta_2 x)-T).\]
We may assume that \[\delta_i x\in \mathfrak{s}^{P_i}(T_0,T,\omega),\quad i=1,2,\] by translating $\delta_i$ by an element in $P_{i,\mathbb{Q}}$.

Thus the projection of $H_0(\delta_i x)-T,\, i=1,2$ onto $\mathfrak{a}_0^P$ can be written as 
\[c_1\hat{\alpha}_1\] and \[-c_2\alpha_1,\]
where $c_1,c_2>0$.

Now we use a standard result from reduction theory (see \cite{L1}): any suitably regular point $T\in\mathfrak{a}_0^+$ has the property, suppose $P_1\subset P$, and $x,\delta x\in\mathfrak{s}^{P_1}(T_0,\omega)$ for points $x\in G_\mathbb{A},\,\delta\in P_\mathbb{Q}$, if \[\alpha(H_0(x)-T)>0,\quad\alpha\in\Phi_0^P\backslash \Phi_0^{P_1},\] then\[\delta\in P_{1,\mathbb{Q}}.\]

Thus in this case, $\alpha(H_0(\delta_i x)-T)>0,\alpha\in\Phi_0^P\backslash \Phi_0^i$. Since $T\in T_0+\mathfrak{a}_0^+$, $\delta_i x\in\mathfrak{s}^P(T_0,\omega)$. We have $\delta_2\delta_1^{-1}\in P_{0,\mathbb{Q}},\,\delta_1\delta_2^{-1}\in P_{21,\mathbb{Q}}$. That is, there exists $\xi\in P_{0,\mathbb{Q}}$, such that \[\delta_2=\xi\delta_1.\]

However, $\delta_1\in P_{0,\mathbb{Q}}\backslash G_\mathbb{Q},\delta\in P_{21,\mathbb{Q}}\backslash G_\mathbb{Q}$, there is a contradiction.

Therefore $P_1=P_2, \,\delta_1,\,\delta_2$ belong to the same $P_{1,\mathbb{Q}}$ coset in $G_\mathbb{Q}$.

For the complete proof, see \cite{A3}.

Let $S\subset G_\mathbb{A}$ be the support of $f$,  then $Z_\infty^+\setminus S$ is compact.  Let $C$ be the closure in $G_\mathbb{A}$ of the set \[\omega_{T_0}^{-1}KSK\omega_{T_0}. \] $C$ is  compact modulo $Z_\infty^+$.

If $P_1\subset P_2$, define
\[\sigma_1^2(H)=\sigma_{P_1}^{P_2}(H)=\sum_{P_3:P_3\supset P_2}(-1)^{\rm{dim}(A_2\backslash A_3)}\tau_1^3(H)\cdot\hat{\tau}_3(H),\quad H\in\mathfrak{a}_0.\]
\begin{lemma} \cite{A3}\label{lemma7.1}
    If $P_2\supset P_1$, $\sigma_1^2$ is the characteristic function of the set of $H\in\mathfrak{a}_1$ such that 
    \begin{itemize}
        \item $\alpha(H)>0$, for all $\alpha\in\Phi_1^2$,
        \item $\alpha(H)\le0$, for all $\alpha\in \Phi_1\backslash \Phi_1^2$, and 
        \item $\hat{\alpha}(H)>0$, for all $\hat{\alpha}\in\hat{\Phi}_2$.
    \end{itemize}
\end{lemma}
We now write \[I_{\rm{ram}}^\mathfrak{o}(f,x,T)\]as the sum over $P$ of 
\begin{eqnarray}\label{eq0}
    (-1)^{\rm{dim} Z\backslash A}\sum_{\delta\in P_\mathbb{Q}\backslash G_\mathbb{Q}}\sum_{\gamma\in M^{\mathfrak{o}}}\sum_{v\in N_\mathbb{Q}}f(x^{-1}\delta^{-1}\gamma v\delta x)\sum_{\{P_1:P_0\subset P_1\subset P\}}\sum_{\xi\in P_{1,\mathbb{Q}}\backslash P_\mathbb{Q}}F^1(\xi\delta x,T)\tau_1^P(H_0(\xi\delta x)-T).\notag\\
    \end{eqnarray}
\begin{lemma}\label{in}
    Given $P,\,x,\,\mathfrak{o}$, (\ref{eq0}) equals 
    \begin{eqnarray}
    \sum_{\{P_1,P_2:P_1\subset P\subset P_2\}}\sum_{\delta\in P_{1,\mathbb{Q}}\backslash G_\mathbb{Q}}F^1(\delta x,T)\sigma^2_1(H_0(\delta x)-T)\sum_{\gamma\in M_1^{\mathfrak{o}}}\sum_{v\in N_{1,\mathbb{Q}}}f(x^{-1}\delta^{-1}\gamma v\delta x).
    \end{eqnarray}
\end{lemma}
\begin{proof}
    By the fact that\[\sum_{\{P:{P_1\subset P\subset P_2}\}}(-1)^{\rm{dim}(A_2\backslash A)}=\begin{cases}
    1,&\text{if}\quad P_1=P_2,\\
    0,&\text{else},
    \end{cases}\] we can write \[\tau_1^P(H_0(\xi\delta x)-T)\hat{\tau}_P(H_0(\xi\delta x)-T)\] as
    \[\sum_{\{P_2,P_3:P\subset P_2\subset P_3\}}(-1)^{\rm{dim}A_2\backslash A_3}\tau_1^3(H_0(\xi\delta x)-T)\hat{\tau}_3(H_0(\xi\delta x)-T),\]
    which is \[\sum_{\{P_2:P_2\supset P\}}\sigma_1^2(H_0(\xi\delta x)-T).\]
    Then the term (\ref{eq0}) becomes
    \[\sum_{\{P_1,P_2:P_1\subset P\subset P_2\}}\sum_{\delta\in P_{1,\mathbb{Q}}\backslash G_\mathbb{A}}(-1)^{\rm{dim} Z\backslash A}F^1(\delta x,T)\sigma_1^2(H_0(\delta x)-T)\sum_{\gamma\in M^\mathfrak{o}}\sum_{v\in N_\mathbb{Q}}f(x^{-1}\delta^{-1}\gamma v\delta x).\]
    We choose a representative $x\in G_\mathbb{A}^1$ such that 
    \[n_2n_0^2mh_ak, \quad k\in K,\] $n_2,n_0^2,m$ respectively belong to fixed compact subsets of $N_{2,\mathbb{A}}, N_{0,\mathbb{A}}^2, M_{0,\mathbb{A}}^1$. By the definition of $F^1(x,T)$, the element $a$ satisfies 
    \begin{eqnarray}\label{ine1}
        \alpha(H_0(h_a)-T_0)>0,\quad  \alpha\in \Phi_0^1,
    \end{eqnarray}
    and 
    \[\hat{\alpha}(H_0(h_a)-T)\le 0,\quad \hat{\alpha}\in \hat{\Phi}_0^1,\]
    then by Lemma \ref{lemma7.1}, if $\sigma_1^2(H_0(h_a)-T)\neq 0$,
    \begin{eqnarray}\label{ine2}
        \alpha(H_0(h_a)-T)>0,\quad  \alpha\in \Phi_1^2.
    \end{eqnarray}
    By the theorey of Siegel domain, for such $h_a$, the element $h_a^{-1}n_0^2mh_a$ lies in a fixed compact subset of $N_{0,\mathbb{A}}^2\times M^1_{0,\mathbb{A}}$.

    Suppose there exists \[\gamma\in M_\mathbb{Q}\cap P_{1,\mathbb{Q}}\backslash M_\mathbb{Q},\] such that 
    \begin{eqnarray}\label{_}
        \sum_{v\in N_\mathbb{Q}}f(k^{-1}h_a^{-1}m^{-1}(n_0^2)^{-1}n_2^{-1}\cdot \gamma v\cdot n_2n_0^2mh_ak)\neq 0.
    \end{eqnarray}
    Since $N_2$ is normal in $P$, we have \[n_2^{-1}v n_2\in N,\]
    the term (\ref{_}) equals 
    \[\sum_{v\in N_\mathbb{Q}}f(k^{-1}(h_a^{-1}mn_0^2h_a)^{-1}\cdot h_a^{-1}\gamma vh_a\cdot (h_a^{-1}mn_0^2h_a)k).\]
    Hence $h_a^{-1}\gamma h_a$ belongs to a compact subset of $M_\mathbb{A}^1$.

    Apply the Bruhat decomposition, for any $\gamma\in M_\mathbb{Q}$ we wirte \[\gamma = \nu w_s \pi,\quad \nu  \in N_{0}^P,  \pi \in P_{0}\cap M_\mathbb{Q},\] where $s$ belongs to the Weyl group of $(M,A_0)$, and $s$ cannot belong to the Weyl group of $(M_1, A_1)$.

 	We can therefore find $\hat{\alpha} \in \hat{\Phi}_{1}^P$ not fixed by $s$.

    Let $\rho$ be a rational representation of $G$ on the vector space $V$ with highest weight $d\hat{\alpha}$,  where $d>0$. Let $v$ be the highest weight vector in $V_\mathbb{Q}$, the space on which $G_\mathbb{Q}$ acts.

 	Choose a height function $\|\cdot\|$ relative to a basis of $V_\mathbb{Q}$ (the definition and properties of the height function are introduced in \cite{A3} or \cite{H1}), and we can assume that $v$ and $\rho(w_s)v$ is included in the basis.

 	Then the component of \[\rho(a^{-1}\gamma a)v=\rho(h_a^{-1}\nu w_s\pi h_a)v\] in the projection of $\rho(w_s)v$ is \[e^{d(\hat{\alpha}-s\hat{\alpha})(H_0(h_a))}\rho(w_s)v. \]

    Therefore,
    \begin{eqnarray}\label{inequality}
        \|\rho(h_a^{-1}\gamma h_a)v\|=\|\rho(h_a^{-1}\nu w_s\pi h_a)v\|\geq e^{d(\hat{\alpha}-s\hat{\alpha})H_0(h_a)}.
    \end{eqnarray}
 	The left side of (\ref{inequality}) is bounded since $h_a^{-1}\gamma h_a \in C$. However, $\hat{\alpha}-s\hat{\alpha}$ is a nonnegative sum of roots in $\Phi_{0}^P$,  and at least one element in $\Phi_{0}^P\backslash \Phi_1^P$ has non-zero coeffcients. The right side of (\ref{inequality}), can be made arbitarily large by large enough $T$ follows from (\ref{ine1}) and (\ref{ine2}). This leads to a contradiction.

\end{proof}
\begin{lemma}\label{l8}
 	The integral of \[\sum_{\gamma\in G^{\mathfrak{o}_G}}f(x^{-1}\gamma x)\] over $Z_\infty^+G_\mathbb{Q}\backslash G_\mathbb{A}$ is absolutely convergent.
\end{lemma}
\begin{proof}
    By the proof of Lemma \ref{in}, we let $P=G$, and $\mathfrak{o}=\mathfrak{o}_G$. In this case, $P_2=G$, after multiply the term (\ref{l71}), the lemma is clear since $\mathfrak{o}_G\cap P=\emptyset, P\neq G.$
\end{proof}
The integral of $G$-elliptic term is given by \[\sum_{\gamma\in\{G_e\}}(n_{\gamma, G})^{-1}\int_{Z_\infty^+G(\gamma)_{\mathbb{Q}}\backslash G(\gamma)_\mathbb{A}}dx_1\int_{G(\gamma)_\mathbb{A}\backslash G_\mathbb{A}}f(x^{-1}\gamma x)dx,\]
which is \[\sum_{\gamma\in \{G_e\}}\tilde{\tau}(\gamma, G)\int_{G(\gamma)_\mathbb{A}\backslash G_\mathbb{A}}f(x^{-1}\gamma x)dx. \]

\section{The convergence associated to some orbits}
In this section, we shll prove the convergence of ramified orbits and introduce the convex hull associated to unramified orbits.
\subsection{The second parabolic terms of ramified orbit}
For any parabolic subgroup $P$, let its simple roots be $\{\alpha_{i_1}, . . . , \alpha_{i_j}\}$.

Since we shall discuss the integral over $P(\gamma_s)_\mathbb{A}$ and $P(\gamma)^1_\mathbb{A}$, for $\gamma\in \{M^\mathfrak{o}\}$, we need to consider the Haar measure.

Define \[\delta_{P(\gamma_s)}(p)=\rm{exp}(-<2\rho_P(\gamma_s), H_0(p)>), \quad p\in P(\gamma_s)_\mathbb{A}\]to be the modular function of $P(\gamma_s)_\mathbb{A}$.

We consider the function $I_{\rm{ram}}^{\mathfrak{o}}(f, x, T)$, it equals
\[\sum_{P}(-1)^{(\rm{dim}\ Z\backslash A)}\sum_{\gamma\in M^\mathfrak{o}}\sum_{\delta\in M_\mathbb{Q}N(\gamma_s)_\mathbb{Q}\backslash G_\mathbb{Q}}\sum_{\substack{v\in N(\gamma_s)_\mathbb{Q}}}f(x^{-1}\delta^{-1}\gamma v\delta x)(\hat{\tau}_P(H_0(\delta x)-T)). \]
We need to consider the integral of its absolute value over $Z_\infty^+G_\mathbb{Q}\backslash G_\mathbb{A}$.

For any $P$, consider
\[\sum_{\gamma\in\{M^\mathfrak{o}\}}(n_{\gamma, M})^{-1}\int_{Z_\infty^+M(\gamma)_\mathbb{Q}N(\gamma_s)_\mathbb{A}\backslash G_\mathbb{A}}\sum_{\substack{v\in N(\gamma_s)_\mathbb{Q}}}f(x^{-1}\gamma vx)(\hat{\tau}_P(H_0(x)-T))dx, \]
it equals
\[c_P\sum_{\gamma\in\{M^\mathfrak{o}\}}(n_{\gamma, M})^{-1}\int_K\int_{Z_\infty^+M(\gamma)_\mathbb{Q}N(\gamma_s)_\mathbb{Q}\backslash P_\mathbb{A}}\sum_{\substack{v\in N(\gamma_s)_\mathbb{Q}}}f(k^{-1}p^{-1}\gamma v pk)(\hat{\tau}_P(H_0(p)-T))\delta_P(p)d_rp\ dk. \]
We can write it as 
\begin{equation}\label{eq8.1}
\begin{aligned}
&c_P\sum_{\gamma\in\{M^\mathfrak{o}\}}(n_{\gamma, M})^{-1}(\Gamma_{\gamma, M})\int_K\int_{Z_\infty^+\backslash A_\infty^+}\int_{M(\gamma)_\mathbb{A}^1N(\gamma_s)_\mathbb{A}\backslash P^1_\mathbb{A}}\int_{M(\gamma)_\mathbb{Q}N(\gamma_s)_\mathbb{Q}\backslash M(\gamma)^1_\mathbb{A}N(\gamma_s)_\mathbb{A}}\\
&\sum_{\substack{v\in N(\gamma_s)_\mathbb{Q}}}f(k^{-1}p^{*-1}\cdot p^{-1}h_a^{-1}\gamma vh_ap\cdot p^*k)(\hat{\tau}_P(H_0(a)-T))\delta_{P(\gamma_s)}(a)da\ d_rp\ dp^*\ dk. 
\end{aligned}
\end{equation}
By Lemma \ref{l10}, the integral over 
\[M(\gamma)_\mathbb{A}^1N(\gamma_s)_\mathbb{A}\backslash P_\mathbb{A}^1\subset P(\gamma)_\mathbb{A}^1\backslash P_\mathbb{A}^1\] can be replaced by a compact subset $C_1$ of $P(\gamma)_\mathbb{A}^1\backslash P_\mathbb{A}^1$, in other words,  over a compact set $C_1(\gamma_s)$ of representatives of $C_1$ in $P_\mathbb{A}^1$.

We define the function $\Phi_\gamma(f, n)$ to be \[c_P(n_{\gamma, M})^{-1}\int_K\int_{C_1(\gamma_s)}f(k^{-1}p^{-1}\gamma npk)dp\ dk, \]for fixed $\gamma\in\{M^{\mathfrak{o}}\}, n\in N(\gamma_s)_\mathbb{A}$. The support, $U(\gamma_s)$, of this function is a compact subset of $N(\gamma_s)_\mathbb{A}$.

If $a\in\mathfrak{a}_G\backslash \mathfrak{a}_P$, we denote $a$ by $\sum_{k=1}^{j}a_{k}\hat{\alpha}_{i_k}. $

Let $\omega(\gamma_s)$ be the relatively compact set of representatives of $M(\gamma)_\mathbb{Q}N(\gamma_s)_\mathbb{Q}\backslash M(\gamma)^1_\mathbb{A}N(\gamma_s)_\mathbb{A}$ in $M(\gamma)^1_\mathbb{A}N(\gamma_s)_\mathbb{A}$. Since $N(\gamma_s)_\mathbb{Q}$ is discrete in $N(\gamma_s)_\mathbb{A}$, we can choose positive numbers $t_1,...t_j$ small enough so that
\begin{eqnarray}\label{intersection}
    \{vh_a\cdot n\cdot h_a^{-1}v^{-1}, v\in \omega(\gamma), a_i\le t_i,n\in U(\gamma_s)\}\cap N(\gamma_s)_\mathbb{Q}=\{e\}.
\end{eqnarray}
Hence, we can write the integral over $Z_\infty^+\backslash A_{\infty}^+$ into the integrals over $a_k\ge t_k,\quad 1\le k\le j$.

Define $\hat{\tau}_P'$ to be the characteristic function
\[\{H\in \mathfrak{a}_0:\hat{\alpha}(H)\le0,\hat{\alpha}\in\hat{\Phi}_P\}.\]

\begin{lemma}\label{l999}
    Given a parabolic subgroup $P_1$, the characteristic function
    \begin{eqnarray}\label{pp}
       \sum_{\substack{P\supset P_1}}(-1)^{\rm{dim}\ Z\backslash A}\hat{\tau}_{P}(H_{0}(\delta x)-T) 
    \end{eqnarray} 
    equals \[\hat{\tau}_{P_1}'(H_{0}(\delta x)-T).\]
\end{lemma}
\begin{proof}
    If $\hat{\tau}_{P_1}(H_{0}(\delta x)-T)=1,$ then the other $\hat{\tau}_{P_2}$ take $1$, for $P_2\supset P_1$.

    If $\hat{\tau}_{P_1}(H_{0}(\delta x)-T)=0,$ then the other $\hat{\tau}_{P_2}$ may not all be $0$. Since $\hat{\tau}_G(H_0(\delta x)-T)=1$, we can find the minimal parabolic subgroup $P_m\supset P_1$, such that $\hat{\tau}_{P_m}(H_0(\delta x)-T)=1$. 
    
    Thus, for any $P$ such that $P\in \mathfrak{P}_m$, where $\mathfrak{P}_m$ is the associated class of $P_m$, \[\hat{\tau}_P(H_0(\delta x)-T)=0,\]
    otherwise we could find a smaller parabolic subgroup $P'_m$ such that $P'_m\subsetneq P_m$, which contradicts the minimality of $P_m$.

    Thus, \[\hat{\tau}_{P_m}(H_0(\delta x)-T)=1,\quad \text{for}\quad P\supset P_m.\]
    Hence, by the fact that\[\sum_{\{P:{P_1\subset P\subset P_2}\}}(-1)^{\rm{dim}(A_2\backslash A)}=\begin{cases}
    1,&\text{if}\quad P_1=P_2,\\
    0,&\text{else},
    \end{cases}\] we have \[1+\sum_{\substack{P\supset P_1\\P\neq G}}(-1)^{\rm{dim}\ Z\backslash A}\hat{\tau}_{P}(H_{0}(\delta x)-T)=\begin{cases}
        1,&\text{if}\quad P_m=G,\\
        0,&\text{if}\quad P_m\neq G.
    \end{cases}\] 
    The lemma follows.
\end{proof}
\begin{lemma}
    The integral  \[\int_{Z_\infty^+G_\mathbb{Q}\backslash G_\mathbb{A}}I_{\rm{ram}}^\mathfrak{o}(f,x,T)dx\] is absolutely convergent.
\end{lemma}
\begin{proof}
    By (\ref{eq8.1}), we have shown that the other integrals are convergent, we only need to verify the convergence of the integral over $\mathfrak{a}_G\backslash \mathfrak{a}_0$.

    By the proof of Lemma \ref{l8}, the integral \[\int_{Z_\infty^+G_\mathbb{Q}\backslash G_\mathbb{A}}|I_{\rm{ram}}^\mathfrak{o}(f,x,T)|dx\] is bounded by the integral of 
    \[\sum_{P_1}\sum_{\delta\in P_1\backslash G_\mathbb{Q}}F^1(\delta x,T)|\sum_{\{P:P_1\subset P\}}(-1)^{\rm{dim} Z\backslash A}\sum_{\gamma\in M_1^\mathfrak{o}}\sum_{v\in N_{1,\mathbb{Q}}}f(x^{-1}\delta^{-1}\gamma v\delta x)\hat{\tau}_P(H_0(\delta x)-T)|,\]

    since $0\le\tau_1^P(H_0(\delta x)-T)\le 1$.

    The expression is bounded by the sum over $P_1$ of
    \begin{eqnarray}\label{eq88888}
        \sum_{\delta\in P_1\backslash G_\mathbb{Q}}F^1(\delta x,T)\sum_{\gamma\in M_1^\mathfrak{o}}\sum_{v\in N_{1,\mathbb{Q}}}|f(x^{-1}\delta^{-1}\gamma v\delta x)|\cdot |\sum_{\{P:P_1\subset P\}}(-1)^{\rm{dim} Z\backslash A}\hat{\tau}_P(H_0(\delta x)-T)|.
    \end{eqnarray}
    Then by Lemma \ref{l999}, we have the equality
    \[\sum_{\{P:P_1\subset P\}}(-1)^{\rm{dim} Z\backslash A}\hat{\tau}_P(H_0(\delta x)-T)=\hat{\tau}_{P_1}'(H_0(\delta x)-T).\]
    Then the integral of (\ref{eq88888}) over $\mathfrak{a}_G\backslash \mathfrak{a}_1$ can be written as a integral over $a_i\le \hat{\alpha}_i(T)$.

    We write \[\mathfrak{a}_0=\mathfrak{a}_0^1\oplus\mathfrak{a}_1.\]
    Then by the definition of $F^1(x,T)$, we now have a upper bound over $\mathfrak{a}_G\backslash \mathfrak{a}_0.$

    Also, we have shown that we have a lower bound by (\ref{intersection}).

    Then the support of the integral of $I_{\rm{ram}}^\mathfrak{o}(f,x,T)$ is compact. Thus \[\int_{Z_\infty^+G_\mathbb{Q}\backslash G_\mathbb{A}}I_{\rm{ram}}^\mathfrak{o}(f,x,T)dx\] is absolutely convergant.
\end{proof}

In order to apply the tools of complex analysis, we shall turn the function \[\int_{Z_\infty^+G_\mathbb{Q}\backslash G_\mathbb{A}}I_{\rm{ram}}^{\mathfrak{o}}(f, x, T)dx\] into a function of $\lambda=(\lambda_1, \lambda_2, \lambda_3)\in \mathbb{C}^3$.

Given $P,\,f,\,\gamma,\,p,\,a$, if $P\neq G$, let $Y_P(f,x,\gamma,p,a)$ be 
\[\sum_{\substack{v\in N(\gamma_s)_\mathbb{Q}}}\Phi_\gamma(f, h_a^{-1}p^{-1}vph_a)\rm{exp}(-<2\rho_P(\gamma_s),\sum_{k=1}^j{(1+<\lambda, \alpha_{i_k}>)a_{k}\hat{\alpha}_{i_k}}>),\]
if $P=G$, we replace $\rho_P(\gamma_s)$ by $\rho_{P_0}(\gamma_s)$.

We define $I^\mathfrak{o}_T(\lambda)$ to be
\begin{eqnarray*}
    \sum_Pa_Pc_P\int_{M(\gamma)_\mathbb{Q}N(\gamma_s)\backslash M(\gamma)_\mathbb{A}^1N(\gamma_s)_\mathbb{A}}\sum_{\gamma\in\{M^\mathfrak{o}\}}\int_{<T, \hat{\alpha}_{i_1}>}^{\infty}. . . \int_{<T, \hat{\alpha}_{i_j}>}^{\infty}Y_P(f,x,\gamma,p,a)dp\ da_{1}. . . da_{j}. 
\end{eqnarray*}

\begin{lemma}
    For any $\lambda$, $I^\mathfrak{o}_T(\lambda)$ is absolutely convergent. The function $I_T(\lambda)$ is entire and its value at $\lambda=0$ is given by the integral 
    \[\int_{Z_\infty^+G_\mathbb{Q}\backslash G_\mathbb{A}}I_{\rm{ram}}^\mathfrak{o}(f,x,T)dx.\]
\end{lemma}

Now, we apply the Poisson summation formula over the group $N(\gamma_s)_\mathbb{Q}$, but it is not abelian.

We write the unipotent group $N_{1111}$ as \[(N_{1111}-N_{211})\oplus (N_{211}-N_{31})\oplus (N_{31}). \]
That is, we can express the group
\begin{eqnarray*}
    \left(\begin{matrix}
        1&*&*&*&\\
        &1&*&*&\\
        &&1&*&\\
        &&&1
    \end{matrix}\right)
\end{eqnarray*}
as the direct sum of three abelian subgroups,
\begin{eqnarray*}
    \left(\begin{matrix}
        1&*&&&\\
        &1&&&\\
        &&1&&\\
        &&&1
    \end{matrix}\right)\oplus\left(\begin{matrix}
        1&&*&&\\
        &1&*&&\\
        &&1&&\\
        &&&1
    \end{matrix}\right)\oplus\left(\begin{matrix}
        1&&&*&\\
        &1&&*&\\
        &&1&*&\\
        &&&1
    \end{matrix}\right). 
\end{eqnarray*}

Each term in this direct sum is an abelian subgroup. To apply Poisson summation formula to a subgroup $H$ of $N_{1111}$, we decompose $H$ into three subgroups corresponding to its intersections with the abelian summands above. Donote this intersections by $H_1, H_2, H_3$.

Let $X(\gamma_s)_\mathbb{A}$ denote the unitary dual group of $\mathfrak{n}(\gamma_s)_\mathbb{A}$ and let $X(\gamma_s)_\mathbb{Q}$ be the subset of $X(\gamma_s)_\mathbb{A}$ consisting of those elements which are trivial on $\mathfrak{n}(\gamma_s)_\mathbb{Q}$.

Recall that the symbol $||\cdot||$ denotes the height function on $X(\gamma_s)_\mathbb{A}$ associated to a fixed basis of $X(\gamma_s)_\mathbb{Q}$. $X(\gamma_s)_\mathbb{Q}$ is a subgroup of $ X_\mathbb{Q}$.

It is easy to verify that there exists $N\in\mathbb{R}$ such that \[\sum_{\substack{\xi\in X_\mathbb{Q}\\ \xi\neq 0}}||\xi||^{-N}<\infty. \]
For $\xi\in X_\mathbb{A}$ and $a\in \mathbb{R}^j$, define \[\xi^a(Y)=\xi(\rm{Ad}(h_a)Y),\quad Y\in\mathfrak{n}_\mathbb{A}. \]Then there exists a number $d>(0, . . . , 0), d\in\mathbb{R}^j$ such that if $\xi$ is primitive and $a\geq 0$, then \[||\xi^a||\geq \rm{exp}(<d, a>)\ ||\xi||. \]

We decompose the group \[N(\gamma_s)_\mathbb{Q}=N(\gamma_s)_{1, \mathbb{Q}}\oplus N(\gamma_s)_{2, \mathbb{Q}}\oplus N(\gamma_s)_{3, \mathbb{Q}}, \]and likewise\[\mathfrak{n}(\gamma_s)_\mathbb{A}=\mathfrak{n}(\gamma_s)_{1, \mathbb{A}}\oplus \mathfrak{n}(\gamma_s)_{2, \mathbb{A}}\oplus \mathfrak{n}(\gamma_s)_{3, \mathbb{A}}. \]
Define \[\Psi_{\gamma, i}(\xi, p)=\int_{\mathfrak{n}(\gamma_s)_{i, \mathbb{A}}}\Phi_\gamma(f, p^{-1}\cdot \rm{exp}\ Y\cdot p)\rm{exp}(\xi(Y))dY, \quad p\in M(\gamma)_\mathbb{A}N(\gamma_s)_\mathbb{A}, \,\xi\in X(\gamma_s)_\mathbb{A}. \]
Applying the Poisson summation formula, we obtain 
\begin{eqnarray}\label{poisson}
    \sum_{\substack{v\in N(\gamma_s)_{i, \mathbb{Q}}}}\Phi_\gamma(f, p^{-1}vp)=\sum_{\substack{\xi\in X(\gamma_s)_{i, \mathbb{Q}}\\\xi\neq 0}}\Psi_\gamma(\xi, p)+\Psi_\gamma(0, p).
\end{eqnarray}
We now consider the integral of the first term on the right-hand side of (\ref{poisson}) after multipling $\hat{\tau}_P$, which is
\begin{eqnarray}
    \int_{M(\gamma)_\mathbb{Q}N(\gamma_s)_\mathbb{Q}\backslash M(\gamma)_\mathbb{A}^1N(\gamma_s)_\mathbb{A}}\int_{<T, \hat{\alpha}_{i_1}>}^{\infty}. . . \int_{<T, \hat{\alpha}_{i_j}>}^{\infty}\sum_{\substack{\xi\in X(\gamma_s)_{i, \mathbb{Q}}\\\xi\neq e}}|\Psi_{\gamma, i}(\xi, h_ap)|\\\notag
\cdot|\rm{exp}(-<2\rho_P(\gamma_s), \sum_{k=1}^j{(1+<\lambda, \alpha_{i_k}>)a_{k}\hat{\alpha}_{i_k}}>)|dp\ da_1, . . . da_j.
\end{eqnarray}
It is easy to verify that \[\Psi_{\gamma,i}(\xi, h_{a}p)=\rm{exp}(<2\rho_P(\gamma_s), a>)\Psi_{\gamma,i}(\xi^{-a}, p). \]
$\Psi_{\gamma,i}(\cdot,p)$ is the Fourier transform of Schwartz-Bruhat function and it is continuous in $p$. By Lemma \ref{l9}, we observe that there are only finitely many $\gamma\in M_{\mathbb{Q}}$, such that \[\Psi_{\gamma,i}(\xi,p)\neq0 .\] Thus, for any $N$, there exists a constant $\Gamma_N$ such that for any primitive $\xi\in X_\mathbb{A}$, \[\sum_{\gamma\in M^\mathfrak{o}}|\Psi_{\gamma,i}(\xi,p)|\le\Gamma_N||\xi||^{-N}. \]
Consequently, for any $N$, the above integral is bounded by 
\[\int_{M(\gamma)_\mathbb{Q}N(\gamma_s)_\mathbb{Q}\backslash M(\gamma)_\mathbb{A}^1N(\gamma_s)_\mathbb{A}}\int^{\infty}_{<T, \hat{\alpha}_{i_1}>}. . . \int^{\infty}_{<T, \hat{\alpha}_{i_1}>}\rm{exp}(<2\rho_P(\gamma_s), a>)(\sum_{\substack{\xi\in X_\mathbb{Q}\\\xi\neq 0}}||\xi^{-a}||^{-N})da, \]
and it is majorized by the product of 
\[\sum_{\substack{\xi\in X_\mathbb{Q}\\\xi\neq 0}}||\xi||^{-N}\]
and
\[\int_{M(\gamma)_\mathbb{Q}N(\gamma_s)_\mathbb{Q}\backslash M(\gamma)_\mathbb{A}^1N(\gamma_s)_\mathbb{A}}\int^{\infty}_{<T, \hat{\alpha}_{i_1}>}. . . \int^{\infty}_{<T, \hat{\alpha}_{i_1}>}\rm{exp}(<2\rho_P, \sum_{k=1}^ja_k\hat{\alpha}_{i_k}>)\rm{exp}(-<d, Na>)da. \]
For sufficiently large $N$, this term is finite and approaches $0$ as $T\rightarrow\infty$. Thus if we sum over $i$, the result is similar.

Now, sum over $i$, when $T$ is sufficiently large, we observe that the function $I_{\rm{ram}}^\mathfrak{o}(f, x, T)$ becomes
\begin{eqnarray*}
    \sum_{P}(-1)^{\rm{dim}\ Z\backslash A}\sum_{\gamma\in M^\mathfrak{o}}\sum_{\delta\in M_\mathbb{Q}N(\gamma_s)_\mathbb{Q}\backslash G_\mathbb{Q}}\int_{N(\gamma_s)_\mathbb{A}}f(x^{-1}\delta^{-1}\gamma n\delta x)(\hat{\tau}_P(H_0(\delta x)-T))dn. 
\end{eqnarray*}
This can be written as
\begin{eqnarray*}
    \sum_{P}(-1)^{\rm{dim}\ Z\backslash A}\sum_{\gamma\in M^\mathfrak{o}}\sum_{\delta\in P_\mathbb{Q}\backslash G_\mathbb{Q}}\sum_{\xi\in N(\gamma_s)_\mathbb{Q}\backslash N_\mathbb{Q}}\int_{N(\gamma_s)_\mathbb{A}}f(x^{-1}\delta^{-1}\xi^{-1}\gamma n\xi\delta x)(\hat{\tau}_P(H_0(\delta x)-T))dn. 
\end{eqnarray*}

The integral of $I_{\rm{ram}}^\mathfrak{o}(f, x, T)$ is
\begin{eqnarray*}
\int_{M(\gamma)_\mathbb{Q}N(\gamma_s)_\mathbb{Q}\backslash M(\gamma)^1_\mathbb{A}N(\gamma_s)_\mathbb{A}}\int_{<T, \hat{\alpha}_{i_1}>}^{\infty}. . . \int_{<T, \hat{\alpha}_{i_j}>}^{\infty}\sum_{\gamma\in M^\mathfrak{o}}\Psi_{\gamma}(0, h_ap)\\\cdot\rm{exp}(-<2\rho_P(\gamma_s), \sum_{k=1}^j(1+<\lambda, \alpha_{i_k}>)a_k\hat{\alpha}_{i_k})dp\ da_1 . . .  da_j, 
\end{eqnarray*}
it is absolutely convergent for $<\rm{Re}\ \lambda, \alpha_{i_k}>>0$, since it is zeta integral.

The integral is 
\begin{eqnarray*}
a_P\Pi_{k=1}^{j}\frac{\rm{exp}(-<2\rho_P(\gamma_s), <\lambda, \alpha_{i_k}><T, \hat{\alpha}_{i_k}>\hat{\alpha}_{i_k}>)}{-<2\rho_P(\gamma_s), <\lambda, \alpha_{i_k}>\hat{\alpha}_{i_k}>}\int_{M(\gamma)_\mathbb{Q}\backslash M(\gamma)^1_\mathbb{A}}\int_{N(\gamma_s)_{\mathbb{A}}}\sum_{\gamma\in M^\mathfrak{o}}\Phi_\gamma(f, n)dn\ dm, 
\end{eqnarray*}
which equals
\begin{eqnarray*}
    a_P\Pi_{k=1}^{j}\frac{\rm{exp}(-<2\rho_P(\gamma_s), <\lambda, T>\hat{\alpha}_{i_k}>)}{<2\rho_P(\gamma_s), <\lambda, \alpha_{i_k}>\hat{\alpha}_{i_k}>}\int_{M(\gamma)_\mathbb{Q}\backslash M(\gamma)^1_\mathbb{A}}\int_{N(\gamma_s)_{\mathbb{A}}}\sum_{\gamma\in M^\mathfrak{o}}\Phi_\gamma(f, n)dn\ dm. 
\end{eqnarray*}
We replace $\lambda$ with $\lambda\lambda_0$, $\lambda_0$ is any regular element which is not on any wall. Since it is a zeta integral, taking the constant term of the Laurent expansion at $\lambda=0$, we botain
\begin{equation}
\begin{aligned}\label{s5}
\frac{(-1)^j}{j!}a_P\Pi_{k=1}^j\frac{<\lambda_0, T>^j}{<\lambda_0, \alpha_{i_k}>}\int_{M(\gamma)_\mathbb{Q}\backslash M(\gamma)^1_\mathbb{A}}\int_{N(\gamma_s)_{\mathbb{A}}}\sum_{\gamma\in M^\mathfrak{o}}\Phi_\gamma(f, n)dn\ dm. 
\end{aligned}
\end{equation}
However, \[\int_{N(\gamma_s)_{ \mathbb{A}}}\Phi_\gamma(f, n)dn\] equals
\[c_P(n_{\gamma, M})^{-1}\int_K\int_{M(\gamma)_\mathbb{A}^1N(\gamma_s)_\mathbb{A}\backslash P_\mathbb{A}^1}\int_{N(\gamma_s)_{\mathbb{A}}}f(k^{-1}p^{-1}\gamma n pk)dn\ dp\ dk. \]
We now decompose $M(\gamma)_\mathbb{A}^1N(\gamma_s)_\mathbb{A}\backslash P_\mathbb{A}^1$ as the product \[N_\mathbb{A}M(\gamma)_\mathbb{A}^1N(\gamma_s)_\mathbb{A}\backslash P_\mathbb{A}^1 \times M(\gamma)_\mathbb{A}^1N(\gamma_s)_\mathbb{A}\backslash N_\mathbb{A}M(\gamma)_\mathbb{A}^1N(\gamma_s)_\mathbb{A}, \]
the integral becomes 
\[c_P(n_{\gamma, M})^{-1}\int_K\int_{M(\gamma)_\mathbb{A}^1\backslash M_\mathbb{A}^1}\int_{N_{\mathbb{A}}}\sum_{\gamma\in M^\mathfrak{o}}f(k^{-1}m^{-1}\gamma n mk)dn\ dm\ dk, \]
according to Lemma \ref{l2}.

Similarly, the term (\ref{s5}) becomes 
\begin{equation}
    \begin{aligned}\label{s6}
        &\frac{(-1)^j}{j!}a_P\Pi_{k=1}^j\frac{<\lambda_0, T>^j}{<\lambda_0, \alpha_{i_k}>}\cdot c_P\int_K\int_{M(\gamma)_\mathbb{Q}\backslash M_\mathbb{A}^1} \int_{N_{\mathbb{A}}}\sum_{\gamma\in M^\mathfrak{o}}f(k^{-1}m^{-1}\gamma nmk)dn\ dm\ dk.
    \end{aligned}
\end{equation}
Which is the product of 
\[\frac{(-1)^{j}}{j!}a_P\Pi_{k=1}^j\frac{<\lambda_0, T>^j}{<\lambda_0, \alpha_{i_k}>}\cdot c_P\] and
\[\int_K\int_{A_\infty^+M(\gamma)_\mathbb{Q}\backslash M_\mathbb{A}} \int_{N_{\mathbb{A}}}\sum_{\gamma\in M^\mathfrak{o}}f(k^{-1}m^{-1}\gamma nmk)\rm{exp}(-<2\rho_P, H_0(m)>)dn\ dm\ dk. \]
For $<\rm{Re}\ \lambda, \alpha_{i}>>0,i=1,2,3$, to get the result, we consider the function 
\begin{eqnarray}\label{s7}
&\int_{Z_\infty^+G_\mathbb{Q}\backslash G_\mathbb{A}}I_{\rm{ram}}^{\mathfrak{o}}(f, x)dx\\\notag
&+\sum_{P\neq G}(-1)^{\rm{dim}(Z\backslash A)}\int_{<T, \hat{\alpha}_{i_1}>}^\infty. . . \int_{<T, \hat{\alpha}_{i_k}>}^\infty\int_{P(\gamma)_\mathbb{Q}\backslash P(\gamma_s)_\mathbb{A}^1}\sum_{v\in N(\gamma_s)_\mathbb{Q}}\\\notag
&\sum_{\gamma\in M^\mathfrak{o}}\Phi_\gamma(f, h^{-1}_{a}vph_{a})\rm{exp}(-<2\rho_P(\gamma_s), (1+<\lambda, \alpha_{k}>)a_k\hat{\alpha}_{i_k}>)dp\ da_1\ . . . \ da_j. 
\end{eqnarray}
For every term in (\ref{s7}) except the first, we can use our statement above to obtain the result similar to the term (\ref{s6}).

For any ramified orbit $\mathfrak{o}$, we define $\mu_\mathfrak{o}(\lambda,f,x)$ to be
\begin{eqnarray}\label{ramified term}
    \sum_{\gamma\in G^\mathfrak{o}}f(x^{-1}\gamma x)\rm{exp}(-<2\rho_{P_{0}}(\gamma_s),\sum_{k=1}^3(1+<\lambda,\alpha_{k}>)a_{k}\hat{\alpha}_{k}>).
\end{eqnarray}
We have shown that \[\int_{Z_\infty^+G_\mathbb{Q}\backslash G_\mathbb{A}}I_{\rm{ram}}^\mathfrak{o}(f,x,T) dx\] is absolutely convergent, and it equals \[\rm{lim}_{\lambda\rightarrow 0}I^\mathfrak{o}_T(\lambda),\] it indicates that the poles at $\lambda=0$ of each term in the sum over $P$ of $I_{\rm{ram}}^\mathfrak{o}(f,x,T)$ can be canceled.

Then, by the fact that zeta integral can be analytically continued to a meromorphic function, and $\lambda=(0,0,0)$ is the pole of this funtion.
\begin{lemma}\label{8.1}
    The integral of $I^\mathfrak{o}_{\rm{ram}}(f, x, T)$ over $Z_\infty^+G_\mathbb{Q}\backslash G_\mathbb{A}$ equals the sum
    \begin{eqnarray}\label{85}
    &\rm{lim}_{\lambda\rightarrow 0}\int_{Z_\infty^+G_\mathbb{Q}\backslash G_\mathbb{A}}\mu_\mathfrak{o}(\lambda,f,x)dx
    \end{eqnarray}
    \begin{eqnarray}\label{s8}
    &+\frac{1}{t!}\sum_{P_1\neq P}a_{P_1}\Pi_{k=1}^t\frac{<\lambda^t_0, T>^j}{<\lambda_0^t, \alpha_{i_k}>}\cdot c_{P_1}\sum_{\gamma\in\{M_{1, n}^\mathfrak{o}\}}(n_{\gamma, M})^{-1}
    \end{eqnarray}
    \begin{eqnarray*}\label{86}
    &\int_K\int_{A_{1, \infty}^+M(\gamma)_{1, \mathbb{Q}}\backslash M_{1, \mathbb{A}}} \int_{N(\gamma_s)_{1, \mathbb{A}}}f(k^{-1}m^{-1}\gamma nmk)\rm{exp}(-<2\rho_{P_1}, H_0(m)>)dn\ dm\ dk
    \end{eqnarray*}
    \begin{eqnarray}\label{s9}
    &+\frac{1}{t!}\sum_{P_1\neq P}a_{P_1}\sum_{\gamma\in\{M_{1, t}^\mathfrak{o}\}}\tilde{\tau}(\gamma,M)\int_K\int_{M(\gamma)_{1, \mathbb{A}}\backslash M_{1, \mathbb{A}}}\int_{N_{1, \mathbb{A}}}
    \end{eqnarray}
    \begin{eqnarray*}
    &f(k^{-1}n^{-1}m^{-1}\gamma mnk)\Pi_{k=1}^t\frac{<\lambda^t_0, T-H_0(m)>^j}{<\lambda_0, \alpha_{i_k}>}dn\ dm\ dk
    \end{eqnarray*}
    \begin{eqnarray}
    &+\frac{1}{j!}a_P\Pi_{k=1}^j\frac{<\lambda_0^j, T>^j}{<\lambda^j_0, \alpha_{i_k}>}\cdot c_P\sum_{\gamma\in\{M_n^\mathfrak{o}\}}(n_{\gamma, M})^{-1}\int_K\int_{A_\infty^+M(\gamma)_\mathbb{Q}\backslash M_\mathbb{A}} \int_{N(\gamma_s)_\mathbb{A}}\\\notag
    &f(k^{-1}m^{-1}\gamma nmk)\rm{exp}(-<2\rho_P, H_0(m)>)dn\ dm\ dk. 
    \end{eqnarray}
    Where $t$ denotes the numer of simple roots of $P_1$, and $\lambda_0^t=(\lambda_1,...,\lambda_t)$. 
\end{lemma}
\begin{proof}
    According to above discussion, it suffices to prove only (\ref{s8}) and (\ref{s9}), which follows from Lemma (\ref{l6}).

    When $N(\gamma_s)$ is trivial, the integral of \[\sum_{\gamma\in \{M_t^\mathfrak{o}\}}(n_{\gamma,M})^{-1}\sum_{\delta\in M(\gamma)_\mathbb{Q}\backslash G_\mathbb{Q}}f(x^{-1}\delta^{-1}\gamma \delta x)\hat{\tau}_P(H_0(\delta x)-T)\] over $Z_\infty^+G_\mathbb{Q}\backslash G_\mathbb{A}$ is 
    \begin{eqnarray*}
        c_P\sum_{\gamma\in\{M_t^\mathfrak{o}\}}(n_{\gamma,M})^{-1}\int_K\int_{A_\infty^+M(\gamma)_{\mathbb{Q}}\backslash P_\mathbb{A}}\int_{Z_\infty^+\backslash A_\infty^+}f(k^{-1}p^{-1}\gamma pk)\hat{\tau}_P(H_0(apk)-T)da\ d_lp\ dk.
    \end{eqnarray*}
    We write it as 
    \begin{eqnarray*}
        c_Pa_P\sum_{\gamma\in\{M_t^\mathfrak{o}\}}\tilde{\tau}(\gamma,M)\int_K\int_{M(\gamma)^1_\mathbb{A}\backslash P^1_\mathbb{A}}f(k^{-1}p^{-1}\gamma pk)\\
        \int_{<T-H_0(p),\hat{\alpha}_{i_1}>}^\infty...\int^\infty_{<T-H_0(p),\hat{\alpha}_{i_k}>}\rm{exp}(-<2\rho_{P}(\gamma_s),\sum_{k=1}^j(1+<\lambda,\alpha_k>)a_k\hat{\alpha}_{i_k}+H_0(p)>)da\ dp\ dk.
    \end{eqnarray*}
\end{proof}
However, the first term can be calculated by taking the constant term of the Laurent expansion at $\lambda=(0,0,0)$, we write it as 
\[\rm{lim}_{\lambda\rightarrow 0}\int_{Z_\infty^+G_\mathbb{Q}\backslash G_\mathbb{A}}D_\lambda\{\lambda\mu_\mathfrak{o}(\lambda, f,x)\}dx, \]
where \[D_\lambda\{\lambda\mu_\mathfrak{o}(\lambda, f,x)\}\] equals
\[\frac{d}{d<\lambda,\alpha_{i_j}>}(<\lambda,\alpha_{i_j}>)...\frac{d}{d<\lambda,\alpha_{i_1}>}(<\lambda,\alpha_{i_1}>\mu_\mathfrak{o}(\lambda,f,x)).\]
\begin{remark}
    Arthur(\cite{A15}) proposed a method to approximate the ramified orbits using the unramified orbits. In our next work, we shall take that way to rewrite (\ref{85}).
\end{remark}
\subsection{The unramified orbit}
In this section, we give the formula of $v_{\mathfrak{o}}(x, T)$, where the orbit $\mathfrak{o}$ is unramified.

We define
\begin{eqnarray}\label{s11}
v_\mathfrak{o}(x, T)=\int_{Z_\infty^+\backslash A_{\mathfrak{o}, \infty}^+}\{\sum_P(-1)^{\rm{dim}\ (Z\backslash A)}\sum_{s\in\Omega(\mathfrak{a}_\mathfrak{o};P)}\hat{\tau}_P(H_{0}(w_sax)-T)da\},
\end{eqnarray}
(see \cite{A3}). 
And recall that $\Omega(\mathfrak{a}_\mathfrak{o};P)$ is the set of elements $s$ in $\cup_{P_1}\Omega(\mathfrak{a}_\mathfrak{o}, \mathfrak{a}_1)$ such that if $\mathfrak{a}_1=s\mathfrak{a}_\mathfrak{o}, \mathfrak{a}_1$ contains $\mathfrak{a}$, and $s^{-1}\alpha$ is positive for every root $\alpha\in\Phi_{P_1}^{P}$.

In fact, $v_\mathfrak{o}(x, T)$ stands for the volume in $\mathfrak{a}_G\backslash \mathfrak{a}_\mathfrak{o}$ of the convex hull of the projection onto $\mathfrak{a}_G\backslash \mathfrak{a}_\mathfrak{o}$ of $\{s^{-1}T-s^{-1}H_0(w_sx);s\in\cup_{P_1}\Omega(\mathfrak{a}_\mathfrak{o}, \mathfrak{a}_1)\}$. It was Langlands who surmised that the volume of a convex hull would play a role in the trace formula.

In \cite{A2}, Arthur gives the following identity
\[v_\mathfrak{o}(x, T)=\sum_{P\in P(A_\mathfrak{o})}\frac{\rm{exp}(<\lambda, T_P-H_P(x)>)}{\Pi_{\eta\in \Phi_P}<\lambda, \eta>}, \]where $P(A_\mathfrak{o})$ is the set of parabolic subgroups which are not necessary standard such that their split component is $A_\mathfrak{o}$. Here we use the property \[s^{-1}H_0(w_sx)=H_{w_s^{-1}P_0w_s}(x)=H_P(x). \]
We observe that this formula replaces the sum over $s$ and $P\in\mathfrak{P}$ with the sum of $P\in P(A)$ in (\ref{s11}).

Then in \cite{A6} Arthur introduces the concept of a $(G, M)$-family associated to the convex hull. He also writes $v_{\mathfrak{o}}(x, T)$ as $v_M(x, T)$. The Haar measure of $\mathfrak{a}_G\backslash \mathfrak{a}$ is defined via the height function in \cite{A2}.
\begin{lemma}\label{l82}
    For any ramified orbit $\mathfrak{o}$, if the parabolic subgroup $P$ contains $P_1$, where $P_1\in \mathfrak{P}_{\{\mathfrak{o}\}}$, then the term 
    \[\sum_{\delta\in M_\mathbb{Q}\backslash G_\mathbb{Q}}\sum_{\gamma\in M_t^\mathfrak{o}}f(x^{-1}\delta^{-1}\gamma\delta x)\] equals
    \begin{eqnarray*}
        \sum_{\delta\in M_{1,\mathbb{Q}}\backslash G_\mathbb{Q}}\sum_{\gamma\in M_{t,1}^\mathfrak{o}}f(x^{-1}\delta^{-1}\gamma\delta x).
    \end{eqnarray*}
\end{lemma}
\begin{proof}
    For $\gamma_1,\gamma_2\in M_{t,1}^\mathfrak{o}$, if there exists $g\in G_\mathbb{Q}$, such that \[g^{-1}\gamma_1g=\gamma_2.\]We aim to show \[g\in M_{1,\mathbb{Q}}.\]

    Since there exists $m_1,m_2\in M_{1,\mathbb{Q}}$ such that \[m_1^{-1}\gamma_1m_1=J=m^{-1}_2\gamma_2m_2,\] where $J$ is the Jordan normal form of $\gamma_1,\gamma_2$.

    Then \[m_2m_1^{-1}\gamma_1m_1m_2^{-1}=\gamma_2=g^{-1}\gamma_1g,\] Thus \[m_1m_2^{-1}g^{-1}\in G(\gamma_1)\subset G(\gamma_{1,s})\subset M_{1},\]
    which implies \[g\in M_{1,\mathbb{Q}}.\]
    We conclude the result.
\end{proof}
By this lemma, we can combine some terms in (\ref{85}) and (\ref{s9}) together to form a convex hull $v_{M_{\{\mathfrak{o}\}}}(x,T)$.

\section{Terms associated to $P_{31}$}
\[\Omega(\mathfrak{a}_{31}, \mathfrak{a}_{31})=\{1\}, \]
\[\Omega(\mathfrak{a}_{31}, \mathfrak{a}_{13})=\{(14)\}. \]
\subsection{The first parabolic term}
The first parabolic term is \[J_{\rm{unram}}^{\mathfrak{o}^0_{31}}(f, x, T)-K'_{P_{31}}(f, x, T)-K'_{P_{13}}(f, x, T). \]
In this section, we shall prove that the first parabolic term equals $0$ after we borrow some terms from other first parabolic term.

Recall $J_{\rm{unram}}^{\mathfrak{o}^0_{31}}(f, x, T)$ is defined by
\begin{eqnarray*}
\sum_{\gamma\in \{M_{t, 31}^{\mathfrak{o}^0_{31}}\}}(n_{\gamma, M_{31}})^{-1}\sum_{\delta\in M(\gamma)_{31, \mathbb{Q}}\backslash G_\mathbb{Q}}f(x^{-1}\delta^{-1}\gamma\delta x)(\hat{\tau}_{P_{31}}(H_0(\delta x)- T)+\hat{\tau}_{P_{13}}(H_0(w_s\delta x)- T)).  
\end{eqnarray*}
Then,
\begin{eqnarray*}
&\sum_{\gamma\in \{M_{t, 31}^{\mathfrak{o}^0_{31}}\}}(n_{\gamma, M_{31}})^{-1}\sum_{\delta\in M(\gamma)_{31, \mathbb{Q}}\backslash G_\mathbb{Q}}f(x^{-1}\delta^{-1}\gamma\delta x)(\hat{\tau}_{P_{13}}(H_0(w_s\delta x)- T))\\
&=\sum_{\gamma\in \{M_{t, 13}^{\mathfrak{o}^0_{31}}\}}(n_{\gamma, M_{13}})^{-1}\sum_{\delta\in M(w_s\gamma w_s^{-1})_{31, \mathbb{Q}}\backslash G_\mathbb{Q}}f(x^{-1}\delta^{-1}w_s\gamma w_s^{-1}\delta x)(\hat{\tau}_{P_{13}}(H_0(w_s\delta x)- T))\\
&=\sum_{\gamma\in \{M^{\mathfrak{o}^0_{31}}_{t, 13}\}}(n_{\gamma, M_{13}})^{-1}\sum_{\delta\in M(\gamma)_{13, \mathbb{Q}}\backslash G_\mathbb{Q}}f(x^{-1}\delta^{-1}\gamma \delta x)(\hat{\tau}_{P_{13}}(H_0(\delta x)- T)). 
\end{eqnarray*}
Thus, 
\begin{eqnarray*}
&J_{\rm{unram}}^{\mathfrak{o}^0_{31}}(f, x, T)\\
&=\sum_{\gamma\in \{M_{t, 31}^{\mathfrak{o}^0_{31}}\}}(n_{\gamma, M_{31}})^{-1}\sum_{\delta\in M(\gamma)_{31, \mathbb{Q}}\backslash G_\mathbb{Q}}f(x^{-1}\delta^{-1}\gamma\delta x)(\hat{\tau}_{P_{31}}(H_0(\delta x))-T)\\
&+\sum_{\gamma\in \{M_{t, 13}^{\mathfrak{o}^0_{31}}\}}(n_{\gamma, M_{13}})^{-1}\sum_{\delta\in M(\gamma)_{13, \mathbb{Q}}\backslash G_\mathbb{Q}}f(x^{-1}\delta^{-1}\gamma\delta x)(\hat{\tau}_{P_{13}}(H_0(\delta x)- T)). 
\end{eqnarray*}
Since for $\gamma\in{M_{t, 31}^{\mathfrak{o}^0_{31}}}$, the group $N(\gamma_s)$ is trivial, by Lemma \ref{l1}, $J_{\rm{unram}}^{\mathfrak{o}^0_{31}}(f, x, T)$ equals
\begin{eqnarray*}
&\sum_{\gamma\in\{M_{t, 31}^{\mathfrak{o}^0_{31}}\}}(n_{\gamma, M_{31}})^{-1}\sum_{\delta\in M(\gamma)_{31, \mathbb{Q}}N_{31, \mathbb{Q}}\backslash G_\mathbb{Q}}\sum_{v\in N_{31, \mathbb{Q}}}f(x^{-1}\delta^{-1}\gamma v\delta x)(\hat{\tau}_{P_{31}}(H_0(\delta x)-T))\\
&+\sum_{\gamma\in\{M_{t, 13}^{\mathfrak{o}^0_{31}}\}}(n_{\gamma, M_{13}})^{-1}\sum_{\delta\in M(\gamma)_{13, \mathbb{Q}}N_{13, \mathbb{Q}}\backslash G_\mathbb{Q}}\sum_{v\in N_{13, \mathbb{Q}}}f(x^{-1}\delta^{-1}\gamma v\delta x)(\hat{\tau}_{P_{13}}(H_0(\delta x)-T)). 
\end{eqnarray*}
Which is 
\begin{eqnarray*}
&\sum_{\delta\in P_{31, \mathbb{Q}}\backslash G_\mathbb{Q}}\sum_{\gamma\in M_{t, 31}^{\mathfrak{o}^0_{31}}}\sum_{v\in N_{31, \mathbb{Q}}}f(x^{-1}\delta^{-1}\gamma v\delta x)(\hat{\tau}_{P_{31}}(H_{0}(\delta x)-T))\\
&+\sum_{\delta\in P_{13, \mathbb{Q}}\backslash G_\mathbb{Q}}\sum_{\gamma\in M_{t, 13}^{\mathfrak{o}^0_{31}}}\sum_{v\in N_{13, \mathbb{Q}}}f(x^{-1}\delta^{-1}\gamma v\delta x)(\hat{\tau}_{P_{13}}(H_{0}(\delta x)-T)). 
\end{eqnarray*}
We now compute the geometric terms corresponding to the orbit $\mathfrak{o}\neq\mathfrak{o}^0_{31}$ in $P_{31}$.

For other orbits, we now choose the terms from $J_{\rm{ram}}^\mathfrak{o}(f,x,T)$ and $J_{\rm{unram}}^\mathfrak{o}(f,x,T)$ whose characteristic functions are $\hat{\tau}_{P_{31}}$ and we combine them.

By Lemma \ref{l6.6}, we define $J_{P_{31}}(f, x, T)=$
\begin{eqnarray}
&J_{\rm{unram}}^{\mathfrak{o}^0_{31}}(f, x, T)\\\notag
&+\sum_{\substack{\mathfrak{o}\ unramified\\\mathfrak{o}\neq\mathfrak{o}_{31}}}\frac{1}{|\Omega(\mathfrak{a}_\mathfrak{o}, P_{31})|}\sum_{s\in\Omega(\mathfrak{a}_\mathfrak{o}, P_{31})}\sum_{\delta\in P_{31, \mathbb{Q}}\backslash G_\mathbb{Q}}\sum_{\gamma\in M_{t, 31}^{\mathfrak{o}}}\sum_{v\in N_{31, \mathbb{Q}}}f(x^{-1}\delta^{-1}\gamma v\delta x)\\\notag
&\cdot(\hat{\tau}_{P_{31}}(H_{0}(w_s\delta x))-T)
\end{eqnarray}
\begin{eqnarray*}
&+\sum_{\substack{\mathfrak{o}\ unramified\\\mathfrak{o}\neq\mathfrak{o}_{31}}}\frac{1}{|\Omega(\mathfrak{a}_\mathfrak{o}, P_{13})|}\sum_{s\in\Omega(\mathfrak{a}_\mathfrak{o}, P_{13})}\sum_{\delta\in P_{13, \mathbb{Q}}\backslash G_\mathbb{Q}}\sum_{\gamma\in M_{t, 13}^{\mathfrak{o}}}\sum_{v\in N_{13, \mathbb{Q}}}f(x^{-1}\delta^{-1}\gamma v\delta x)\\
&\cdot(\hat{\tau}_{P_{13}}(H_{0}(w_s\delta x))-T)
\end{eqnarray*}
\begin{eqnarray*}
&+\sum_{\mathfrak{o}\ ramified}\sum_{\delta\in M_{31, \mathbb{Q}}N(\gamma_s)_{31, \mathbb{Q}}\backslash G_\mathbb{Q}}\sum_{\gamma\in M_{31}^\mathfrak{o}}\sum_{v\in N(\gamma_s)_{31, \mathbb{Q}}}f(x^{-1}\delta^{-1}\gamma v\delta x)\\
&\cdot(\hat{\tau}_{P_{31}}(H_{0}(\delta x))-T)
\end{eqnarray*}
\begin{eqnarray*}
&+\sum_{\mathfrak{o}\ ramified}\sum_{\delta\in M_{13, \mathbb{Q}}N(\gamma_s)_{13, \mathbb{Q}}\backslash G_\mathbb{Q}}\sum_{\gamma\in M_{13}^\mathfrak{o}}\sum_{v\in N(\gamma_s)_{13, \mathbb{Q}}}f(x^{-1}\delta^{-1}\gamma v\delta x)\\
&\cdot(\hat{\tau}_{P_{13}}(H_{0}(\delta x))-T). 
\end{eqnarray*}

Observe that when fixing an unramified orbit $\mathfrak{o}$, for $s_1\in\Omega(\mathfrak{a}_\mathfrak{o}, P_{31})$, 
\begin{eqnarray*}
    \sum_{\delta\in P_{31, \mathbb{Q}}\backslash G_\mathbb{Q}}\sum_{\gamma\in M_{t,31}^{\mathfrak{o}}}\sum_{v\in N_{31, \mathbb{Q}}}f(x^{-1}\delta^{-1}\gamma v\delta x)(\hat{\tau}_{P_{31}}(H_{0}(w_s\delta x))-T)
\end{eqnarray*}
equals
\begin{eqnarray*}
    \sum_{\delta\in P_{31, \mathbb{Q}}\backslash G_\mathbb{Q}}\sum_{\gamma\in M_{t,31}^{\mathfrak{o}}}\sum_{v\in N_{31, \mathbb{Q}}}f(x^{-1}\delta^{-1}\gamma v\delta x)(\hat{\tau}_{P_{31}}(H_{0}(\delta x))-T).
\end{eqnarray*}

For ramified orbits in $P_{31}$, the terms can be written as
\[\sum_{\delta\in P_{31, \mathbb{Q}}\backslash G_\mathbb{Q}}\sum_{\gamma\in M_{31}^\mathfrak{o}}\sum_{\delta_1\in N(\gamma_s)_{31, \mathbb{Q}}\backslash N_{31, \mathbb{Q}}}\sum_{v\in N(\gamma_s)_{31, \mathbb{Q}}}f(x^{-1}\delta^{-1}\delta_1^{-1}\gamma v\delta_1\delta x)(\hat{\tau}_{P_{31}}(H_{0}(\delta x))-T). \]
By Lemma \ref{l1}, it equals
\[\sum_{\delta\in P_{31, \mathbb{Q}}\backslash G_\mathbb{Q}}\sum_{\gamma\in M_{31}^\mathfrak{o}}\sum_{v\in N_{31, \mathbb{Q}}}f(x^{-1}\delta^{-1}\gamma v\delta x)(\hat{\tau}_{P_{31}}(H_{0}(\delta x))-T). \]
Since $M_{31}=\cup_{\mathfrak{o}} M_{31}^\mathfrak{o}$ and $M_{13}=\cup_{\mathfrak{o}} M_{13}^\mathfrak{o}$, we can see $J_{P_{31}}(f, x, T)$ equals
\begin{eqnarray*}
&\sum_{\delta\in P_{31, \mathbb{Q}}\backslash G_\mathbb{Q}}\sum_{\gamma\in M_{31}}\sum_{v\in N_{31, \mathbb{Q}}}f(x^{-1}\delta^{-1}\gamma v\delta x)(\hat{\tau}_{P_{31}}(H_{0}(\delta x)-T))\\
&+\sum_{\delta\in P_{13, \mathbb{Q}}\backslash G_\mathbb{Q}}\sum_{\gamma\in M_{13}}\sum_{v\in N_{13, \mathbb{Q}}}f(x^{-1}\delta^{-1}\gamma v\delta x)(\hat{\tau}_{P_{13}}(H_{0}(\delta x)-T)). 
\end{eqnarray*}
For any parabolic subgroup $P$, let the rank of $A_P$ be $j$.

The space $\mathfrak{n}_\mathbb{A}$ is a locally compact abelian group, and $\mathfrak{n}_\mathbb{Q}$ is a discrete group of it. Let $X_\mathbb{A}$ be the unitary dual group of $\mathfrak{n}_\mathbb{A},$ and let $ X_\mathbb{Q}$ be the subgroup of $X_\mathbb{A}$ which is trivial on $\mathfrak{n}_\mathbb{Q}$.

We now apply the Poisson summation formula omitting the decomposition $N=N_1\oplus N_2\oplus N_3$, 
\[\sum_{v\in N_\mathbb{Q}}f(x^{-1}\delta^{-1}\gamma v\delta x)\hat{\tau}_P(H_0(\delta x)-T)\]equals 
\begin{eqnarray}\label{e1}
\Psi(0, \gamma, \delta x)\hat{\tau}_P(H_0(\delta x)-T)
\end{eqnarray}
\begin{eqnarray}\label{e2}
    +\sum_{\substack{\xi\in X_\mathbb{Q}\\\xi\neq 0}}\Psi(\xi, \gamma, \delta x)\hat{\tau}_P(H_0(\delta x)-T),
\end{eqnarray}
where \[\Psi(\xi,\gamma,y)=\int_{\mathfrak{n}_\mathbb{A}}f(y^{-1}\cdot \gamma\rm{exp}\, Y\cdot y)\rm{exp}(\xi(Y))dY.\]

Summing the absolute value of (\ref{e1}) over $\gamma\in M^\mathfrak{o}$ and $\delta\in P_\mathbb{Q}\backslash G_\mathbb{Q}$, then integrating over $Z_\infty^+G_\mathbb{Q}\backslash G_\mathbb{A}$, then it is bounded by
\begin{equation}
\begin{aligned}
\int_{Z_\infty^+G_\mathbb{Q}\backslash G_\mathbb{A}}\sum_{\delta\in P_\mathbb{Q}\backslash G_\mathbb{Q}}\sum_{\gamma\in M^\mathfrak{o}}\sum_{\substack{\xi\in X_\mathbb{Q}\\\xi\neq 0}}|\Psi(\xi, \gamma, \delta x)\hat{\tau}_P(H_0(\delta x)-T)|dx,
\notag
\end{aligned}
\end{equation}
which is \[\int_{Z_\infty^+ P_\mathbb{Q}\backslash G_\mathbb{A}}\sum_{\gamma\in M^\mathfrak{o}}\sum_{\substack{\xi\in X_\mathbb{Q}\\\xi\neq 0}}|\Psi(\xi, \gamma, x)\hat{\tau}_P(H_0(x)-T)|dx. \]
If $\omega$ is a relatively compact fundamental domain for $P_\mathbb{Q}\backslash P_\mathbb{A}^1$ in $P_\mathbb{A}^1$, this integral equals
\begin{eqnarray}\label{omega}
\\\notag
c_P\int_K\int^{\infty}_{<T, \hat{\alpha}_{i_1}>}. . . \int^{\infty}_{<T, \hat{\alpha}_{i_j}>}\int_{\omega}\sum_{\gamma\in M^\mathfrak{o}}\sum_{\substack{\xi\in X_\mathbb{Q}\\\xi\neq 0}}|\Psi(\xi, \gamma, vh_{a}k)|\rm{exp}(-<2\rho_P, a>)\ dv\ dk\ da_1. . . da_k. 
\end{eqnarray}
We assume that $h_a^{-1}\omega h_a$ is contained in $\omega$ for every $a\in\mathfrak{a}^+$.

Then the integral (\ref{omega}) is bounded by \[c_P\int_{K}\int_\omega\int^{\infty}_{<T, \hat{\alpha}_{i_1}>}. . . \int^{\infty}_{<T, \hat{\alpha}_{i_j}>}\sum_{\gamma\in M^\mathfrak{o}}\sum_{\substack{\xi\in X_\mathbb{Q}\\\xi\neq 0}}|\Psi(\xi, \gamma, vh_{a}k)|\ dv\ dk\ da_1. . . da_j, \]
we denote \[vh_{a}k\] by \[h_{a}\cdot h^{-1}_{a}vh_{a}k\] in this term.

It is easy to verify that \[\Psi(\xi, \gamma, vh_{a}k)=\rm{exp}(<2\rho_P, \sum_{k=1}^ja_k\hat{\alpha}_{i_k}>)\Psi(\xi^{a}, \gamma, vk). \]
$\Psi(\cdot, \gamma, vk)$ is the Fourier transform of Schwartz-Bruhat function and it is continuous in $vk$. By Lemma \ref{l9}, we observe that there are only finitely many $\gamma\in M_{\mathbb{Q}}$, such that \[\Psi(\xi, \gamma, vk)\neq0 \]for some $\xi\in X_\mathbb{A}$ and some $vk\in \omega\times K$. Thus, for any $N$, there exists a constant $\Gamma_N$ such that for any primitive $\xi\in X_\mathbb{A}$, \[\sum_{\gamma\in M}|\Psi(\xi, \gamma, vk)|\le\Gamma_N||\xi||^{-N}. \]
Consequently, for any $N$, the above integral is bounded by 
\[c_P\Gamma_N\tau(M_{\mathbb{Q}})\int^{\infty}_{<T, \hat{\alpha}_{i_1}>}. . . \int^{\infty}_{<T, \hat{\alpha}_{i_1}>}\rm{exp}(<2\rho_P, \sum_{k=1}^ja_k\hat{\alpha}_{i_k}>)(\sum_{\substack{\xi\in X_\mathbb{Q}\\\xi\neq 0}}||\xi^{-a}||^{-N})da, \]
and it is majorized by 
\[c_P\Gamma_N\tau(M)\int^{\infty}_{<T, \hat{\alpha}_{i_1}>}. . . \int^{\infty}_{<T, \hat{\alpha}_{i_1}>}\rm{exp}(<2\rho_P, \sum_{k=1}^ja_k\hat{\alpha}_{i_k}>)\rm{exp}(-<d, Na>)da\cdot\sum_{\substack{\xi\in X_\mathbb{Q}\\\xi\neq 0}}||\xi||^{-N}. \]
For sufficiently large $N$, this term is finite and approaches $0$ as $T\rightarrow\infty$.

We now deal with (\ref{e1}). Summing over $\gamma\in M^\mathfrak{o}_\mathbb{Q},\, \delta\in P_\mathbb{Q}\backslash G_\mathbb{Q}$, we obtain
\begin{equation}\label{9.6}
\begin{aligned}
\sum_{\gamma\in M^\mathfrak{o}_{\mathbb{Q}}}\sum_{\delta\in P_\mathbb{Q}\backslash G_\mathbb{Q}}\Psi(0, \gamma, \delta x)\hat{\tau}_P(H_0(\delta x)-T). 
\end{aligned}
\end{equation}
For fixed $x$, there are only finitely many $\delta\in P_\mathbb{Q}\backslash G_\mathbb{Q}$ in (\ref{9.6}) such that this term is not equal to zero. Therefore, the inner sum is finite. The outer sum is also finite by the same argument.

As a result, the term $J_{P_{31}}(f, x, T)$ equals
\begin{equation}\label{e3}
\begin{aligned}
\sum_{\gamma\in M_{31,\mathbb{Q}}}\sum_{\delta\in P_{31,\mathbb{Q}}\backslash G_\mathbb{Q}}\Psi(0, \gamma, \delta x)\hat{\tau}_P(H_0(\delta x)-T). 
\end{aligned}
\end{equation}

Thus, $J_{P_{31}}(f, x, T)$ is 
\begin{eqnarray*}
    &\sum_{\delta\in P_{31, \mathbb{Q}}\backslash G_\mathbb{Q}}\sum_{\gamma\in M_{31}}\int_{N_{31, \mathbb{A}}}f(x^{-1}\delta^{-1}\gamma v\delta x)dn(\hat{\tau}_{P_{31}}(H_{0}(\delta x)-T))\\
    &+\sum_{\delta\in P_{13, \mathbb{Q}}\backslash G_\mathbb{Q}}\sum_{\gamma\in M_{13}}\int_{N_{13, \mathbb{A}}}f(x^{-1}\delta^{-1}\gamma v\delta x)dn(\hat{\tau}_{P_{13}}(H_{0}(\delta x)-T)). 
    \end{eqnarray*}

Recall $K'_{P_{31}}(f, x, T)+K'_{P_{13}}(f, x, T)$ equals 
\begin{eqnarray*}
    &\frac{1}{4\pi i}\sum_{P_{31, \mathbb{Q}}\backslash G_\mathbb{Q}}\sum_{\chi}\int_{i\mathfrak{a}_{G}\backslash i\mathfrak{a}_{31}}\{\sum_{\beta\in \mathscr{B}_{P, \chi}}E^{c_{31}}_{P_{31}}(\pi_{P_{31}}(\lambda, f)\Phi_\beta, \lambda, \delta x)\overline{E^{c_{31}}_{P_{31}}(\Phi_\beta, \lambda, \delta x)}\}d\lambda\hat{\tau}_{P_{31}}(H_{0}(\delta x)-T)\\
    &+\frac{1}{4\pi i}\sum_{P_{13, \mathbb{Q}}\backslash G_\mathbb{Q}}\sum_{\chi}\int_{i\mathfrak{a}_{G}\backslash i\mathfrak{a}_{13}}\{\sum_{\beta\in \mathscr{B}_{P, \chi}}E^{c_{13}}_{P_{13}}(\pi_{P_{13}}(\lambda, f)\Phi_\beta, \lambda, \delta x)\overline{E^{c_{13}}_{P_{13}}(\Phi_\beta, \lambda, \delta x)}\}d\lambda\hat{\tau}_{P_{13}}(H_{0}(\delta x)- T)\\
    &+\frac{1}{4\pi i}\sum_{P_{31, \mathbb{Q}}\backslash G_\mathbb{Q}}\sum_{\chi}\int_{i\mathfrak{a}_{G}\backslash i\mathfrak{a}_{31}}\{\sum_{\beta\in \mathscr{B}_{P, \chi}}E^{c_{13}}_{P_{31}}(\pi_{P_{31}}(\lambda, f)\Phi_\beta, \lambda, \delta x)\overline{E^{c_{13}}_{P_{31}}(\Phi_\beta, \lambda, \delta x)}\}d\lambda\hat{\tau}_{P_{31}}(H_{0}(\delta x)-T)\\
    &+\frac{1}{4\pi i}\sum_{P_{13, \mathbb{Q}}\backslash G_\mathbb{Q}}\sum_{\chi}\int_{i\mathfrak{a}_{G}\backslash i\mathfrak{a}_{13}}\{\sum_{\beta\in \mathscr{B}_{P, \chi}}E^{c_{31}}_{P_{13}}(\pi_{P_{13}}(\lambda, f)\Phi_\beta, \lambda, \delta x)\overline{E^{c_{31}}_{P_{13}}(\Phi_\beta, \lambda, \delta x)}\}d\lambda\hat{\tau}_{P_{13}}(H_{0}(\delta x)- T). 
\end{eqnarray*}
This term is equal to
\begin{eqnarray*}
    &\sum_{\gamma\in M_{31}}\sum_{\delta\in P_{31, \mathbb{Q}}\backslash G_\mathbb{Q}}\int_{N_{31, \mathbb{A}}}f(x^{-1}\delta^{-1}\gamma\delta x)dn\cdot (\hat{\tau}_{P_{31}}(H_{0}(\delta x)-T))\\
    &+\sum_{\gamma\in M_{13}}\sum_{\delta\in P_{13, \mathbb{Q}}\backslash G_\mathbb{Q}}\int_{N_{13, \mathbb{A}}}f(x^{-1}\delta^{-1}\gamma\delta x)dn\cdot (\hat{\tau}_{P_{13}}(H_{0}(\delta x)-T)). 
\end{eqnarray*} 
Thus we have 
\begin{lemma}
    $J_{P_{31}}(f,x,T)-K'_{P_{31}}(f, x, T)-K'_{P_{13}}(f, x, T)=0.$
\end{lemma}

\subsection{The second parabolic term}
In this section, we shall prove the convergence of the integral of the second parabolic term associated to $\mathfrak{o}_{31}^0$.

Recall \[I_{\rm{unram}}^{\mathfrak{o}^0_{31}}(f, x, T)=\sum_{\gamma\in \{M_{t,31}^{\mathfrak{o}_{31}}\}}(n_{\gamma, M_{31}})^{-1}\sum_{\delta\in M_{31}(\gamma)_{\mathbb{Q}}\backslash G_\mathbb{Q}}f(x^{-1}\delta^{-1}\gamma\delta x)(1-\hat{\tau}_{P_{31}}(H_0(\delta x)- T)-\hat{\tau}_{P_{13}}(H_0(\delta x)-T)). \]
The integral \[\int_{Z_\infty^+G_\mathbb{Q}\backslash G_\mathbb{A}}|I_{\rm{unram}}^{\mathfrak{o}^0_{31}}(f, x, T)|dx\] is bounded by \[\sum_{\gamma\in\{M_{t,31}^{\mathfrak{o}^0_{31}}\}}(n_{\gamma, M})^{-1}\int_{Z_\infty^+M(\gamma)_{31, \mathbb{Q}}\backslash G_\mathbb{A}}|f(x^{-1}\gamma x)|\cdot(1-\hat{\tau}_{P_{31}}(H_{0}(x)-T)-\hat{\tau}_{P_{13}}(H_{0}(w_sx)-T))dx. \]
It equals
\begin{eqnarray*}
c_{P_{31}}\sum_{\gamma\in\{M_{t,31}^{\mathfrak{o}^0_{31}}\}}(n_{\gamma, M})^{-1}\int_{K}\int_{A_{31, \infty}^+M(\gamma)_{31, \mathbb{Q}}\backslash P_{31, \mathbb{A}}}\int_{Z_\infty^+\backslash A_{31, \infty}^+}|f(k^{-1}p^{-1}\gamma pk)|\\
\cdot(1-\hat{\tau}_{P_{31}}(H_{0}(ap)-T)-\hat{\tau}_{P_{13}}(H_{0}(w_sap)-T))da\ d_lp\ dk. 
\end{eqnarray*}
Then the integral becomes
\begin{eqnarray*}
    c_{P_{31}}\sum_{\gamma\in\{M_{t,31}^{\mathfrak{o}^0_{31}}\}}\tilde{\tau}(\gamma, M)\int_{K}\int_{M(\gamma)_{31, \mathbb{A}}\backslash P_{31, \mathbb{A}}}|f(k^{-1}p^{-1}\gamma pk)|\\
    \cdot\int_{Z_\infty^+\backslash A_{31, \infty}^+}(1-\hat{\tau}_{P_{31}}(H_{0}(ap)-T)-\hat{\tau}_{P_{13}}(H_{0}(w_sap)-T))da\ d_lp\ dk. 
\end{eqnarray*}
We already know from Lemma \ref{l9} that the sum over $\gamma$ is finite. Since the function \[f^K(p)=\int_{K}f(k^{-1}pk)dk,\quad p\in P_{31, \mathbb{A}}\] has compact support, by Lemma \ref{l10}, the integral on $M(\gamma)_{31, \mathbb{A}}\backslash P_{31, \mathbb{A}}$ can be taken over a compact set. For any $p$ the function 
\[a\longrightarrow 1-\hat{\tau}_{P_{31}}(H_{0}(ap)-T)-\hat{\tau}_{P_{13}}(H_{0}(w_sap)-T), \quad a\in Z_\infty^+\backslash A_{31, \infty}^+, \] has compact support. So $I_{\rm{unram}}^{\mathfrak{o}_{31}}(f, x, T)$ is integrable over $Z_\infty^+G_\mathbb{Q}\backslash G_\mathbb{A}$.

The integral is
\begin{eqnarray*}
    c_{P_{31}}\sum_{\gamma\in\{M_{31}^{\mathfrak{o}_{31}}\}}\tilde{\tau}(\gamma, M)\int_{K}\int_{N_{31, \mathbb{A}}}\int_{M(\gamma)_{31, \mathbb{A}}\backslash M_{31, \mathbb{A}}}f(k^{-1}n^{-1}m^{-1}\gamma mnk)\\
    \cdot \int_{Z_\infty^+\backslash A_{31, \infty}^+}(1-\hat{\tau}_{P_{31}}(H_{0}(amnk)-T)-\hat{\tau}_{P_{13}}(H_{0}(w_samnk)-T))da\ dm\ dn\ dk. 
\end{eqnarray*}
For fixed $m$, $k$ and $n$,\[1-\hat{\tau}_{P_{31}}(H_{0}(amnk)-T)-\hat{\tau}_{P_{13}}(H_{0}(w_samnk)-T)\] is the characteristic function of the interval \[[-\hat{\alpha}_1(T)-\hat{\alpha}_3(H_0(m))+\hat{\alpha}_1(H_0(w_sn)), \hat{\alpha}_3(T)-\hat{\alpha}_3(H_0(m))], \]here $w_s^{-1}\hat{\alpha}_{1}=-\hat{\alpha}_3$.

Thus the integral is the sum of 

\begin{eqnarray}
   \hat{\alpha}_3(T)\cdot c_{P_{31}}a_{P_{31}}\sum_{\gamma\in\{M_{31}^{\mathfrak{o}_{31}}\}}\tilde{\tau}(\gamma, M)\cdot\label{9111}\\\notag
    \int_{K}\int_{N_{31, \mathbb{A}}}\int_{M(\gamma)_{31, \mathbb{A}}\backslash M_{31, \mathbb{A}}}f(k^{-1}n^{-1}m^{-1}\gamma mnk)dm\ dn\ dk\\
    +\hat{\alpha}_1(T)\cdot c_{P_{13}}a_{P_{13}}\sum_{\gamma\in\{M_{13}^{\mathfrak{o}_{31}}\}}\tilde{\tau}(\gamma, M)\cdot\label{9222}\\\notag
    \int_{K}\int_{N_{13, \mathbb{A}}}\int_{M(\gamma)_{13, \mathbb{A}}\backslash M_{13, \mathbb{A}}}f(k^{-1}n^{-1}m^{-1}\gamma mnk)dm\ dn\ dk
\end{eqnarray}
and
\begin{eqnarray}\label{9.6}
    -c_{P_{31}}a_{P_{31}}\sum_{\gamma\in\{M_{31}^{\mathfrak{o}_{31}}\}}\tilde{\tau}(\gamma, M)\\
    \int_{K}\int_{N_{31, \mathbb{A}}}\int_{M(\gamma)_{31, \mathbb{A}}\backslash M_{31, \mathbb{A}}}f(k^{-1}n^{-1}m^{-1}\gamma mnk)
    \cdot \hat{\alpha}_1(H_{0}(w_{(14)}n))dm\ dn\ dk.\notag
\end{eqnarray}
We change the variable of integration on $N_\mathbb{A}$,  by Lemma \ref{l2},  the sum of (\ref{9111}) and (\ref{9222}) become 
\begin{eqnarray*}
    \hat{\alpha}_3(T)\cdot c_{P_{31}}a_{P_{31}}\sum_{\gamma\in\{M_{31}^{\mathfrak{o}_{31}}\}}\tilde{\tau}(\gamma, M)\cdot\\
    \int_{K}\int_{M(\gamma)_{31, \mathbb{A}}\backslash M_{31, \mathbb{A}}}\int_{N_{31, \mathbb{A}}}f(k^{-1}m^{-1}\gamma nmk)\cdot \rm{exp}(-<2\rho_{P_{31}}, H_0(m)>)dn\ dm\ dk\\
    +\hat{\alpha}_1(T)\cdot c_{P_{13}}a_{P_{13}}\sum_{\gamma\in\{M_{13}^{\mathfrak{o}_{31}}\}}\tilde{\tau}(\gamma, M)\cdot\\
    \int_{K}\int_{M(\gamma)_{13, \mathbb{A}}\backslash M_{13, \mathbb{A}}}\int_{N_{13, \mathbb{A}}}f(k^{-1}m^{-1}\gamma nmk)\cdot \rm{exp}(-<2\rho_{P_{13}}, H_0(m)>)dn\ dm\ dk. 
\end{eqnarray*}
This term equals 
\begin{eqnarray*}
    \hat{\alpha}_3(T)\cdot c_{P_{31}}a_{P_{31}}\sum_{\gamma\in\{M_{31}^{\mathfrak{o}_{31}}\}}(n_{\gamma, M})^{-1}\cdot\\
    \int_{K}\int_{A_\infty^+M(\gamma)_{31, \mathbb{Q}}\backslash M_{31, \mathbb{A}}}\int_{N_{31, \mathbb{A}}}f(k^{-1}m^{-1}\gamma nmk)\cdot \rm{exp}(-<2\rho_{P_{31}}, H_0(m)>)dn\ dm\ dk\\
    +\hat{\alpha}_1(T)\cdot c_{P_{13}}a_{P_{13}}\sum_{\gamma\in\{M_{13}^{\mathfrak{o}_{31}}\}}(n_{\gamma, M})^{-1}\cdot\\
    \int_{K}\int_{M(\gamma)_{13, \mathbb{A}}\backslash M_{13, \mathbb{A}}}\int_{N_{13, \mathbb{A}}}f(k^{-1}m^{-1}\gamma nmk)\cdot \rm{exp}(-<2\rho_{P_{13}}, H_0(m)>)dn\ dm\ dk, 
\end{eqnarray*}
that is 
\begin{eqnarray}\label{9.9}
    \hat{\alpha}_3(T)\cdot c_{P_{31}}a_{P_{31}}\int_{K}\int_{A_{31, \infty}^+M_{31, \mathbb{Q}}\backslash M_{31, \mathbb{A}}}\sum_{\gamma\in M_{31}^{\mathfrak{o}_{31}}}\int_{N_{31, \mathbb{A}}}f(k^{-1}m^{-1}\gamma nmk)\\\notag
    \cdot \rm{exp}(-<2\rho_{P_{31}}, H_0(m)>)dn\ dm\ dk
\end{eqnarray}
\begin{eqnarray}\label{9.10}
    +\hat{\alpha}_1(T)\cdot c_{P_{13}}a_{P_{13}}\int_{K}\int_{A_{13, \infty}^+M_{13, \mathbb{Q}}\backslash M_{13, \mathbb{A}}}\sum_{\gamma\in M_{13}^{\mathfrak{o}_{31}}}\int_{N_{13, \mathbb{A}}}f(k^{-1}m^{-1}\gamma nmk)\\\notag
    \cdot \rm{exp}(-<2\rho_{P_{13}}, H_0(m)>)dn\ dm\ dk. 
\end{eqnarray}
 
\subsection{The third parabolic term}
In this section, we compute the integral over $Z_\infty^+G_\mathbb{Q}\backslash G_\mathbb{A}$ of $-K_{P_{31}}''(f, x, T)-K_{P_{13}}''(f, x, T)$ which is the third parabolic term associated to $P_{31}$. Then the second parabolic term can be canceled.

This integral is 
\begin{eqnarray*}
    -\frac{1}{4\pi i}\sum_{\chi}\int_{i\mathfrak{a}_{G}\backslash i\mathfrak{a}_{31}}\int_{Z_\infty^+G_\mathbb{Q}\backslash G_\mathbb{A}}\sum_{\alpha, \beta\in \mathscr{B}_{P_{31},\chi}}E''^{T}_{P_{31}}(\Phi_\alpha, \lambda, x)\overline{E''^{T}_{P_{31}}(\Phi_\beta, \lambda, x)}dx\ d\lambda\\
    -\frac{1}{4\pi i}\sum_\chi\int_{i\mathfrak{a}_{G}\backslash i\mathfrak{a}_{13}}\int_{Z_\infty^+G_\mathbb{Q}\backslash G_\mathbb{A}}\sum_{\alpha, \beta\in \mathscr{B}_{P_{13},\chi}}E''^{T}_{P_{13}}(\Phi_\alpha, \lambda, x)\overline{E''^{T}_{P_{13}}(\Phi_\beta, \lambda, x)}dx\ d\lambda. 
\end{eqnarray*}
\begin{lemma}
    For $\alpha, \beta\in I_{P_{31}}$ and $\lambda$ a nonzero imaginary number in $i\mathfrak{a}_{G}\backslash i\mathfrak{a}_{31}$, $s=(14)$, the integral \[\int_{Z_\infty^+G_\mathbb{Q}\backslash G_\mathbb{A}}E''^{T}_{P_{31}}(\Phi_\alpha, \lambda, x)\overline{E''^{T}_{P_{31}}(\Phi_\beta, \lambda, x)}dx\] is
    \begin{eqnarray}\label{E1}
        &2a_{P_{31}}\hat{\alpha}_3(T)(\Phi_\alpha, \Phi_\beta)
    \end{eqnarray}
    \begin{eqnarray}\label{E2}
        &-a_{P_{31}}(M_{P_{31}}(s^{-1}, s\lambda)\cdot \frac{d}{d\lambda}M_{P_{31}}(s, \lambda)\Phi_\alpha, \Phi_\beta)
    \end{eqnarray}
\end{lemma}
\begin{proof}
    Suppose that $\lambda_1, \overline{\lambda}_2$ are distinct complex numbers in $i\mathfrak{a}_{G}\backslash i\mathfrak{a}_{31}$, with real parts suitably regular. Then by the formula of the inner product which  Langlands introduced in \cite{L1}, 
    \begin{eqnarray*}
    &\int_{Z_\infty^+G_\mathbb{Q}\backslash G_\mathbb{A}}E''^{T}_{P_{31}}(\Phi_\alpha, \lambda_1, x)\overline{E''^{T}_{P_{31}}(\Phi_\beta, \lambda_2, x)}dx\\
    &=\frac{\rm{exp}(<\lambda_1+\overline{\lambda_2}, T>)}{<\lambda_1+\overline{\lambda_2}, \alpha_3>}(\Phi_\alpha, \Phi_\beta)+\frac{\rm{exp}(<s\lambda_1+s\overline{\lambda_2}, T>)}{<s\lambda_1+s\overline{\lambda_2}, \alpha_1>}(M_{P_{31}}(s, \lambda_1)\Phi_\alpha, M_{P_{31}}(s, \lambda_2)\Phi_\beta).
    \end{eqnarray*}

    We observe that this function is meromorphic in $\lambda_1, \lambda_2$. Set $\lambda_1-\lambda_2=a\hat{\alpha}_3$, then we will let this term be the limit as $a$ approaches $0$ of 
    \begin{eqnarray*}
        \frac{\rm{exp}(<a\hat{\alpha}_3, T>)(\Phi_\alpha, \Phi_\beta)-\rm{exp}(-<a\hat{\alpha}_3, T>)(M_{P_{31}}(s, (a\hat{\alpha}_3+1)\lambda_2)\Phi_\alpha, M_{P_{31}}(s, \lambda_2)\Phi_\beta)}{<a\hat{\alpha}_3, \alpha_3>}. 
    \end{eqnarray*}
    Recall that $M_P(s,\lambda)$ is unitry if $\lambda$ is a pure imaginary number. Applying L'Hopital's rule yields the desired result. 
\end{proof}
We now obtain an analogous lemma for $P_{13}$, namely
\begin{lemma}
    For $\alpha, \beta\in I_{P_{13}}$ and $\lambda$ a nonzero imaginary number, $s=(14)$, the integral \[\int_{Z_\infty^+G_\mathbb{Q}\backslash G_\mathbb{A}}E''^{T}_{P_{13}}(\Phi_\alpha, \lambda, x)\overline{E''^{T}_{P_{13}}(\Phi_\beta, \lambda, x)}dx\] is
    \begin{eqnarray}\label{E3}
        &2a_{P_{13}}\hat{\alpha}_1(T)(\Phi_\alpha, \Phi_\beta)
    \end{eqnarray}
    \begin{eqnarray}\label{E4}
        &-a_{P_{13}}(M_{P_{13}}(s^{-1}, s\lambda)\cdot \frac{d}{d\lambda}M_{P_{13}}(s, \lambda)\Phi_\alpha, \Phi_\beta). 
    \end{eqnarray}
\end{lemma}
Substituting (\ref{E1}) and (\ref{E3}) into $K_{P_{31}}''(f, x, T)+K_{P_{13}}''(f, x, T)$, we obtain
\begin{eqnarray*}
\frac{\hat{\alpha}_3(T)}{2\pi i}a_{P_{31}}\int_{i\mathfrak{a}_{G}\backslash i\mathfrak{a}_{31}}\rm{tr}\,\pi_{P_{31}}(\lambda, f)d\lambda\\
+\frac{\hat{\alpha}_1(T)}{2\pi i}a_{P_{13}}\int_{i\mathfrak{a}_{G}\backslash i\mathfrak{a}_{13}}\rm{tr}\,\pi_{P_{13}}(\lambda, f)d\lambda. 
\end{eqnarray*}
According to Lemma \ref{lemma4.4}, we can write this as
\begin{eqnarray*}
    c_{P_{31}}a_{P_{31}}\cdot\frac{\hat{\alpha}_3(T)}{2\pi i}\int_{i\mathfrak{a}_{G}\backslash i\mathfrak{a}_{31}}\int_{A_{31, \infty}^+M_{31, \mathbb{Q}}\backslash M_{31\mathbb{A}}}P_{P_{31}}(\lambda, f, mk, mk)dm\ dk\ d\lambda\\
    +c_{P_{13}}a_{P_{13}}\cdot\frac{\hat{\alpha}_1(T)}{2\pi i}\int_{i\mathfrak{a}_{G}\backslash i\mathfrak{a}_{13}}\int_{A_{13, \infty}^+M_{13, \mathbb{Q}}\backslash M_{13, \mathbb{A}}}P_{P_{13}}(\lambda, f, mk, mk)dm\ dk\ d\lambda, 
\end{eqnarray*}by the continity of $P_{P}(\lambda,x,y,x)$, which is the product of 
\begin{eqnarray*}
\rm{exp}(<\lambda+\rho_P, H_P(y)>)\rm{exp}(<-\lambda-\rho_P, H_P(x)>)
\end{eqnarray*}
and  
\begin{eqnarray*}
    \sum_{\gamma\in M_{\mathbb{Q}}}\int_{N_\mathbb{A}}\int_{\mathfrak{a}_G\backslash\mathfrak{a}}f(x^{-1}nh_a\gamma y)\rm{exp}(<-\lambda-\rho_P, a>)da\ dn. 
\end{eqnarray*}
We now apply the Fourier inversion formula to obtain
\begin{eqnarray*}
    c_{P_{31}}a_{P_{31}}\cdot \hat{\alpha}_3(T)\int_K\int_{A_{31, \infty}^+M_{31, \mathbb{Q}}\backslash M_{31, \mathbb{A}}}\sum_{\gamma\in M_{31}}\int_{N_{31, \mathbb{A}}}f(k^{-1}m^{-1}\gamma n mk)\\
    \rm{exp}(-<2\rho_{P_{31}}, H_0(m)>)dn\ dm\ dk\\
    +c_{P_{13}}a_{P_{13}}\cdot \hat{\alpha}_1(T)\int_K\int_{A_{13, \infty}^+M_{13, \mathbb{Q}}\backslash M_{13, \mathbb{A}}}\sum_{\gamma\in M_{13}}\int_{N_{13, \mathbb{A}}}f(k^{-1}m^{-1}\gamma n mk)\\
    \rm{exp}(-<2\rho_{P_{13}}, H_0(m)>)dn\ dm\ dk. 
\end{eqnarray*}
Now the terms corresponding to (\ref{E2}) and (\ref{E4}) are 
\begin{eqnarray}\label{9.18}
\frac{a_{P_{31}}}{4\pi i}\sum_{\chi}\int_{i\mathfrak{a}_{G}\backslash i\mathfrak{a}_{31}}\rm{tr}\{M_{P_{31}}(s^{-1}, s\lambda)\cdot(\frac{d}{d\lambda}M_{P_{31}}(s, \lambda))\cdot\pi_{P_{31}, \chi}(\lambda, f)\}d\lambda\\\notag
+\frac{a_{P_{13}}}{4\pi i}\sum_{\chi}\int_{i\mathfrak{a}_{G}\backslash i\mathfrak{a}_{13}}\rm{tr}\{M_{P_{13}}(s^{-1}, s\lambda)\cdot(\frac{d}{d\lambda}M_{P_{13}}(s, \lambda))\cdot\pi_{P_{13}, \chi}(\lambda, f)\}d\lambda. 
\end{eqnarray}
where $\pi_{P_{31}, \chi}(\lambda, f), \pi_{P_{13}, \chi}(\lambda, f)$ are the restrictions of $\pi_{P_{31}}(\lambda, f), \pi_{P_{13}}(\lambda, f)$ to $\mathscr{H}_{P, \chi}$. 
The term (\ref{9.18}) is finite since all of other terms are convergent with respect to $T$.

Observe that the terms (\ref{9.9}), (\ref{9.10}) can be canceled, but there are some additional remaining terms:
\begin{eqnarray}\label{4444}
    c_{P_{31}}a_{P_{31}}\cdot \hat{\alpha}_3(T)\int_K\int_{A_{31, \infty}^+M_{31, \mathbb{Q}}\backslash M_{31, \mathbb{A}}}\sum_{\gamma\in M_{31}-M_{31}^{\mathfrak{o}^0_{31}}}\int_{N_{31, \mathbb{A}}}\\\notag
    f(k^{-1}m^{-1}\gamma n mk)\rm{exp}(-<2\rho_{P_{31}}, H_0(m)>)dn\ dm\ dk\\\notag
    +c_{P_{13}}a_{P_{13}}\cdot \hat{\alpha}_1(T)\int_K\int_{A_{13, \infty}^+M_{13, \mathbb{Q}}\backslash M_{13, \mathbb{A}}}\sum_{\gamma\in M_{13}-M_{13}^{\mathfrak{o}^0_{31}}}\int_{N_{13, \mathbb{A}}}\\\notag
    f(k^{-1}m^{-1}\gamma n mk)\rm{exp}(-<2\rho_{P_{13}}, H_0(m)>)dn\ dm\ dk. 
\end{eqnarray}
However, for any unramified orbit $\mathfrak{o}\neq\mathfrak{o}^0_{31}$ in $M_{31}$, $\gamma\in M_{31}^\mathfrak{o}$, suppose fix $\chi$ and $P=P_{31}$, the sum $\sum_{\alpha,\beta\in \mathscr{B}_{P,\chi}}(\Phi_\alpha,\Phi_\beta)$ is the trace of $H_{P,\chi}$ which is the finite dimensional subspace of cusp forms.

Apply Lemma \ref{l2}, we write the first integral in (\ref{4444}) as 
\begin{eqnarray*}
    c_{P_{31}}a_{P_{31}}\cdot \hat{\alpha}_3(T)\int_K\int_{N_{31, \mathbb{A}}}\int_{A_{31, \infty}^+M_{31, \mathbb{Q}}\backslash M_{31, \mathbb{A}}}\sum_{\gamma\in M_{31}^{\mathfrak{o}}}f(k^{-1}n^{-1}m^{-1}\gamma  mnk)dm\ dn\ dk.
\end{eqnarray*}
For fixed $k$ and $n$, we consider the inner integral 
\[\int_{A_{31, \infty}^+M_{31, \mathbb{Q}}\backslash M_{31, \mathbb{A}}}\sum_{\gamma\in M_{31}^{\mathfrak{o}}}f(k^{-1}n^{-1}m^{-1}\gamma  mnk)dm.\]
According to \cite{H2}, $\gamma$ is the regular semisimple element but not elliptic element in $M_{31}$, $A_\infty^+\backslash M(\gamma)_\mathbb{A}$ is not compact, thus the orbital integral of this $\gamma$ equals zero. Therefore we only need to consider the elements in the ramified orbits:
\begin{eqnarray}\label{yyy}
    \sum_{\rm{ramified}\,\mathfrak{o}}c_{P_{31}}a_{P_{31}}\cdot \hat{\alpha}_3(T)\int_K\int_{A_{31, \infty}^+M_{31, \mathbb{Q}}\backslash M_{31, \mathbb{A}}}\sum_{\gamma\in M_{31}^{\mathfrak{o}}}\int_{N_{31, \mathbb{A}}}\\\notag
    f(k^{-1}m^{-1}\gamma n mk)\rm{exp}(-<2\rho_{P_{31}}, H_0(m)>)dn\ dm\ dk\\\notag
    +\sum_{\rm{ramified}\,\mathfrak{o}}c_{P_{13}}a_{P_{13}}\cdot \hat{\alpha}_1(T)\int_K\int_{A_{13, \infty}^+M_{13, \mathbb{Q}}\backslash M_{13, \mathbb{A}}}\sum_{\gamma\in M_{13}^{\mathfrak{o}}}\int_{N_{13, \mathbb{A}}}\\\notag
    f(k^{-1}m^{-1}\gamma n mk)\rm{exp}(-<2\rho_{P_{13}}, H_0(m)>)dn\ dm\ dk. 
\end{eqnarray}
We have 
\begin{lemma}
    $I^{\mathfrak{o}_{31}^0}_{\rm{unram}}(f,x,T)-K''_{P_{31}}(f,x,T)-K''_{P_{13}}(f,x,T)$ equals the sum of (\ref{9.6}), (\ref{9.18}) and (\ref{yyy}).
\end{lemma}

\section{Terms associated to $P_{22}$}
\[\Omega(\mathfrak{a}_{22}, \mathfrak{a}_{22})=\{(1), (14)(23)\}. \]

\subsection{The first parabolic term}
The first parabolic term is \[J_{\rm{unram}}^{\mathfrak{o}_{22}^0}(f, x, T)+J_{\rm{ram}}^{\mathfrak{o}_{22}^2}(f, x, T)-K'_{P_{22}}(f, x, T). \]
We shall prove that this term approaches $0$ as $T\rightarrow \infty$.

Recall $J_{\rm{unram}}^{\mathfrak{o}_{22}^0}(f, x, T)$ equals
\[\frac{1}{2}\sum_{\gamma\in \{M_{t, 22}^{\mathfrak{o}_{22}^0}\}}(n_{\gamma, M_{22}})^{-1}\sum_{\delta\in M(\gamma)_{22, \mathbb{Q}}\backslash G_\mathbb{Q}}f(x^{-1}\delta^{-1}\gamma\delta x)(\hat{\tau}_{P_{22}}(H_0(\delta x)- T)+\hat{\tau}_{P_{22}}(H_0(w_s\delta x)- T)).  \]
Then
\begin{eqnarray*}
&\sum_{\gamma\in \{M_{t, 22}^{\mathfrak{o}_{22}^0}\}}(n_{\gamma, M_{22}})^{-1}\sum_{\delta\in M(\gamma)_{22, \mathbb{Q}}\backslash G_\mathbb{Q}}f(x^{-1}\delta^{-1}\gamma\delta x)(\hat{\tau}_{P_{22}}(H_0(w_s\delta x)- T))\\
&=\sum_{\gamma\in \{M_{t, 22}^{\mathfrak{o}_{22}^0}\}}(n_{\gamma, M_{22}})^{-1}\sum_{\delta\in M(w_s\gamma w_s^{-1})_{22, \mathbb{Q}}\backslash G_\mathbb{Q}}f(x^{-1}\delta^{-1}w_s\gamma w_s^{-1}\delta x)(\hat{\tau}_{P_{22}}(H_0(w_s\delta x)- T))\\
&=\sum_{\gamma\in \{M_{t, 22}^{\mathfrak{o}_{22}^0}\}}(n_{\gamma, M_{22}})^{-1}\sum_{\delta\in M(\gamma)_{22, \mathbb{Q}}\backslash G_\mathbb{Q}}f(x^{-1}\delta^{-1}\gamma \delta x)(\hat{\tau}_{P_{22}}(H_0(\delta x)- T)). 
\end{eqnarray*}
Thus, 
\begin{eqnarray*}
&J_{\rm{unram}}^{\mathfrak{o}_{22}^0}(f, x, T)\\
&=\sum_{\gamma\in \{M_{t, 22}^{\mathfrak{o}_{31}^0}\}}(n_{\gamma, M_{22}})^{-1}\sum_{\delta\in M(\gamma)_{22, \mathbb{Q}}\backslash G_\mathbb{Q}}f(x^{-1}\delta^{-1}\gamma\delta x)(\hat{\tau}_{P_{22}}(H_0(\delta x))-T). 
\end{eqnarray*}
Since for $\gamma\in{M_{t, 22}^{\mathfrak{o}_{22}^0}}$, the group $N(\gamma_s)$ is trivial, by Lemma \ref{l1}, $J_{\rm{unram}}^{\mathfrak{o}_{22}}(f, x, T)$ equals
\begin{eqnarray*}
&\sum_{\gamma\in\{M_{t, 22}^{\mathfrak{o}_{22}}\}}(n_{\gamma, M_{22}})\sum_{\delta\in N_{22, \mathbb{Q}}M(\gamma)_{22, \mathbb{Q}}\backslash G_\mathbb{Q}}\sum_{v\in N_{22, \mathbb{Q}}}f(x^{-1}\delta^{-1}\gamma v\delta x)(\hat{\tau}_{P_{31}}(H_0(\delta x)-T)). 
\end{eqnarray*}
Which is 
\begin{eqnarray*}
&\sum_{\delta\in P_{22, \mathbb{Q}}\backslash G_\mathbb{Q}}\sum_{\gamma\in M_{t, 22}^{\mathfrak{o}_{22}^0}}\sum_{v\in N_{22, \mathbb{Q}}}f(x^{-1}\delta^{-1}\gamma v\delta x)(\hat{\tau}_{P_{22}}(H_0(\delta x)-T)). 
\end{eqnarray*}
By Lemma \ref{l1}, $J_{\rm{ram}}^{\mathfrak{o}_{22}^2}(f, x, T)$ equals 
\begin{eqnarray*}
    \sum_{\delta\in P_{22, \mathbb{Q}}\backslash G_\mathbb{Q}}\sum_{\gamma\in M_{n, 22}^{\mathfrak{o}_{22}^2}}\sum_{v\in N_{22, \mathbb{Q}}}f(x^{-1}\delta^{-1}\gamma v\delta x)\hat{\tau}_{22}(H_0(\delta x)-T). 
\end{eqnarray*}
$J_{P_{22}}(f, x, T)$ equals
\begin{eqnarray}\label{j1}
&J_{\rm{unram}}^{\mathfrak{o}_{22}^0}(f, x, T)+J_{\rm{ram}}^{\mathfrak{o}_{22}^2}(f, x, T)
\end{eqnarray}
\begin{eqnarray*}
&+\sum_{\mathfrak{o}\in\{\mathfrak{o}_{211}^0, \mathfrak{o}_{1111}^0\}}\frac{1}{|\Omega(\mathfrak{a}_\mathfrak{o}, P_{22})|}\sum_{s\in\Omega(\mathfrak{a}_\mathfrak{o},P_{22})}\sum_{\delta\in P_{22, \mathbb{Q}}\backslash G_\mathbb{Q}}\sum_{\gamma\in M_{t, 22}^{\mathfrak{o}}}\sum_{v\in N_{22, \mathbb{Q}}}f(x^{-1}\delta^{-1}\gamma v\delta x)\\\notag
&(\hat{\tau}_{P_{22}}(H_0(w_s\delta x)-T))
\end{eqnarray*}
\begin{eqnarray*}
&+\sum_{\mathfrak{o}\ ramified}\sum_{\delta\in M_{22, \mathbb{Q}}N(\gamma_s)_{22, \mathbb{Q}}\backslash G_\mathbb{Q}}\sum_{\gamma\in M_{22}^\mathfrak{o}}\sum_{v\in N(\gamma_s)_{22, \mathbb{Q}}}f(x^{-1}\delta^{-1}\gamma v\delta x)(\hat{\tau}_{P_{22}}(H_0(\delta x)-T)). \notag
\end{eqnarray*}
When we fix an unramified orbit $\mathfrak{o}$, for $s_1\in\Omega(\mathfrak{a}_\mathfrak{o},P_{22})$, 
\begin{eqnarray*}
    \sum_{\delta\in P_{22, \mathbb{Q}}\backslash G_\mathbb{Q}}\sum_{\gamma\in M_{22}^{\mathfrak{o}}}\sum_{v\in N_{22, \mathbb{Q}}}f(x^{-1}\delta^{-1}\gamma v\delta x)(\hat{\tau}_{P_{22}}(H_0(w_s\delta x))-T)
\end{eqnarray*}
is the same as
\begin{eqnarray*}
    \sum_{\delta\in P_{22, \mathbb{Q}}\backslash G_\mathbb{Q}}\sum_{\gamma\in M_{22}^{\mathfrak{o}}}\sum_{v\in N_{22, \mathbb{Q}}}f(x^{-1}\delta^{-1}\gamma v\delta x)(\hat{\tau}_{P_{22}}(H_0(\delta x))-T). 
\end{eqnarray*}
By Lemma \ref{l1}, the sum over ramified orbits in (\ref{j1}) equals
\[\sum_{\mathfrak{o}\ ramified}\sum_{\delta\in P_{22, \mathbb{Q}}\backslash G_\mathbb{Q}}\sum_{\gamma\in M_{22}^\mathfrak{o}}\sum_{v\in N_{22, \mathbb{Q}}}f(x^{-1}\delta^{-1}\gamma v\delta x)(\hat{\tau}_{P_{22}}(H_0(\delta x)-T)).\]

Since $M_{22}=\cup_{\mathfrak{o}} M_{22}^\mathfrak{o}$, $J_{P_{22}}(f, x, T)$ equals
\begin{eqnarray*}
&\sum_{\delta\in P_{22, \mathbb{Q}}\backslash G_\mathbb{Q}}\sum_{\gamma\in M_{22, \mathbb{Q}}}\sum_{v\in N_{22, \mathbb{Q}}}f(x^{-1}\delta^{-1}\gamma v\delta x)(\hat{\tau}_{P_{22}}(H_0(\delta x)-T)). 
\end{eqnarray*}
Similar to the term (\ref{e3}), $J_{P_{22}}(f, x, T)$ is 
\begin{eqnarray*}
&\sum_{\delta\in P_{22, \mathbb{Q}}\backslash G_\mathbb{Q}}\sum_{\gamma\in M_{22}}\int_{N_{22, \mathbb{A}}}f(x^{-1}\delta^{-1}\gamma n\delta x)dn(\hat{\tau}_{P_{22}}(H_0(\delta x)-T)). 
\end{eqnarray*}
Recall $K'_{P_{22}}(f, x, T)$ equals 
\begin{eqnarray*}
    &\frac{1}{4\pi i}\sum_{P_{22, \mathbb{Q}}\backslash G_\mathbb{Q}}\sum_{\chi}\int_{i\mathfrak{a}_{G}\backslash i\mathfrak{a}_{22}}\{\sum_{\beta\in \mathscr{B}_{P_{22}, \chi}}E^{c_{22}}_{P_{22}}(\pi_{P_{22}}(\lambda, f)\Phi_\beta, \lambda, \delta x)\overline{E^{c_{22}}_{P_{22}}(\Phi_\beta, \lambda, \delta x)}\}d\lambda\\
    &\hat{\tau}_{P_{22}}(H_0(\delta x)-T). 
\end{eqnarray*}
This term is the sum of 
\begin{eqnarray*}
    &\sum_{\gamma\in M_{22}}\sum_{\delta\in P_{22, \mathbb{Q}}\backslash G_\mathbb{Q}}\int_{N_{22, \mathbb{A}}}f(x^{-1}\delta^{-1}\gamma\delta x)dn\cdot \hat{\tau}_{P_{22}}(H_0(\delta x)-T)
\end{eqnarray*} 
and 
\begin{eqnarray*}
    &\frac{1}{4\pi i}\sum_{P_{22, \mathbb{Q}}\backslash G_\mathbb{Q}}\sum_{\chi}\int_{i\mathfrak{a}_{G}\backslash i\mathfrak{a}_{22}}\{\sum_{\beta\in \mathscr{B}_{P_{22}, \chi}}(M_{P_{22}}(s, \lambda)\pi_{P_{22}}(\lambda, f)\Phi_\beta)(\delta x)\overline{\Phi_\beta(\delta x)}\}\\
    &\rm{exp}(<-2\lambda, H_0(\delta x)>)d\lambda\ \hat{\tau}_{P_{22}}(H_0(\delta x)-T)\\
    &+\frac{1}{4\pi i}\sum_{P_{22, \mathbb{Q}}\backslash G_\mathbb{Q}}\sum_{\chi}\int_{i\mathfrak{a}_{G}\backslash i\mathfrak{a}_{22}}\{\sum_{\beta\in \mathscr{B}_{P_{22}, \chi}}(\pi_{P_{22}}(\lambda, f)\Phi_\beta)(\delta x)\overline{M_{P_{22}}(s, \lambda)\Phi_\beta(\delta x)}\}\\
    &\rm{exp}(<2\lambda, H_0(\delta x)>)d\lambda\ \rm{exp}(<-2\rho_{P_{22}}, H_0(\delta x)>)\hat{\tau}_{P_{22}}(H_0(\delta x)-T). 
\end{eqnarray*}
By Lemma \ref{l7}, the second function's integral over $Z_\infty^+G_\mathbb{Q}\backslash G_\mathbb{A}$ approaches $0$ as $T\rightarrow \infty$. Thus 
\begin{lemma}
    The sum \[J_{P_{22}}(f,x,T)-K'_{P_{22}}(f,x,T)\] approaches $0$ as $T\rightarrow\infty$.
\end{lemma}

\subsection{The second parabolic term}
The second parabolic term correspond to $\mathfrak{o}_{22}^0$ and $\mathfrak{o}_{22}^2$ is given by \[I_{\rm{unram}}^{\mathfrak{o}_{22}^0}(f, x, T)+I_{\rm{ram}}^{\mathfrak{o}_{22}^2}(f, x, T). \]
Recall $I_{\rm{unram}}^{\mathfrak{o}^0_{22}}(f, x, T)$ is \[\frac{1}{2}\sum_{\gamma\in \{M_{t,22}^{\mathfrak{o}^0_{22}}\}}(n_{\gamma, M_{22}})^{-1}\sum_{\delta\in M(\gamma)_{22, \mathbb{Q}}\backslash G_\mathbb{Q}}f(x^{-1}\delta^{-1}\gamma\delta x)(1-\hat{\tau}_{P_{22}}(H_0(\delta x)- T)-\hat{\tau}_{P_{22}}(H_0(\delta x)-T)). \]

Now, the integral $\int_{Z_\infty^+G_\mathbb{Q}\backslash G_\mathbb{A}}|I_{\rm{unram}}^{\mathfrak{o}^0_{22}}(f, x, T)|dx$ is bounded by \[\frac{1}{2}\sum_{\gamma\in\{M_{t,22}^{\mathfrak{o}_{22}^0}\}}(n_{\gamma, M})^{-1}\int_{Z_\infty^+M(\gamma)_{22, \mathbb{Q}}\backslash G_\mathbb{A}}|f(x^{-1}\gamma x)|\cdot(1-\hat{\tau}_{P_{22}}(H_0(x)-T)-\hat{\tau}_{P_{22}}(H_0(w_sx)-T))dx. \]
It equals
\begin{eqnarray*}
\frac{c_{P_{22}}}{2}\sum_{\gamma\in\{M_{t,22}^{\mathfrak{o}^0_{22}}\}}(n_{\gamma, M})^{-1}\int_{K}\int_{A_{22, \infty^+}M(\gamma)_{22, \mathbb{Q}}\backslash P_{22, \mathbb{A}}}\int_{Z_\infty^+\backslash A_{22, \infty}^+}|f(k^{-1}p^{-1}\gamma pk)|\\
\cdot(1-\hat{\tau}_{P_{22}}(H_0(ap)-T)-\hat{\tau}_{P_{22}}(H_0(w_sap)-T))da\ d_rp\ dk. 
\end{eqnarray*}
Hence the integral becomes
\begin{eqnarray*}
    \frac{c_{P_{22}}}{2}\sum_{\gamma\in\{M_{t,22}^{\mathfrak{o}_{22}^0}\}}\tilde{\tau}(\gamma, M)\int_{K}\int_{M(\gamma)_{22, \mathbb{A}}\backslash P_{22, \mathbb{A}}}|f(k^{-1}p^{-1}\gamma pk)|\\
    \cdot\int_{Z_\infty^+\backslash A_{22, \infty}^+}(1-\hat{\tau}_{P_{22}}(H_0(ap)-T)-\hat{\tau}_{P_{22}}(H_0(w_sap)-T))da\ d_lp\ dk. 
\end{eqnarray*}
The sum over $\gamma$ is finite by Lemma \ref{l9}.

Since the function \[f^K(p)=\int_{K}f(k^{-1}pk)dk,\quad p\in P_{22, \mathbb{A}}\] has compact support, by Lemma \ref{l10}, the integral on $M(\gamma)_{22, \mathbb{A}}\backslash P_{22, \mathbb{A}}$ can be taken over a compact set. For any $p$, the function \[a\longrightarrow 1-\hat{\tau}_{P_{22}}(H_0(ap)-T)-\hat{\tau}_{P_{22}}(H_0(w_sap)-T),\quad a\in Z_\infty^+\backslash A_{22, \infty}^+, \] has compact suppoet. $I_{\rm{unram}}^{\mathfrak{o}_{22}^0}(f, x, T)$ is integrable over $Z_\infty^+G_\mathbb{Q}\backslash G_\mathbb{A}$, and its integral equals
\begin{eqnarray*}
    \frac{c_{P_{22}}}{2}\sum_{\gamma\in\{M_{t,22}^{\mathfrak{o}_{22}^0}\}}\tilde{\tau}(\gamma, M)\int_{K}\int_{N_{22, \mathbb{A}}}\int_{M(\gamma)_{22, \mathbb{A}}\backslash M_{22, \mathbb{A}}}f(k^{-1}n^{-1}m^{-1}\gamma mnk)\\
    \cdot \int_{Z_\infty^+\backslash A_{22, \infty}^+}(1-\hat{\tau}_{P_{22}}(H_0(amnk)-T)-\hat{\tau}_{P_{22}}(H_0(w_samnk)-T))da\ dm\ dn\ dk. 
\end{eqnarray*}
For fixed $m$, $n$ and $k$,the function \[1-\hat{\tau}_{P_{22}}(H_0(amnk)-T)-\hat{\tau}_{P_{22}}(H_0(w_samnk)-T)\]
is the characteristic function of the interval
\[[-\hat{\alpha}_2(T)-\hat{\alpha}_2(H_0(m))+\hat{\alpha}_2(H_0(w_sn)), \hat{\alpha}_2(T)-\hat{\alpha}_2(H_0(m))]. \]
Hence, the integral is the sum of 
\begin{eqnarray}\label{102}
    -\frac{c_{P_{22}}}{2}a_{P_{22}}\sum_{\gamma\in\{M_{t,22}^{\mathfrak{o}_{22}^0}\}}\tilde{\tau}(\gamma, M)\int_{K}\int_{N_{22, \mathbb{A}}}\int_{M(\gamma)_{22, \mathbb{A}}\backslash M_{22, \mathbb{A}}}f(k^{-1}n^{-1}m^{-1}\gamma mnk)\\
    \cdot \hat{\alpha}_2(H(w_{(14)(23)}n))dm\ dn\ dk\notag
\end{eqnarray}
and 
\begin{eqnarray}
    +\hat{\alpha}_2(T)\cdot c_{P_{22}}a_{P_{22}}\sum_{\gamma\in\{M_{t,22}^{\mathfrak{o}_{22}^0}\}}\tilde{\tau}(\gamma, M)\cdot\int_{K}\int_{N_{22, \mathbb{A}}}\int_{M(\gamma)_{22, \mathbb{A}}\backslash M_{22, \mathbb{A}}}\\\notag
    f(k^{-1}n^{-1}m^{-1}\gamma mnk)dm\ dn\ dk. 
\end{eqnarray}
We change the variable of integration on $N_{22, \mathbb{A}}$, use Lemma \ref{l2}, the second term becomes 
\begin{eqnarray*}
    \hat{\alpha}_2(T)\cdot c_{P_{22}}a_{P_{22}}\sum_{\gamma\in\{M_{t,22}^{\mathfrak{o}_{22}^0}\}}\tilde{\tau}(\gamma, M)\cdot\\
    \int_{K}\int_{M(\gamma)_{22, \mathbb{A}}\backslash M_{22, \mathbb{A}}}\int_{N_{22, \mathbb{A}}}f(k^{-1}m^{-1}\gamma nmk)\cdot \rm{exp}(-<2\rho_{P_{22}}, H_0(m)>)dn\ dm\ dk. 
\end{eqnarray*}
This term equals 
\begin{eqnarray*}
    \hat{\alpha}_2(T)\cdot c_{P_{22}}a_{P_{22}}\sum_{\gamma\in\{M_{t,22}^{\mathfrak{o}_{22}^0}\}}(n_{\gamma, M})^{-1}\cdot\\
    \int_{K}\int_{N_{22, \mathbb{A}}}\int_{A_{22, \infty}^+M(\gamma)_{22, \mathbb{Q}}\backslash M_{22, \mathbb{A}}}\int_{N_{22, \mathbb{A}}}f(k^{-1}m^{-1}\gamma nmk)\cdot \rm{exp}(-<2\rho_{P_{22}}, H_0(m)>)dn\ dm\ dk, 
\end{eqnarray*}
that is 
\begin{eqnarray}\label{104}
    \hat{\alpha}_2(T)\cdot c_{P_{22}}a_{P_{22}}\int_{K}\int_{A_{22, \infty}^+M_{22, \mathbb{Q}}\backslash M_{22, \mathbb{A}}}\sum_{\gamma\in M_{t,22}^{\mathfrak{o}_{22}^{0}}}\int_{N_{22, \mathbb{A}}}f(k^{-1}m^{-1}\gamma nmk)\\\notag
    \cdot \rm{exp}(-<2\rho_{P_{22}}, H_0(m)>)dn\ dm\ dk. 
\end{eqnarray}
By Lemma \ref{8.1}, the term of ramified orbit is the sum of 
\begin{eqnarray}\label{107}
\rm{lim}_{\lambda\rightarrow 0}\int_{Z_\infty^+G_\mathbb{Q}\backslash G_\mathbb{A}}D_\lambda\{\lambda\mu_{\mathfrak{o}_{22}^2}(\lambda, f,x)\}dx
\end{eqnarray}
and
\begin{equation}\label{108}
    \begin{aligned}
    \hat{\alpha}_2(T)\cdot c_{P_{22}}a_{P_{22}}\int_K\int_{A_{22, \infty}^+M_{22, \mathbb{Q}}\backslash M_{22, \mathbb{A}}} \sum_{\gamma\in M_{n,22}^{\mathfrak{o}_{22}^{2}}}\int_{N_{22, \mathbb{A}}}f(k^{-1}m^{-1}\gamma nmk)\\
    \cdot \rm{exp}(-<2\rho_{P_{22}}, H_0(m)>)dn\ dm\ dk. 
    \end{aligned}
\end{equation} 
We now combine this term with (\ref{104}), we have
\begin{eqnarray}\label{1010}
    \hat{\alpha}_2(T)\cdot c_{P_{22}}a_{P_{22}}\int_K\int_{A_{22, \infty}^+M_{22, \mathbb{Q}}\backslash M_{22, \mathbb{A}}} \sum_{\gamma\in M_{22}^{\mathfrak{o}_{22}}}\int_{N_{22, \mathbb{A}}}f(k^{-1}m^{-1}\gamma nmk)\\\notag
    \cdot \rm{exp}(-<2\rho_{P_{22}}, H_0(m)>)dn\ dm\ dk. 
\end{eqnarray}
\subsection{The third parabolic term}
We shall prove that the second parabolic term associated to $\mathfrak{o}_{22}^0$ and $\mathfrak{o}_{22}^2$ can be canceled by the third parabolic term.

The integral of $-K_{P_{22}}''(f, x, T)$ is 
\begin{eqnarray*}
    -\frac{1}{4\pi i}\sum_{\alpha, \beta\in I_{P_{22}}}\int_{i\mathfrak{a}_{P_{G}}\backslash i\mathfrak{a}_{22}}\int_{Z_\infty^+G_\mathbb{Q}\backslash G_\mathbb{A}}E''^{T}_{P_{22}}(\Phi_\alpha, \lambda, x)\overline{E''^{T}_{P_{22}}(\Phi_\beta, \lambda, x)}dx\ d\lambda. 
\end{eqnarray*}
\begin{lemma}
    For $\alpha, \beta\in I_{P_{22}}$ and $\lambda$ a nonzero imaginary number in $i\mathfrak{a}_{G}\backslash i\mathfrak{a}_{22}$, $s=(14)(23)$, the integral \[\int_{Z_\infty^+G_\mathbb{Q}\backslash G_\mathbb{A}}E''^{T}_{P_{22}}(\Phi_\alpha, \lambda, x)\overline{E''^{T}_{P_{22}}(\Phi_\beta, \lambda, x)}dx\] is the sum of
    \begin{eqnarray}\label{1011}
        &2a_{P_{22}}\hat{\alpha}_2(T)(\Phi_\alpha, \Phi_\beta)
    \end{eqnarray}
    \begin{eqnarray}\label{1012}
        &-a_{P_{22}}(M_{P_{22}}(s^{-1}, s\lambda)\cdot \frac{d}{d\lambda}M_{P_{22}}(s, \lambda)\Phi_\alpha, \Phi_\beta)
    \end{eqnarray}
    and
    \begin{eqnarray}\label{1013}
        &\frac{a_{P_{22}}}{<2\lambda, \alpha_2>}\{\rm{exp}(<2\lambda, T>)(\Phi_\alpha, M_{P_{22}}(s, \lambda)\Phi_\beta)-\rm{exp}(<-2\lambda, T>)(M_{P_{22}}(s, \lambda)\Phi_\alpha, \Phi_\beta)\}.
    \end{eqnarray}
\end{lemma}
\begin{proof}
    Suppose that $\lambda_1, \lambda$ are distinct complex numbers in $i\mathfrak{a}_{G}\backslash i\mathfrak{a}_{22}$, whose real parts are suitably regular. By \cite{L1}, 
    \begin{eqnarray*}
    &\int_{Z_\infty^+G_\mathbb{Q}\backslash G_\mathbb{A}}E''^{T}_{P_{22}}(\Phi_\alpha, \lambda_1, x)\overline{E''^{T}_{P_{22}}(\Phi_\beta, \lambda, x)}dx
    \end{eqnarray*}
    equals the sum of 
    \begin{eqnarray*}
    &a_{P_{22}}\frac{\rm{exp}(<\lambda_1+\overline{\lambda}, T>)}{<\lambda_1+\overline{\lambda}, \alpha_2>}(\Phi_\alpha, \Phi_\beta)+a_{P_{22}}\frac{\rm{exp}(<-\lambda_1-\overline{\lambda}, T>)}{<-\lambda_1-\overline{\lambda}, \alpha_2>}(M_{P_{22}}(s, \lambda_1)\Phi_\alpha, M_{P_{22}}(s, \lambda)\Phi_\beta)
    \end{eqnarray*}
    and
    \begin{eqnarray*}
    &+a_{P_{22}}\frac{\rm{exp}(<\lambda_1-\overline{\lambda}, T>)}{<\lambda_1-\overline{\lambda}, \alpha_2>}(\Phi_\alpha, M_{P_{22}}(s, \lambda)\Phi_\beta)+a_{P_{22}}\frac{\rm{exp}(<-\lambda_1+\overline{\lambda}, T>)}{<-\lambda_1+\overline{\lambda}, \alpha_2>}(\Phi_\alpha, M_{P_{22}}(s, \lambda)\Phi_\beta).
    \end{eqnarray*}

    This function is meromorphic in $\lambda_1, \lambda$. Let $\lambda_1-\lambda=a\hat{\alpha}_2$, and take the limit as $a$ approaches $0$ of the sum of
    \begin{eqnarray*}
        a_{P_{22}}\frac{\rm{exp}(<a\hat{\alpha}_2, T>)(\Phi_\alpha, \Phi_\beta)-\rm{exp}(-<a\hat{\alpha}_2, T>)(M_{P_{22}}(s, \lambda+a\hat{\alpha}_2)\Phi_\alpha, M_{P_{22}}(s, \lambda)\Phi_\beta)}{<a\hat{\alpha}_2, \alpha_2>}
    \end{eqnarray*}
    and
    \begin{eqnarray*}
        a_{P_{22}}\frac{\rm{exp}(<a\hat{\alpha}_2+2\lambda, T>)(\Phi_\alpha, M_{P_{22}}(s, \lambda)\Phi_\beta)-\rm{exp}(-<a\hat{\alpha}_2-2\lambda, T>)(M_{P_{22}}(s, \lambda+a\hat{\alpha}_2)\Phi_\alpha, \Phi_\beta)}{<a\hat{\alpha}_2+2\lambda, \alpha_2>}. 
    \end{eqnarray*}
    We apply L'Hopital's rule to obtain the result. 
\end{proof}
The term corresponding to (\ref{1012}) is
\begin{eqnarray}\label{1014}
\frac{a_{P_{22}}}{4\pi i}\sum_{\chi}\int_{i\mathfrak{a}_{G}\backslash i\mathfrak{a}_{22}}\rm{tr}\{M_{P_{22}}(s^{-1}, s\lambda)\cdot(\frac{d}{d\lambda}M_{P_{22}}(s, \lambda))\cdot\pi_{P_{22}, \chi}(\lambda, f)\}d\lambda. 
\end{eqnarray}
This term is finite.

Substituting (\ref{1011}) into $K_{P_{22}}''(f, x, T)$, it equals
\begin{eqnarray*}
    a_{P_{22}}\frac{\hat{\alpha}_2(T)}{2\pi i}\int_{i\mathfrak{a}_{G}\backslash i\mathfrak{a}_{22}}\rm{tr}\,\pi_{P_{22}}(\lambda, f)d\lambda. 
\end{eqnarray*}
By Lemma \ref{lemma4.4}, We can write it as
\begin{eqnarray*}
    c_{P_{22}}a_{P_{22}}\cdot\frac{\hat{\alpha}_2(T)}{\pi i}\int_{i\mathfrak{a}_{G}\backslash i\mathfrak{a}_{22}}\int_{A_{22, \infty}^+M_{22, \mathbb{Q}}\backslash M_{22, \mathbb{A}}}P_{P_{22}}(\lambda, f, mk, mk)dm\ dk\ d\lambda, 
\end{eqnarray*}by the continity of $P_{P_{22}}$.

Then, using the Fourier inversion formula, we obtain
\begin{eqnarray*}
    c_{P_{22}}a_{P_{22}}\cdot \hat{\alpha}_2(T)\int_K\int_{A_{22, \infty}^+M_{22, \mathbb{Q}}\backslash M_{22, \mathbb{A}}}\sum_{\gamma\in M_{22, \mathbb{Q}}}\int_{N_{22, \mathbb{A}}}\\
    f(k^{-1}m^{-1}\gamma n mk)\rm{exp}(-<2\rho_{P_{22}}, H_0(m)>)dn\ dm\ dk. 
\end{eqnarray*}
Then (\ref{104}), (\ref{108}) can be canceled, but there is also something left, 
\begin{eqnarray}\label{p}
    \sum_{\rm{ramified}\, \mathfrak{o}}c_{P_{22}}a_{P_{22}}\cdot \hat{\alpha}_2(T)\int_K\int_{A_{22, \infty}^+M_{22, \mathbb{Q}}\backslash M_{22, \mathbb{A}}}\sum_{\gamma\in M_{22}^{\mathfrak{o}}}\int_{N_{22, \mathbb{A}}}\\\notag
    f(k^{-1}m^{-1}\gamma n mk)\rm{exp}(-<2\rho_{P_{22}}, H_0(m)>)dn\ dm\ dk. 
\end{eqnarray}
We consider the term (\ref{1013}). We insert it into the function $-K''_{P_{22}}(f, x, T)$. We write it as the sum of
\begin{eqnarray*}
\frac{a_{P_{22}}}{4\pi i}\sum_{\alpha,\beta\in I_{P_{22}}}\int_{i\mathfrak{a}_{G}\backslash i\mathfrak{a}_{22}}\frac{\rm{exp}(<2\lambda, T>)}{<2\lambda, \alpha_2>}(M_{P_{22}}(s, -\lambda)\Phi_\alpha, \Phi_\beta)-(M_{P_{22}}(s, \lambda)\Phi_\alpha, \Phi_\beta)d\lambda
\end{eqnarray*}
and
\begin{eqnarray*}
+\frac{a_{P_{22}}}{4\pi i}\sum_{\alpha,\beta\in I_{P_{22}}}\int_{i\mathfrak{a}_{G}\backslash i\mathfrak{a}_{22}}\frac{\rm{exp}(<2\lambda, T>)-\rm{exp}(-<2\lambda, T>)}{<2\lambda, \alpha_2>}(M_{P_{22}}(s, \lambda)\Phi_\alpha, \Phi_\beta)d\lambda. 
\end{eqnarray*}
We have known that for every term above, the sum over $\beta$ is finite. And the first term approaches $0$ as $T\rightarrow\infty$ by the Riemann-Lebesgue lemma. The second term approaches 
\begin{eqnarray}\label{1015}
-\frac{a_{P_{22}}}{4}\rm{tr}\{M_{P_{22}}((13)(24), 0)\pi_{P_{22}}(0, f)\}. 
\end{eqnarray}
\begin{lemma}
    The sum \[I_{\rm{unram}}^{\mathfrak{o}_{22}^0}(f,x,T)+I_{\rm{ram}}^{\mathfrak{o}_{22}^2}(f,x,T)-K''_{P_{22}}(f,x,T)\] equals the sum of (\ref{102}),  (\ref{107}), (\ref{1014}), (\ref{p}) and (\ref{1015}).
\end{lemma}

\section{Terms associated to $P_{211}$}
\[\Omega(\mathfrak{a}_{211}, \mathfrak{a}_{211})=\{(1), (34)\}, \]
\[\Omega(\mathfrak{a}_{211}, \mathfrak{a}_{121})=\{(13), (134)\}, \]
\[\Omega(\mathfrak{a}_{211}, \mathfrak{a}_{112})=\{(13)(24), (14)(23)\}. \]
\subsection{The first parabolic term}
The first parabolic term is \[J_{\rm{unram}}^{\mathfrak{o}_{211}^0}(f, x, T)+J_{\rm{ram}}^{\mathfrak{o}_{211}^2}(f, x, T)-\sum_{P\in\mathfrak{P}_{211}}K'_{P}(f, x, T). \]
In this section, we shall prove that the first parabolic term associated to $\mathfrak{o}_{211}^0$ and $\mathfrak{o}_{211}^2$ approaches $0$ as $T\rightarrow\infty$.

Recall $J_{\rm{unram}}^{\mathfrak{o}_{211}^0}(f, x, T)$ is \[-\frac{1}{2}\sum_{\gamma\in \{M_{t,211}^{\mathfrak{o}_{211}^0}\}}(n_{\gamma, M_{211}})^{-1}\sum_{\delta\in M_{211}(\gamma)_{\mathbb{Q}}\backslash G_\mathbb{Q}}f(x^{-1}\delta^{-1}\gamma\delta x)(\sum_{P\in \mathfrak{P}_{211}}\sum_{s\in\Omega(\mathfrak{a}_{211}, \mathfrak{a})}\hat{\tau}_{P}(H_0(w_s\delta x)- T)).  \]
But the characteristic functions indexed by $P\supsetneq P_{211}$ have been borrowed according to we have done in the last two sections.
Thus, $J_{\rm{unram}}^{\mathfrak{o}_{211}^0}(f, x, T)$ equals (we still write it as $J_{\rm{unram}}^{\mathfrak{o}_{211}^0}(f, x, T)$)
\begin{eqnarray*}
    -\sum_{\gamma\in \{M_{t, 211}^{\mathfrak{o}_{211}^0}\}}(n_{\gamma, M_{211}})^{-1}\sum_{\delta\in M_{211}(\gamma)_{\mathbb{Q}}\backslash G_\mathbb{Q}}f(x^{-1}\delta^{-1}\gamma\delta x)(\hat{\tau}_{P_{211}}(H_0(\delta x)- T))\\
    -\sum_{\gamma\in \{M_{t, 121}^{\mathfrak{o}_{211}^0}\}}(n_{\gamma, M_{121}})^{-1}\sum_{\delta\in M_{121}(\gamma)_{\mathbb{Q}}\backslash G_\mathbb{Q}}f(x^{-1}\delta^{-1}\gamma\delta x)(\hat{\tau}_{P_{211}}(H_0(\delta x)- T))\\
    -\sum_{\gamma\in \{M_{t, 112}^{\mathfrak{o}_{211}^0}\}}(n_{\gamma, M_{112}})^{-1}\sum_{\delta\in M_{112}(\gamma)_{\mathbb{Q}}\backslash G_\mathbb{Q}}f(x^{-1}\delta^{-1}\gamma\delta x)(\hat{\tau}_{P_{112}}(H_0(\delta x)- T)). 
\end{eqnarray*}
Since for $\gamma\in{M_{t,211}^{\mathfrak{o}_{211}^0}}$, the group $N(\gamma_s)$ is trivial, by Lemma \ref{l1}, $J_{\rm{unram}}^{\mathfrak{o}_{211}^0}(f, x, T)$ is
\begin{eqnarray*}
&-\sum_{P\in \mathfrak{P}_{211}}\sum_{\gamma\in\{M_{t}^{\mathfrak{o}_{211}}\}}(n_{\gamma, M})^{-1}\sum_{\delta\in N_\mathbb{Q}M(\gamma)_{\mathbb{Q}}\backslash G_\mathbb{Q}}\sum_{v\in N_{\mathbb{Q}}}f(x^{-1}\delta^{-1}\gamma v\delta x)(\hat{\tau}_{P}(H_0(\delta x)-T)). 
\end{eqnarray*}
Which is 
\begin{eqnarray*}
&-\sum_{P\in \mathfrak{P}_{211}}\sum_{\delta\in P_{\mathbb{Q}}\backslash G_\mathbb{Q}}\sum_{\gamma\in M_{t}^{\mathfrak{o}_{211}^0}}\sum_{v\in N_{\mathbb{Q}}}f(x^{-1}\delta^{-1}\gamma v\delta x)(\hat{\tau}_{P}(H_0(\delta x)-T)). 
\end{eqnarray*}
By Lemma \ref{l1}, $J_{\rm{ram}}^{\mathfrak{o}_{211}^2}(f, x, T)$ equals 
\begin{eqnarray*}
    -\sum_{P\in \mathfrak{P}_{211}}\sum_{\delta\in P_{\mathbb{Q}}\backslash G_\mathbb{Q}}\sum_{\gamma\in M^{\mathfrak{o}_{211}^2}}\sum_{v\in N_{\mathbb{Q}}}f(x^{-1}\delta^{-1}\gamma v\delta x)(\hat{\tau}_{P}(H_0(\delta x)-T)). 
\end{eqnarray*}
We define $J_{P_{211}}(f, x, T)$ to be
\begin{eqnarray}\label{111}
&J_{\rm{unram}}^{\mathfrak{o}_{211}^0}(f, x, T)+J_{\rm{ram}}^{\mathfrak{o}_{211}^2}(f, x, T)\\\notag
&-\sum_{P\in \mathfrak{P}_{211}}\frac{1}{|\Omega(\mathfrak{a}_{1111}, P)|}\sum_{s\in\Omega(\mathfrak{a}_{1111};P)}\sum_{\delta\in P_{\mathbb{Q}}\backslash G_\mathbb{Q}}\sum_{\gamma\in M^{\mathfrak{o}_{1111}^0}}\sum_{v\in N_{\mathbb{Q}}}f(x^{-1}\delta^{-1}\gamma v\delta x)\\\notag
&(\hat{\tau}_{P}(H_0(w_s\delta x))-T)
\end{eqnarray}
\begin{eqnarray*}
&-\sum_{P\in \mathfrak{P}_{211}}\sum_{\mathfrak{o}\ ramified}\sum_{\delta\in P_{\mathbb{Q}}\backslash G_\mathbb{Q}}\sum_{\gamma\in M^\mathfrak{o}}\sum_{v\in N_{\mathbb{Q}}}f(x^{-1}\delta^{-1}\gamma v\delta x)\\
&\cdot(\hat{\tau}_{P}(H_0(\delta x))-T). 
\end{eqnarray*}
When we fix an unramified orbit $\mathfrak{o}_{1111}^0$, for $s_1\in\Omega(\mathfrak{a}_{1111};P)$,
\begin{eqnarray*}
    -\sum_{P\in \mathfrak{P}_{211}}\sum_{\delta\in P_{\mathbb{Q}}\backslash G_\mathbb{Q}}\sum_{\gamma\in M^{\mathfrak{o}}}\sum_{v\in N_{\mathbb{Q}}}f(x^{-1}\delta^{-1}\gamma v\delta x)(\hat{\tau}_{P}(H_0(w_s\delta x))-T)
\end{eqnarray*}
equals
\begin{eqnarray*}
    -\sum_{P\in \mathfrak{P}_{211}}\sum_{\delta\in P_{\mathbb{Q}}\backslash G_\mathbb{Q}}\sum_{\gamma\in M^\mathfrak{o}}\sum_{v\in N_{\mathbb{Q}}}f(x^{-1}\delta^{-1}\gamma v\delta x)(\hat{\tau}_{P}(H_0(\delta x))-T). 
\end{eqnarray*}
Since $M_{211}=\cup_{\mathfrak{o}} M_{211}^\mathfrak{o}$, $J_{P_{211}}(f, x, T)$ equals
\begin{eqnarray*}
&-\sum_{P\in\mathfrak{P}_{211}}\sum_{\delta\in P_{\mathbb{Q}}\backslash G_\mathbb{Q}}\sum_{\gamma\in M_{\mathbb{Q}}}\sum_{v\in N_{\mathbb{Q}}}f(x^{-1}\delta^{-1}\gamma v\delta x)(\hat{\tau}_{P}(H_0(\delta x)-T)),
\end{eqnarray*}
which is 
\begin{eqnarray*}
&-\sum_{P\in \mathfrak{P}_{211}}\sum_{\delta\in P_{\mathbb{Q}}\backslash G_\mathbb{Q}}\sum_{\gamma\in M}\int_{N_{\mathbb{A}}}f(x^{-1}\delta^{-1}\gamma v\delta x)dn(\hat{\tau}_{P}(H_0(\delta x)-T)). 
\end{eqnarray*}
Recall $\sum_{P\in\mathfrak{P}_{211}}K'_{P}(f, x, T)$ equals 
\begin{eqnarray*}
    &-\frac{1}{24(\pi i)^2}\sum_{P\in \mathfrak{P}_{211}}\sum_{P_1, P_2\in \mathfrak{P}_{211}}\sum_{P_{\mathbb{Q}}\backslash G_\mathbb{Q}}\sum_{\chi}\int_{i\mathfrak{a}_G\backslash i\mathfrak{a}}\\
    &\{\sum_{\beta\in \mathscr{B}_{P, \chi}}E^{c_{P_1}}_{P}(\pi_{P}(\lambda, f)\Phi_\beta, \lambda, \delta x)\overline{E^{c_{P_2}}_{P}(\Phi_\beta, \lambda, \delta x)}\}d\lambda\hat{\tau}_{P}(H_0(\delta x)-T). 
\end{eqnarray*}
This term is the sum of 
\begin{eqnarray*}
    &-\frac{1}{24(\pi i)^2}\sum_{P\in \mathfrak{P}_{211}}\sum_{P_1}\sum_{\gamma\in M_{\mathbb{Q}}}\sum_{\delta\in P_{\mathbb{Q}}\backslash G_\mathbb{Q}}\int_{N_{\mathbb{A}}}f(x^{-1}\delta^{-1}\gamma\delta x)dn\cdot \hat{\tau}_{P}(H_0(\delta x)-T)
\end{eqnarray*} 
and 
\begin{eqnarray*}
    -\frac{1}{24(\pi i)^2}\sum_{P_{211, \mathbb{Q}}\backslash G_\mathbb{Q}}\sum_{s\neq t}\sum_{\chi}\int_{i\mathfrak{a}_{G}\backslash i\mathfrak{a}_{211}}\{\sum_{\beta\in \mathscr{B}_{P_{211}, \chi}}(M_{P_{211}}(s, \lambda)\pi_{P_{211}}(\lambda, f)\Phi_\beta)(\delta x)\overline{M_{P_{211}}(t, \lambda)\Phi_\beta(\delta x)}\}\\
    \rm{exp}(<-2\lambda, H_0(\delta x)>)d\lambda\ \hat{\tau}_{P_{211}}(H_0(\delta x)-T)
\end{eqnarray*}
\begin{eqnarray*}
    -\frac{1}{24(\pi i)^2}\sum_{P_{121, \mathbb{Q}}\backslash G_\mathbb{Q}}\sum_{s\neq t}\sum_{\chi}\int_{i\mathfrak{a}_{G}\backslash i\mathfrak{a}_{121}}\{\sum_{\beta\in \mathscr{B}_{P_{121}, \chi}}(M_{P_{211}}(s, \lambda)\pi_{P_{211}}(\lambda, f)\Phi_\beta)(\delta x)\overline{M_{P_{211}}(t, \lambda)\Phi_\beta(\delta x)}\}\\
    \rm{exp}(<-2\lambda, H_0(\delta x)>)d\lambda\  \hat{\tau}_{P_{121}}(H_0(\delta x)-T)
\end{eqnarray*}
\begin{eqnarray*}
    -\frac{1}{24(\pi i)^2}\sum_{P_{112, \mathbb{Q}}\backslash G_\mathbb{Q}}\sum_{s\neq t}\sum_{\chi}\int_{i\mathfrak{a}_{G}\backslash i\mathfrak{a}_{112}}\{\sum_{\beta\in \mathscr{B}_{P_{112}}, \chi}(M_{P_{211}}(s, \lambda)\pi_{P_{112}}(\lambda, f)\Phi_\beta)(\delta x)\overline{M_{P_{211}}(t, \lambda)\Phi_\beta(\delta x)}\}\\
    \rm{exp}(<-2\lambda, H_0(\delta x)>)d\lambda\ \hat{\tau}_{P_{112}}(H_0(\delta x)-T), 
\end{eqnarray*}
but the last three terms' integrals over $Z_\infty^+G_\mathbb{Q}\backslash G_\mathbb{A}$ approach $0$ as $T\rightarrow\infty$ by Lemma \ref{l7}. Thus
\begin{lemma}
The sum \[J_{P_{211}}(f,x,T)-\sum_{P\in\mathfrak{P}_{211}}K'_{P}(f, x, T)\] approaches $0$ as $T\rightarrow\infty$.    
\end{lemma}

\subsection{The second parabolic term}
The second parabolic term correspond to $\mathfrak{o}_{211}^0$ and $\mathfrak{o}_{211}^2$ is \[I_{\rm{unram}}^{\mathfrak{o}_{211}^0}(f, x, T)+I_{\rm{ram}}^{\mathfrak{o}_{211}^2}(f, x, T). \]
In this section, we shall prove that the integral of this term is absolutely convergant.

Recall 
\begin{eqnarray*}
I_{\rm{unram}}^{\mathfrak{o}_{211}^0}(f, x, T)=\frac{1}{2}\sum_{\gamma\in \{M_{t,211}^{\mathfrak{o}_{211}^0}\}}(n_{\gamma, M_{211}})^{-1}\sum_{\delta\in M(\gamma)_{211, \mathbb{Q}}\backslash G_\mathbb{Q}}f(x^{-1}\delta^{-1}\gamma\delta x)\\
(1+\sum_{P\neq G}\sum_{s\in\Omega(\mathfrak{a}_{211}, P)}(-1)^{\rm{dim}\ Z\backslash A}\hat{\tau}_{P}(H_0(w_s\delta x)- T)). 
\end{eqnarray*}
The integral $\int_{Z_\infty^+G_\mathbb{Q}\backslash G_\mathbb{A}}|I_{\rm{unram}}^{\mathfrak{o}_{211}^0}(f, x, T)|dx$ is bounded by 
\begin{eqnarray*}
\frac{1}{2}\sum_{\gamma\in\{M_{t,211}^{\mathfrak{o}_{211}^0}\}}(n_{\gamma, M})^{-1}\int_{Z_\infty^+M(\gamma)_{211, \mathbb{Q}}\backslash G_\mathbb{A}}|f(x^{-1}\gamma x)|\\
\cdot(1+\sum_{P\neq G}\sum_{s\in\Omega(\mathfrak{a}_{211}, P)}(-1)^{\rm{dim}\ Z\backslash A}\hat{\tau}_{P}(H_0(w_s\delta x)- T))dx. 
\end{eqnarray*}
It equals
\begin{eqnarray*}
\frac{c_{P_{211}}}{2}\sum_{\gamma\in\{M_{t,211}^{\mathfrak{o}_{211}^0}\}}(n_{\gamma, M})^{-1}\int_{K}\int_{A_{211, \infty^+}M(\gamma)_{211, \mathbb{Q}}\backslash P_{211, \mathbb{A}}}\int_{Z_\infty^+\backslash A_{211, \infty}^+}|f(k^{-1}p^{-1}\gamma pk)|\\
\cdot(1+\sum_{P\neq G}\sum_{s\in\Omega(\mathfrak{a}_{211}, P)}(-1)^{\rm{dim}\ Z\backslash A}\hat{\tau}_{P}(H_0(w_s\delta x)- T))da\ d_lp\ dk. 
\end{eqnarray*}
Then the integral becomes
\begin{eqnarray*}
    \frac{c_{P_{211}}}{2}\sum_{\gamma\in\{M_{t,211}^{\mathfrak{o}_{211}^0}\}}\tilde{\tau}(\gamma, M)\int_{K}\int_{M(\gamma)_{211, \mathbb{A}}\backslash P_{211, \mathbb{A}}}|f(k^{-1}p^{-1}\gamma pk)|\\
    \cdot\int_{Z_{\infty}^+\backslash A_{211, \infty}^+}(1+\sum_{P\neq G}\sum_{s\in\Omega(\mathfrak{a}_{211}, P)}(-1)^{\rm{dim}\ Z\backslash A}\hat{\tau}_{P}(H_0(w_s\delta x)- T))da\ d_lp\ dk. 
\end{eqnarray*}
We have known the sum over $\gamma$ is finite by Lemma \ref{l9}. Since the function \[f^K(p)=\int_{K}f(k^{-1}pk)dk,\quad p\in P_{211,\mathbb{A}}\] has compact support, by Lemma \ref{l10}, the integral on $M(\gamma)_{211, \mathbb{A}}\backslash P_{211, \mathbb{A}}$ can be taken over a compact set. For any $p$, the function \[a\longrightarrow (1+\sum_{P\neq G}\sum_{s\in\Omega(\mathfrak{a}_{211}, P)}(-1)^{\rm{dim}\ Z\backslash A}\hat{\tau}_{P}(H_0(w_s\delta x)- T)),\quad a\in Z_\infty^+\backslash A_{211, \infty}^+, \] has compact suppoet.

$I_{\rm{unram}}^{\mathfrak{o}_{211}^0}(f, x, T)$ is integrable over $Z_\infty^+G_\mathbb{Q}\backslash G_\mathbb{A}$, its integral is
\begin{eqnarray*}
    \frac{c_{P_{211}}}{2}\sum_{\gamma\in\{M_{t,211}^{\mathfrak{o}_{211}^0}\}}\tilde{\tau}(\gamma, M)\int_{K}\int_{N_{211, \mathbb{A}}}\int_{M(\gamma)_{211, \mathbb{A}}\backslash M_{211, \mathbb{A}}}f(k^{-1}n^{-1}m^{-1}\gamma mnk)\\
    \cdot \int_{Z_\infty^+\backslash A_{211, \infty}^+}(1+\sum_{P\neq G}\sum_{s\in\Omega(\mathfrak{a}_{211}, P)}(-1)^{\rm{dim}\ Z\backslash A}\hat{\tau}_{P}(H_0(w_s\delta x)- T))da\ dm\ dn\ dk. 
\end{eqnarray*}
Then we can use the Arthur's $(G, M)$-family to see that the volume of \[\int_{Z_\infty^+\backslash A_{211, \infty}^+}(1+\sum_{P\neq G}\sum_{s\in\Omega(\mathfrak{a}_{211}, P)}(-1)^{\rm{dim}\ Z\backslash A}\hat{\tau}_{P}(H_0(w_s\delta x)- T))da\]is 
\[\frac{a_{P_{211}}}{2}\sum_{P\in P(A_{211})}\frac{<\lambda_0, T_P-H_P
(\delta x)>^2}{\Pi_{\eta\in\Phi_P}<\lambda, \eta>},\quad \lambda\in\mathfrak{a}_{G}\backslash \mathfrak{a}_{211}, \]
where $T_P$ and $H_P(\delta x)$ is the projection of $T$ and $H_0(\delta x)$ to $\mathfrak{a}_G\backslash \mathfrak{a}_{211}$, this sum is independent of $\lambda_0$.

The integral is
\begin{eqnarray}\label{convex}
    \sum_{P\in\mathfrak{P}_{211}}\frac{c_{P}}{2}\cdot a_{P}\sum_{\gamma\in\{M_t^{\mathfrak{o}_{211}}\}}\tilde{\tau}(\gamma, M)\int_{K}\int_{N_{\mathbb{A}}}\int_{M(\gamma)_{\mathbb{A}}\backslash M_{\mathbb{A}}}f(k^{-1}n^{-1}m^{-1}\gamma mnk)\\\notag
    \cdot v_{M}(x, T)dm\ dn\ dk. 
\end{eqnarray}
According to \cite{A6}, we write $v_{M_{211}}(x, T)=(cd)_{M_{211}}$ where
\begin{eqnarray*}
    &c_{M_{211}}=\rm{lim}_{\lambda\rightarrow 0}\sum_{P\in P(A_{211})}\frac{\rm{exp}(<\lambda, X_p>)}{\Pi_{\eta\in\Phi_P}<\lambda, \eta>}, 
\end{eqnarray*}
\begin{eqnarray*}
    &d_{M_{211}}(\lambda)=\rm{lim}_{\lambda\rightarrow 0}\sum_{P\in P(A_{211})}\frac{\rm{exp}(<\lambda, Y_p>)}{\Pi_{\eta\in\Phi_P}<\lambda, \eta>},\quad X_P=-H_P(x), Y_P=T_P,
\end{eqnarray*}
Similarly for other $c_M$ and $d_M$.

By \cite[Cor 6.5]{A6}, we can write $(cd)_{M_{211}}$ as
\begin{eqnarray*}
    2c_{M_{211}}^{M_{31}}d_{M_{31}}+2c_{M_{211}}^{M_{22}}d_{M_{22}}+c_{M_{211}}^{M_{211}}d_{M_{211}}+c_{M_{211}}^Gd_G, 
\end{eqnarray*}
where $c_{M_{211}}^{M_{211}}=d_G=1. $ Note that the Levi in \cite[Cor 6.5]{A6} is not necessarily standard, that is why we multiple $2$ in front of that.

We put them into $I_{\rm{unram}}^{\mathfrak{o}_{211}^0}(f, x, T)$.

$d_{M_{211}}$ corresponds to 
\begin{eqnarray}\label{112}
    &\sum_{P\in\mathfrak{P}_{211}}\sum_{s\in\Omega(\mathfrak{a}_{211}, \mathfrak{a})}\frac{<s\lambda_0, T>^2}{\Pi_{\eta\in\Phi_P}<s\lambda_0, \eta>}\frac{c_{P}}{4}\cdot a_{P}\sum_{\gamma\in\{M_{t}^{\mathfrak{o}_{211}^0}\}}\tilde{\tau}(\gamma, M)
\end{eqnarray}
\begin{eqnarray*}
    &\int_{K}\int_{N_{\mathbb{A}}}\int_{M(\gamma)_{\mathbb{A}}\backslash M_{\mathbb{A}}}f(k^{-1}n^{-1}m^{-1}\gamma mnk)dm\ dn\ dk. 
\end{eqnarray*}
$c_{M_{211}}^G$ correspond to 
\begin{eqnarray}\label{113}
    &+\sum_{P\in\mathfrak{P}_{211}}\sum_{s\in\Omega(\mathfrak{a}_{211}, \mathfrak{a})}\frac{c_{P_{211}}}{4}\cdot a_{P}\sum_{\gamma\in \{M_{t}^{\mathfrak{o}_{211}^0}\}}\tilde{\tau}(\gamma, M)
\end{eqnarray}
\begin{eqnarray*}
    &\int_{K}\int_{N_{\mathbb{A}}}\int_{M(\gamma)_{\mathbb{A}}\backslash M_{\mathbb{A}}}f(k^{-1}n^{-1}m^{-1}\gamma mnk)\frac{<\lambda_0, -s^{-1}H_0(mnk)>^2}{\Pi_{\eta\in\Phi_P}<\lambda_0, s^{-1}\eta>}dm\ dn\ dk. 
\end{eqnarray*}
Now, we deal with the remaining two terms. We observe that $d_{M_{31}}$ and $d_{M_{22}}$ is $\hat{\alpha}_3(T)+\hat{\alpha}_1(T)$ and $2\hat{\alpha}_2(T)$. Then we can write the integrals corresponding to them as the integrals associated to $P_{31}, P_{22}$ by Lemma \ref{l6.6}.

$c_{M_{211}}^{M_{31}}d_{M_{31}}$ corresponds to
\begin{eqnarray*}
    (\hat{\alpha}_3(T)+\hat{\alpha}_1(T))\frac{c_{P_{211}}}{2}a_{P_{211}}\sum_{\gamma\in \{M_{t,211}^{\mathfrak{o}^0_{211}}\}}(n_{\gamma, M})^{-1}\int_K\int_{N_{211, \mathbb{A}}}\int_{A_{211, \infty}^+M(\gamma)_{211, \mathbb{Q}}\backslash M_{211, \mathbb{A}}}\\
    f(k^{-1}n^{-1}m^{-1}\gamma mnk)\int_{A_{31, \infty}^+\backslash A_{211, \infty}^+}1-\hat{\tau}_{P_{22}}(H_0(amnk))-\hat{\tau}_{P_{22}}(H_0(w_{(13)}amnk))da\ dm\ dn\ dk
\end{eqnarray*}
Which is the sum of
\begin{eqnarray*}\label{114}
    \hat{\alpha}_3(T)\frac{c_{P_{31}}}{2}a_{P_{31}}\sum_{\gamma\in \{M_{t,31}^{\mathfrak{o}_{211}^0}\}}(n_{\gamma, M})^{-1}\int_K\int_{N_{31, \mathbb{A}}}\int_{A_{211, \infty}^+M(\gamma)_{31, \mathbb{Q}}\backslash M_{31, \mathbb{A}}}\\
    f(k^{-1}n^{-1}m^{-1}\gamma mnk)\int_{A_{31,\infty}^+\backslash A_{211,\infty}^+}(1-\hat{\tau}_{P_{22}}(H_0(amnk))-\hat{\tau}_{P_{22}}(H_0(w_{(13)}amnk)))da\ dm\ dn\ dk
\end{eqnarray*} 
and
\begin{eqnarray*}\label{1141}
    \hat{\alpha}_1(T)\frac{c_{P_{13}}}{2}a_{P_{13}}\sum_{\gamma\in \{M_{t,13}^{\mathfrak{o}_{211}^0}\}}(n_{\gamma, M})^{-1}\int_K\int_{N_{13, \mathbb{A}}}\int_{A_{211, \infty}^+M(\gamma)_{13, \mathbb{Q}}\backslash M_{13, \mathbb{A}}}\\
    f(k^{-1}n^{-1}m^{-1}\gamma mnk)\int_{A_{13,\infty}^+\backslash A_{211,\infty}^+}(1-\hat{\tau}^{P_{211}}_{P_{31}}(H_0(amnk))-\hat{\tau}_{P_{31}}^{P_{211}}(H_0(w_{(13)}amnk)))da\ dm\ dn\ dk.
\end{eqnarray*} 
For the term (\ref{114}), we write \[w_{13}amnk\] as \[w_{13}aw_{13}^{-1}\cdot w_{13}mnk,\] where $w_{13}m\in M^1_{31}$, we can see the integral over $A_{31,\infty}^+\backslash A_{211,\infty}^+$ equals zero. Then this term equals zero. Others are similar.

The terms associated to $c_{M_{211}}^{M_{31}}d_{M_{31}}+c_{M_{211}}^{M_{22}}d_{M_{22}}$ in (\ref{convex}) are all zero.

Change the variable of integration on $N_{211, \mathbb{A}}$ of (\ref{112}), apply Lemma \ref{l2}, then (\ref{convex}) becomes 
\begin{eqnarray*}
    \sum_{P\in\mathfrak{P}_{211}}\sum_{s\in\Omega(\mathfrak{a}_{211}, \mathfrak{a})}\frac{<\lambda_0, s^{-1}T>^2}{\Pi_{\eta\in\Phi_P}<\lambda_0, s^{-1}\eta>}\frac{c_{P}}{4}\cdot a_P\sum_{\gamma\in \{M_{t}^{\mathfrak{o}_{211}^0}\}}\tilde{\tau}(\gamma, M)\\
    \int_{K}\int_{M(\gamma)_{\mathbb{A}}\backslash M_{\mathbb{A}}}\int_{N_{\mathbb{A}}}f(k^{-1}m^{-1}\gamma nmk)\rm{exp}(-<2\rho_{P}, H_0(m)>)dn\ dm\ dk. 
\end{eqnarray*}
This term equals 
\begin{eqnarray*}
    \sum_{P\in\mathfrak{P}_{211}}\sum_{s\in\Omega(\mathfrak{a}_{211}, \mathfrak{a})}\frac{<\lambda_0, s^{-1}T>^2}{\Pi_{\eta\in\Phi_P}<\lambda_0, s^{-1}\eta>}\frac{c_{P}}{4}\cdot a_{P}\sum_{\gamma\in \{M_{t}^{\mathfrak{o}_{211}^0}\}}(n_{\gamma, M})^{-1}\\\notag
    \int_{K}\int_{A_{\infty}^+M(\gamma)_{\mathbb{A}}\backslash M_{\mathbb{A}}}\int_{N_{\mathbb{A}}}f(k^{-1}m^{-1}\gamma nmk)\rm{exp}(-<2\rho_{P}, H_0(m)>)dn\ dm\ dk, 
\end{eqnarray*}
that is 
\begin{eqnarray}
    \sum_{P\in\mathfrak{P}_{211}}\sum_{s\in\Omega(\mathfrak{a}_{211}, \mathfrak{a})}\frac{<\lambda_0, s^{-1}T>^2}{\Pi_{\eta\in\Phi_P}<\lambda_0, s^{-1}\eta>}\frac{c_{P}}{4}\cdot a_{P}\sum_{\gamma\in M_{t}^{\mathfrak{o}_{211}^0}}\\\notag
    \int_{K}\int_{N_{\mathbb{A}}}\int_{A_{\infty}^+M_{\mathbb{Q}}\backslash M_{\mathbb{A}}}f(k^{-1}m^{-1}\gamma nmk)\rm{exp}(-<2\rho_{P}, H_0(m)>)dn\ dm\ dk. 
\end{eqnarray}
Recall $I_{\rm{ram}}^{\mathfrak{o}_{211}^2}(f, x, T)$ equals 
\begin{eqnarray*}
    &\sum_{\gamma\in G^{\mathfrak{o}_{211}^2}}f(x^{-1}\gamma x)\\
    &-\sum_{\delta\in M_{22, \mathbb{Q}}\backslash G_\mathbb{Q}}\sum_{\gamma\in M_{t,22}^{\mathfrak{o}_{211}^2}}f(x^{-1}\delta^{-1}\gamma\delta x)\hat{\tau}_{P_{22}}(H_0(\delta x)-T)\\
    &-\sum_{\delta\in M_{31, \mathbb{Q}}N(\gamma_s)_{31, \mathbb{Q}}\backslash G_\mathbb{Q}}\sum_{\gamma\in M_{n,31}^{\mathfrak{o}_{211}^2}}\sum_{v\in N(\gamma_s)_{31, \mathbb{Q}}}f(x^{-1}\delta^{-1}\gamma v\delta x)\hat{\tau}_{P_{31}}(H_0(\delta x)-T)\\
    &-\sum_{\delta\in M_{13, \mathbb{Q}}N(\gamma_s)_{13, \mathbb{Q}}\backslash G_\mathbb{Q}}\sum_{\gamma\in M_{n,13}^{\mathfrak{o}_{211}^2}}\sum_{v\in N(\gamma_s)_{31, \mathbb{Q}}}f(x^{-1}\delta^{-1}\gamma v\delta x)\hat{\tau}_{P_{13}}(H_0(\delta x)-T)\\
    &+\sum_{P\in\mathfrak{P}_{211}}\sum_{\delta\in M_{\mathbb{Q}}N(\gamma_s)_{\mathbb{Q}}\backslash G_\mathbb{Q}}\sum_{\gamma\in M_n^{\mathfrak{o}_{211}^2}}\sum_{v\in N(\gamma_s)_{\mathbb{Q}}}f(x^{-1}\delta^{-1}\gamma v\delta x)\hat{\tau}_{P}(H_0(\delta x)-T). 
\end{eqnarray*}
The integral of $I_{\rm{ram}}^{\mathfrak{o}_{211}^2}(f, x, T)$ over $Z_\infty^+G_\mathbb{Q}\backslash G_\mathbb{A}$ equals the sum of
\begin{eqnarray}
    &+\sum_{P\in\mathfrak{P}_{211}}\frac{c_{P}}{2}\cdot a_{P}\frac{<\lambda_0, T>^2}{\Pi_{\eta\in\Phi_P}<\lambda_0, \eta>}\sum_{\gamma\in \{M_{n}^{\mathfrak{o}_{211}^2}\}}(n_{\gamma, M})^{-1}\notag\\\label{1114}
    &\int_{K}\int_{N_{\mathbb{A}}}\int_{A_{\infty}^+M(\gamma)_{\mathbb{Q}}\backslash M_{211, \mathbb{A}}}f(k^{-1}m^{-1}\gamma nmk)\rm{exp}(-<2\rho_{P}, H_0(m)>)dn\ dm\ dk
\end{eqnarray}
and 
\begin{eqnarray}
    &+\hat{\alpha}_3(T)c_{P_{31}}a_{P_{31}}\int_K\int_{A_{31, \infty}^+M_{31, \mathbb{Q}}\backslash M_{31, \mathbb{A}}} \sum_{\gamma\in M_{n, 31}^{\mathfrak{o}_{211}^2}}\int_{N(\gamma)_{1, \mathbb{A}}}\\\notag
    &f(k^{-1}m^{-1}\gamma nmk)\rm{exp}(-<2\rho_{P_{31}}, H_0(m)>)dn\ dm\ dk\label{1115}\\
    &+\hat{\alpha}_1(T)c_{P_{13}}a_{P_{13}}\int_K\int_{A_{13, \infty}^+M_{13, \mathbb{Q}}\backslash M_{13, \mathbb{A}}} \sum_{\gamma\in M_{n, 13}^{\mathfrak{o}_{211}^2}}\int_{N(\gamma)_{1, \mathbb{A}}}\\\notag
    &f(k^{-1}m^{-1}\gamma nmk)\rm{exp}(-<2\rho_{P_{13}}, H_0(m)>)dn\ dm\ dk\label{1116}\\
    &+\hat{\alpha}_2(T)c_{P_{22}}a_{P_{22}}\int_K\int_{A_{22, \infty}^+M_{22, \mathbb{Q}}\backslash M_{22, \mathbb{A}}} \sum_{\gamma\in M_{n,22}^{\mathfrak{o}_{211}^2}}\int_{N(\gamma)_{22, \mathbb{A}}}\\\notag
    &f(k^{-1}m^{-1}\gamma nmk)\rm{exp}(-<2\rho_{P_{22}}, H_0(m)>)dn\ dm\ dk\label{1117}\\\label{o}
    &+c_{P_{22}}a_{P_{22}}\int_K\int_{A_{22, \infty}^+M_{22, \mathbb{Q}}\backslash M_{22, \mathbb{A}}} \sum_{\gamma\in M_{t,22}^{\mathfrak{o}_{211}^2}}\int_{N_{22, \mathbb{A}}}\\\notag
    &f(k^{-1}m^{-1}\gamma m^{-1}n{-1}k)\hat{\alpha}_2(H_0(w_{(13)(24)}n))dn\ dm\ dk,
\end{eqnarray}
the term (\ref{o}) is obtained by Lemma \ref{l82}.
\subsection{The third parabolic term}
This integral of $-\sum_{P\in\mathfrak{P}_{211}}K_{P}''(f, x, T)$ is 
\begin{eqnarray*}
    -\frac{1}{24(\pi i)^2}\sum_{P\in\mathfrak{P}_{211}}\sum_{\alpha, \beta\in I_{P}}\int_{i\mathfrak{a}_{G}\backslash i\mathfrak{a}_{P}}\int_{Z_\infty^+G_\mathbb{Q}\backslash G_\mathbb{A}}E''^{T}_{P}(\Phi_\alpha, \lambda, x)\overline{E''^{T}_{P}(\Phi_\beta, \lambda, x)}dx\ d\lambda. 
\end{eqnarray*}
\begin{lemma}\label{lemma11.1}
    For $\alpha, \beta\in I_{P_{211}}$ and $\lambda$ a nonzero imaginary number in $i\mathfrak{a}_G\backslash i\mathfrak{a}$, the integral \[\int_{Z_\infty^+G_\mathbb{Q}\backslash G_\mathbb{A}}E''^{T}_{P}(\Phi_\alpha, \lambda, x)\overline{E''^{T}_{P}(\Phi_\beta, \lambda, x)}dx\] is
    \begin{eqnarray}
        &\frac{a_{P_{211}}}{2}\sum_{P\in\mathfrak{P}_{211}}\sum_{s\in \Omega(\mathfrak{a_{211}, \mathfrak{a}})}\frac{<\lambda_0, s^{-1}T>^2}{\Pi_{\eta\in\Phi_P}<\lambda_0, s^{-1}\eta>}(\Phi_\alpha, \Phi_\beta)\label{1119}\\
        &+\frac{a_{P_{211}}}{2}\sum_{P\in\mathfrak{P}_{211}}\sum_{t\in\Omega(\mathfrak{a}_{211}, \mathfrak{a})}(M_P(t^{-1}, t\lambda)D_\lambda M_P(t, \lambda)\Phi_\alpha, \Phi_\beta)\label{1120}\\
        &+\frac{a_{P_{211}}}{2}\sum_{P\in\mathfrak{P}_{211}}\sum_{s\neq t\in\Omega(\mathfrak{a}_{211}, \mathfrak{a})}\frac{\rm{exp}(<t\lambda-s\lambda, T>)(M_P(t, \lambda)\Phi_\alpha, M_P(s, \lambda)\Phi_\beta)}{\Pi_{\eta\in \Phi_P}<t\lambda-s\lambda, \eta>}. \label{1121}
    \end{eqnarray}
\end{lemma}
\begin{proof}
    Suppose that $\lambda_1, \overline{\lambda}$ are different complex numbers in $i\mathfrak{a}_G\backslash i\mathfrak{a}_{211}$, whose real parts are suitably regular. Then, 
    \begin{eqnarray*}
    &\int_{Z_\infty^+G_\mathbb{Q}\backslash G_\mathbb{A}}E''^{T}_{P_{211}}(\Phi_\alpha, \lambda_1, x)\overline{E''^{T}_{P_{211}}(\Phi_\beta, \lambda, x)}dx\\
    &=\frac{a_{P_{211}}}{2}\sum_{P\in \mathfrak{P}_{211}}\sum_{t\in\Omega(\mathfrak{a}_{211}, \mathfrak{a})}\sum_{s\in \Omega(\mathfrak{a}_{211}, a)}\frac{\rm{exp}(<t\lambda_1+s\overline{\lambda}, T>)}{\Pi_{\eta\in\Phi_{P}}<t\lambda_1+s\overline{\lambda}, \eta>}(M_P(t, \lambda_1)\Phi_\alpha, M_P(s, \lambda)\Phi_\beta). 
    \end{eqnarray*}
   This function is meromorphic in $\lambda_1, \lambda$. Set $\lambda_1-\lambda=a\lambda_0$, then we will let this term be the limit as $a$ approaches $0$.

    We decompose it into two cases: $t=s$ and $t\neq s$.

    We deal with the term of $t=s$ by applying L'Hopital's rule twice. The result is (\ref{1119}) and (\ref{1120}).

    When $t\neq s$, directly let $a$ approaches $0$, we can obtain (\ref{1121}). 
\end{proof}
The terms associated to other $P\in\mathfrak{P}_{211}$ are similar.

The term correspond to (\ref{1120}) is
\begin{eqnarray}\label{1122}
    &-\frac{a_{P_{211}}}{48(\pi i)^2}\sum_{\chi}\int_{i\mathfrak{a}_{G}\backslash \mathfrak{a}_{211}}\rm{tr}\{M_{P_{211}}((34), (34)\lambda)\cdot(D_\lambda M_{P_{211}}((34), \lambda))\cdot\pi_{P_{211}, \chi}(\lambda, f)\}d\lambda\\\notag
    &-\frac{a_{P_{211}}}{48(\pi i)^2}\sum_{\chi}\int_{i\mathfrak{a}_{G}\backslash \mathfrak{a}_{211}}\rm{tr}\{M_{P_{211}}((143), (134)\lambda)\cdot(D_\lambda M_{P_{211}}((134), \lambda))\cdot\pi_{P_{211}, \chi}(\lambda, f)\}d\lambda\\\notag
    &-\frac{a_{P_{211}}}{48(\pi i)^2}\sum_{\chi}\int_{i\mathfrak{a}_{G}\backslash \mathfrak{a}_{211}}\rm{tr}\{M_{P_{211}}((14)(23), (14)(23)\lambda)\cdot(D_\lambda M_{P_{211}}((14)(23), \lambda))\cdot\pi_{P_{211}, \chi}(\lambda, f)\}d\lambda. 
\end{eqnarray}
This term is finite.

(\ref{1119}) can be written as
\begin{eqnarray*}
    \frac{1}{2}\sum_{P_1\in\mathfrak{P}_{211}}\sum_{P\in\mathfrak{P}_{211}}a_{P}\sum_{s\in\Omega(\mathfrak{a}_{1}, \mathfrak{a})}\frac{<s\lambda_0, T>^2}{\Pi_{\eta\in\Phi_P}<s\lambda_0, \eta>}\int_{i\mathfrak{a}_{G}\backslash \mathfrak{a}}\rm{tr}\,\pi_{P}(\lambda, f)d\lambda. 
\end{eqnarray*}
We substitute this term into $\sum_{P\in\mathfrak{P}_{211}}K_{P}''(f, x, T)$, it equals
\begin{eqnarray*}
    \frac{1}{16(\pi i)^2}\sum_{P\in\mathfrak{P}_{211}}c_{P}a_P\sum_{s\in\Omega(\mathfrak{a}_{211}, \mathfrak{a})}\frac{<s\lambda_0, T>^2}{\Pi_{\eta\in\Phi_P}<s\lambda_0, \eta>}\int_{i\mathfrak{a}_{G}\backslash \mathfrak{a}}\int_{A_{\infty}^+M_{\mathbb{Q}}\backslash M_{\mathbb{A}}}\\
    P_{P}(\lambda, f, mk, mk)dm\ dk\ d\lambda, 
\end{eqnarray*}by the continity of $P_{P}$.

Then applying the Fourier inversion formula, we obtain
\begin{eqnarray*}
    \sum_{P\in\mathfrak{P}_{211}} \frac{c_{P}}{16(\pi i)^2}a_{P}\sum_{s\in\Omega(\mathfrak{a}_{211}, \mathfrak{a})}\frac{<s\lambda_0, T>^2}{\Pi_{\eta\in\Phi_P}<s\lambda_0, \eta>}\\
    \int_K\int_{A_{\infty}^+M_{\mathbb{Q}}\backslash M_{\mathbb{A}}}\sum_{\gamma\in M_{\mathbb{Q}}}\int_{N_{\mathbb{A}}}f(k^{-1}m^{-1}\gamma n mk)\rm{exp}(-<2\rho_{P}, H_0(m)>)dn\ dm\ dk. 
\end{eqnarray*}
The terms (\ref{112}), (\ref{1114}), and (\ref{1116}) can be canceled, but there is also something left, 
\begin{eqnarray}\label{211}
    \sum_{\rm{ramified}\,\mathfrak{o}}\sum_{P\in\mathfrak{P}_{211}}\frac{a_{P}}{16(\pi i)^2}c_{P}\sum_{s\in\Omega(\mathfrak{a}_{211}, \mathfrak{a})}\frac{<s\lambda_0, T>^2}{\Pi_{\eta\in\Phi_P}<s\lambda_0, \eta>}\\\notag
    \int_K\int_{A_{\infty}^+M_{\mathbb{Q}}\backslash M_{\mathbb{A}}}\sum_{\substack{\gamma\in M^{\mathfrak{o}}}}\int_{N_{\mathbb{A}}}f(k^{-1}m^{-1}\gamma n mk)\rm{exp}(-<2\rho_{P}, H_0(m)>)dn\ dm\ dk. 
\end{eqnarray}
The integral of other orbits in $P\in\mathfrak{P}_{211}$ are left.

Consider the term (\ref{1121}). We put it into the function $-\sum_{P\in\mathfrak{P}_{211}}K''_{P}(f, x, T)$. Write it as 
\begin{eqnarray*}
    \sum_{P\in\mathfrak{P}_{211}}\frac{a_{P}}{16(\pi i)^2}\sum_{\alpha, \beta\in I_{P}}\int_{i\mathfrak{a}_{G}\backslash \mathfrak{a}}\frac{\rm{exp}(<t\lambda-s\lambda, T>(M_P(t, \lambda)\Phi_\alpha), M_P(s, \lambda)\Phi_\beta)}{\Pi_{\eta\in \Phi_P}<t\lambda-s\lambda, \eta>}d\lambda. 
\end{eqnarray*}
For every term above, the sum over $\beta$ is finite.

Define $\iota_P(s, t)=$
\[\sum_{\alpha, \beta\in I_{P}}\int_{i\mathfrak{a}_{G}\backslash \mathfrak{a}}\frac{\rm{exp}(<t\lambda-s\lambda, T>(M_P(t, \lambda)\Phi_\alpha), M_P(s, \lambda)\Phi_\beta)}{\Pi_{\eta\in \Phi_P}<t\lambda-s\lambda, \eta>}d\lambda. \]
To obtain the value of one $\iota_P(s, t)$, we need to decompose the integral into the integral over the lines where the dual simple roots lie. Then we write $\lambda=\sum_{i=1}^2a_k\hat{\alpha}_{i_k}$, we just need to calculation the residue at $(0, 0)$. Of course the result has no $T$. 
Thus (\ref{1121}) is 
\begin{eqnarray}\label{1123}
-\sum_{P\in\mathfrak{P}_{211}}\frac{a_{P}}{16(\pi i)^2}\sum_{s\neq t\in\Omega(\mathfrak{a}_{\mathfrak{o}}, \mathfrak{a})}\iota_{P_{211}}(s, t). 
\end{eqnarray}
\begin{lemma}
    The sum \[I_{\rm{unram}}^{\mathfrak{o}_{211}^0}(f,x,T)+I_{\rm{ram}}^{\mathfrak{o}_{211}^2}(f,x,T)-\sum_{P\in\mathfrak{P}_{211}}K''_P(f,x,T)\] is the sum of (\ref{113}), (\ref{1122}), (\ref{211}) and (\ref{1123}).
\end{lemma}
\section{Terms associated to $P_{1111}$}
\[\Omega(\mathfrak{a}_{1111}, \mathfrak{a}_{1111})=S_4. \]
\subsection{The first parabolic term}
The first parabolic term is \[J_{\rm{unram}}^{\mathfrak{o}_{1111}^0}(f, x, T)+\sum_{k}J_{\rm{ram}}^{\mathfrak{o}_{1111}^k}(f, x, T)-K'_{P_{1111}}(f, x, T),\quad k\in\{31, 22, 211, 4\}. \]
In this section, we shall prove that this term approaches $0$ as $T\rightarrow\infty$.

Recall $J_{\rm{unram}}^{\mathfrak{o}_{1111}^0}(f, x, T)$ equals
\begin{eqnarray*}
\frac{1}{24}\sum_{\gamma\in \{M_{t,1111}^{\mathfrak{o}_{1111}^0}\}}(n_{\gamma, M_{1111}})^{-1}\sum_{\delta\in M_{1111}(\gamma)_{\mathbb{Q}}\backslash G_\mathbb{Q}}f(x^{-1}\delta^{-1}\gamma\delta x)(\sum_{s\in\Omega(\mathfrak{a}_{1111}, \mathfrak{a}_{1111})}\hat{\tau}_{P_{1111}}(H_0(w_s\delta x)- T)).  
\end{eqnarray*}
Then,
\[J_{\rm{unram}}^{\mathfrak{o}_{1111}^0}(f, x, T)=\sum_{\gamma\in \{M_{t,1111}^{\mathfrak{o}_{1111}^0}\}}(n_{\gamma, M_{1111}})^{-1}\sum_{\delta\in M_{1111}(\gamma)_{\mathbb{Q}}\backslash G_\mathbb{Q}}f(x^{-1}\delta^{-1}\gamma\delta x)\hat{\tau}_{P_{1111}}(H_0(\delta x)- T). \]
Since for $\gamma\in{M_{1111}^{\mathfrak{o}_{1111}^0}}$, the group $N(\gamma_s)$ is trivial, by Lemma \ref{l1}, this term is
\begin{eqnarray*}
&\sum_{\gamma\in\{M_{t,1111}^{\mathfrak{o}_{1111}^0}\}}(n_{\gamma, M_{1111}})^{-1}\sum_{\delta\in N_{1111, \mathbb{Q}}M(\gamma)_{1111, \mathbb{Q}}\backslash G_\mathbb{Q}}\sum_{v\in N_{1111, \mathbb{Q}}}f(x^{-1}\delta^{-1}\gamma v\delta x)(\hat{\tau}_{P_{1111}}(H_0(\delta x)-T)). 
\end{eqnarray*}
Which is 
\begin{eqnarray}\label{121}
&\sum_{\delta\in P_{1111, \mathbb{Q}}\backslash G_\mathbb{Q}}\sum_{\gamma\in M_{t,1111}^{\mathfrak{o}_{1111}^0}}\sum_{v\in N_{1111, \mathbb{Q}}}f(x^{-1}\delta^{-1}\gamma v\delta x)(\hat{\tau}_{P_{1111}}(H_0(\delta x)-T)). 
\end{eqnarray}
Since the term associated to $\hat{\tau}_P$ where $P\neq P_{1111}$ has been borrowed, $\sum_{k}J_{\rm{ram}}^{\mathfrak{o}^k_{1111}}(f, x, T)$ equals 
\begin{eqnarray}\label{122}
    \sum_k\sum_{\delta\in M_{1111, \mathbb{Q}}N(\gamma_s)_{1111, \mathbb{Q}}\backslash G_\mathbb{Q}}\sum_{\gamma\in M_{1111}^{\mathfrak{o}_{1111}^k}}\sum_{v\in N(\gamma_s)_{1111, \mathbb{Q}}}f(x^{-1}\delta^{-1}\gamma v\delta x)(\hat{\tau}_{P_{1111}}(H_0(\delta x)-T)). 
\end{eqnarray}
Then $J_{P_{1111}}(f, x, T)$ equals 
\begin{eqnarray*}
    \sum_{\delta\in P_{1111, \mathbb{Q}}\backslash G_\mathbb{Q}}\sum_{\gamma\in M_{1111}}\int_{N_{1111, \mathbb{A}}}f(x^{-1}\delta^{-1}\gamma n\delta x)dn(\hat{\tau}_{P_{1111}}(H_0(\delta x)-T)). 
\end{eqnarray*}
Recall $K'_{P_{1111}}(f, x, T)$ equals 
\begin{eqnarray*}
    \frac{1}{192(\pi i)^3}\sum_{P_{1111, \mathbb{Q}}\backslash G_\mathbb{Q}}\sum_{\chi}\int_{i\mathfrak{a}_G\backslash i\mathfrak{a}_{1111}}\sum_{\alpha,\beta\in \mathscr{B}_{P, \chi}}E^{c_{1111}}_{P_{1111}}(\Phi_\alpha, \lambda, \delta x)\overline{E^{c_{1111}}_{P_{1111}}(\Phi_\beta, \lambda, \delta x)}d\lambda\hat{\tau}_{P_{1111}}(H_0(\delta x)-T). 
\end{eqnarray*}
It is the sum of 
\begin{eqnarray*}
    &\sum_{\gamma\in M_{1111}}\sum_{\delta\in P_{1111, \mathbb{Q}}\backslash G_\mathbb{Q}}\int_{N_{1111, \mathbb{A}}}f(x^{-1}\delta^{-1}\gamma n\delta x)dn\cdot \hat{\tau}_{P_{1111}}(H_0(\delta x)-T)
\end{eqnarray*} 
and 
\begin{eqnarray*}
    \frac{1}{192(\pi i)^3}\sum_{P_{1111, \mathbb{Q}}\backslash G_\mathbb{Q}}\sum_{s\neq t\in\Omega(\mathfrak{a}_{1111}, \mathfrak{a}_{1111})}\sum_{\chi}\int_{i\mathfrak{a}_{G}\backslash\mathfrak{a}_{1111}}\{\sum_{\beta\in \mathscr{B}_{P_{1111}, \chi}}(M_{P_{1111}}(s, \lambda)\pi_{P_{1111}}(\lambda, f)\Phi_\beta)(\delta x)\\
    \overline{M_{P_{1111}}(t,\lambda)\Phi_\beta(\delta x)}\}\rm{exp}(<-2\lambda, H_0(\delta x)>)d\lambda\ \hat{\tau}_{P_{1111}}(H_0(\delta x)-T), 
\end{eqnarray*}
the second function's integral over $Z_\infty^+G_\mathbb{Q}\backslash G_\mathbb{A}$ approaches $0$ as $T\rightarrow\infty$ by Lemma \ref{l7}. Thus,
\begin{lemma}
    The sum \[J_{P_{1111}}(f,x,T)-K'_{P_{1111}}(f,x,T)\] approaches $0$ as $T\rightarrow \infty$.
\end{lemma}
\subsection{The second parabolic term}
The second parabolic term  is
\[I_{\rm{unram}}^{\mathfrak{o}_{1111}^0}(f, x, T)+\sum_{k}I_{\rm{ram}}^{\mathfrak{o}_{1111}^k}(f, x, T). \]
In this section we shall prove that the integrl of this term is absolutely convergant.

Recall $I_{\rm{unram}}^{\mathfrak{o}^0_{1111}}(f, x, T)$ equals
\begin{eqnarray*}
\frac{1}{24}\sum_{\gamma\in \{M_{t, 1111}^{\mathfrak{o}_{1111}}\}}(n_{\gamma, M_{1111}})^{-1}\sum_{\delta\in M(\gamma)_{1111, \mathbb{Q}}\backslash G_\mathbb{Q}}f(x^{-1}\delta^{-1}\gamma\delta x)\\
\cdot(1+\sum_{P\neq G}\sum_{s\in\Omega(\mathfrak{a}_{1111}, P)}(-1)^{\rm{dim}\ Z\backslash A}\hat{\tau}_{P}(H_0(w_s\delta x)- T)). 
\end{eqnarray*}
The integral $\int_{Z_\infty^+G_\mathbb{Q}\backslash G_\mathbb{A}}|I_{\rm{unram}}^{\mathfrak{o}^0_{1111}}(f, x, T)|dx$ is bounded by 
\begin{eqnarray*}
\frac{1}{24}\sum_{\gamma\in\{M_{t, 1111}^{\mathfrak{o}_{1111}}\}}(n_{\gamma, M})^{-1}\int_{Z_\infty^+M(\gamma)_{1111, \mathbb{Q}}\backslash G_\mathbb{A}}|f(x^{-1}\gamma x)|\\
\cdot((1+\sum_{P\neq G}\sum_{s\in\Omega(\mathfrak{a}_{1111}, P)}(-1)^{\rm{dim}\ Z\backslash A}\hat{\tau}_{P}(H_0(w_s\delta x)- T)))dx. 
\end{eqnarray*}
It equals
\begin{eqnarray*}
\frac{c_{P_{1111}}}{24}\sum_{\gamma\in\{M_{t, 1111}^{\mathfrak{o}_{1111}}\}}(n_{\gamma, M})^{-1}\int_{K}\int_{A_{1111, \infty^+}M(\gamma)_{1111, \mathbb{Q}}\backslash P_{1111, \mathbb{A}}}\int_{Z_\infty^+\backslash A_{1111, \infty}^+}|f(k^{-1}p^{-1}\gamma pk)|\\
\cdot((1+\sum_{P\neq G}\sum_{s\in\Omega(\mathfrak{a}_{1111}, P)}(-1)^{\rm{dim}\ Z\backslash A}\hat{\tau}_{P}(H_0(w_sap)- T)))da\ dp\ dk. 
\end{eqnarray*}
Then the integral becomes
\begin{eqnarray*}
    \frac{c_{P_{1111}}}{24}\sum_{\gamma\in\{M_{t, 1111}^{\mathfrak{o}_{1111}^0}\}}\tilde{\tau}(\gamma, M)\int_{K}\int_{M(\gamma)_{1111, \mathbb{A}}\backslash P_{1111, \mathbb{A}}}|f(k^{-1}p^{-1}\gamma pk)|\\
    \cdot\int_{Z_\infty^+\backslash A_{1111, \infty}^+}((1+\sum_{P\neq G}\sum_{s\in\Omega(\mathfrak{a}_{1111}, P)}(-1)^{\rm{dim}\ Z\backslash A}\hat{\tau}_{P}(H_0(w_sap)- T)))da\ dp\ dk. 
\end{eqnarray*}
The sum over $\gamma$ is finite by Lemma \ref{l9}. Also, since the function \[f^K(p)=\int_{K}f(k^{-1}pk)dk,\quad p\in P_{1111, \mathbb{A}}\] has compact support, the integral on $M(\gamma)_{1111, \mathbb{A}}\backslash P_{1111, \mathbb{A}}$ can be taken over a compact set, by Lemma \ref{l10}.  For any $p$, the function \[a\longrightarrow 1+\sum_{P\neq G}\sum_{s\in\Omega(\mathfrak{a}_{1111}, P)}(-1)^{dim\ Z\backslash A}\hat{\tau}_{P}(H_0(w_sap)- T),\quad a\in Z_\infty^+\backslash A_{1111, \infty}^+, \] has compact support. $I_{\rm{unram}}^{\mathfrak{o}_{1111}^0}(f, x, T)$ is integrable over $Z_\infty^+G_\mathbb{Q}\backslash G_\mathbb{A}$, and its integral is
\begin{eqnarray*}
    \frac{c_{P_{1111}}}{24}\sum_{\gamma\in\{M_{t, 1111}^{\mathfrak{o}_{1111}}\}}\tilde{\tau}(\gamma, M)\int_{K}\int_{N_{1111, \mathbb{A}}}\int_{M(\gamma)_{1111, \mathbb{A}}\backslash M_{1111, \mathbb{A}}}f(k^{-1}n^{-1}m^{-1}\gamma mnk)\\
    \cdot \int_{Z_\infty^+\backslash A_{1111, \infty}^+}((1+\sum_{P\neq G}\sum_{s\in\Omega(\mathfrak{a}_{1111}, P)}(-1)^{\rm{dim}\ Z\backslash A}\hat{\tau}_{P}(H_0(w_samnk)- T)))da\ dm\ dn\ dk. 
\end{eqnarray*}
The volume is 
\[\frac{a_{P_{1111}}}{6}\sum_{P\in P(A_{1111})}\frac{<\lambda_0, T_P-H_P(\delta x)>^3}{\Pi_{\eta\in\Phi_P}<\lambda_0, \eta>},\quad \lambda\in\mathfrak{a}_{G}\backslash \mathfrak{a}_{1111}. \]
Thus the integral is
\begin{eqnarray}\label{123}
    \frac{c_{P_{1111}}}{144}\cdot a_{P_{1111}}\sum_{\gamma\in\{M_{t, 1111}^{\mathfrak{o}_{1111}}\}}\tilde{\tau}(\gamma, M)\int_{K}\int_{N_{1111, \mathbb{A}}}\int_{M(\gamma)_{1111, \mathbb{A}}\backslash M_{1111, \mathbb{A}}}f(k^{-1}n^{-1}m^{-1}\gamma mnk)\\
    \cdot v_{M_{1111}}(x, T)dm\ dn\ dk. \notag
\end{eqnarray}
By the $(G, M)$-family, this term is the sum of 
\begin{eqnarray}
    \sum_{s\in \Omega(\mathfrak{a}_{1111}, \mathfrak{a}_{1111})}\frac{<\lambda_0, s^{-1}T>^3}{\Pi_{\eta\in\Phi_P}<\lambda_0, s^{-1}\eta>}\frac{c_{P_{1111}}}{144}\cdot a_{P_{1111}}\sum_{\gamma\in\{M_{t, 1111}^{\mathfrak{o}_{1111}^0}\}}\tilde{\tau}(\gamma, M)\label{124}\\
    \int_{K}\int_{N_{1111, \mathbb{A}}}\int_{M(\gamma)_{1111, \mathbb{A}}\backslash M_{1111, \mathbb{A}}}f(k^{-1}n^{-1}m^{-1}\gamma mnk)dm\ dn\ dk \notag\\
    +\sum_{s\in \Omega(\mathfrak{a}_{1111}, \mathfrak{a}_{1111})}\frac{c_{P_{1111}}}{144}\cdot a_{P_{1111}}\sum_{\gamma\in\{M_{1111}^{\mathfrak{o}_{1111}}\}}\tilde{\tau}(\gamma, M)\int_{K}\int_{N_{1111, \mathbb{A}}}\int_{M(\gamma)_{1111, \mathbb{A}}\backslash M_{1111, \mathbb{A}}}\label{125}\\\notag
    f(k^{-1}n^{-1}m^{-1}\gamma mnk)\frac{<\lambda_0, -s^{-1}H_{0}(w_smnk)>^3}{\Pi_{\eta\in\Phi_P}<\lambda_0, s^{-1}\eta>}dm\ dn\ dk, 
\end{eqnarray}
and some terms equal zero.

Change the variable of integration on $N_{1111, \mathbb{A}}$ of (\ref{124}), the term becomes 
\begin{eqnarray*}
    \sum_{s\in \Omega(\mathfrak{a}_{1111}, \mathfrak{a}_{1111})}\frac{<\lambda_0, s^{-1}T>^3}{\Pi_{\eta\in\Phi_P}<\lambda_0, s^{-1}\eta>}\frac{c_{P_{1111}}}{144}\cdot a_{P_{1111}}\sum_{\gamma\in\{M_{t, 1111}^{\mathfrak{o}^0_{1111}}\}}\tilde{\tau}(\gamma, M)\int_{K}\int_{M(\gamma)_{1111, \mathbb{A}}\backslash M_{1111, \mathbb{A}}}\int_{N_{1111, \mathbb{A}}}\\\notag
    f(k^{-1}m^{-1}\gamma nmk)dn\ dm\ dk. 
\end{eqnarray*}
Then by Lemma \ref{l2}, this term equals 
\begin{eqnarray*}
    \sum_{s\in \Omega(\mathfrak{a}_{1111}, \mathfrak{a}_{1111})}\frac{<\lambda_0, s^{-1}T>^3}{\Pi_{\eta\in\Phi_P}<\lambda_0, s^{-1}\eta>}\frac{c_{P_{1111}}}{144}\cdot a_{P_{1111}}\sum_{\gamma\in\{M_{t, 1111}^{\mathfrak{o}^0_{1111}}\}}\tilde{\tau}(\gamma, M)\\\notag
    \int_{K}\int_{N_{1111, \mathbb{A}}}\int_{M(\gamma)_{1111, \mathbb{A}}\backslash M_{1111, \mathbb{A}}}f(k^{-1}m^{-1}\gamma nmk)\rm{exp}(-<2\rho_{P_{1111}}, H_{0}(m)>)dn\ dm\ dk, 
\end{eqnarray*}
that is 
\begin{eqnarray}
    \sum_{s\in \Omega(\mathfrak{a}_{1111}, \mathfrak{a}_{1111})}\frac{<\lambda_0, s^{-1}T>^3}{\Pi_{\eta\in\Phi_P}<\lambda_0, s^{-1}\eta>}\frac{c_{P_{1111}}}{144}\cdot a_{P_{1111}}\sum_{\gamma\in \{M_{t, 1111}^{\mathfrak{o}^0_{1111}}\}}(n_{\gamma, M})^{-1}\\\notag
    \int_{K}\int_{A_{1111, \infty}^+M(\gamma)_{1111, \mathbb{Q}}\backslash M_{1111, \mathbb{A}}}\int_{N_{1111, \mathbb{A}}}f(k^{-1}m^{-1}\gamma nmk)\rm{exp}(-<2\rho_{P_{1111}}, H_{0}(m)>)dn\ dm\ dk. 
\end{eqnarray}
For $\gamma\in M_{1111}^{\mathfrak{o}_{1111}^k},$ $\sum_{k}I_{\rm{ram}}^{\mathfrak{o}_{1111}^k}(f, x, T)$ equals the sum of 
\begin{eqnarray}\label{1211}
\rm{lim}_{\lambda\rightarrow 0}\int_{Z_\infty^+G_\mathbb{Q}\backslash G_\mathbb{A}}D_\lambda\{\lambda\mu_{\mathfrak{o}_{1111}^4}(\lambda, f)\}
\end{eqnarray}
\begin{eqnarray}\label{1212}
    +\frac{<\lambda_0, T>^3}{\Pi_{\eta\in\Phi_P}<\lambda_0, \eta>}\sum_k\frac{c_{P_{1111}}}{6}\cdot a_{P_{1111}}\sum_{\gamma\in \{M_{n, 1111}^{\mathfrak{o}^k_{1111}}\}}(n_{\gamma, M})^{-1}\\\notag
    \int_{K}\int_{N_{1111, \mathbb{A}}}\int_{A_{1111, \infty}^+M(\gamma)_{1111, \mathbb{Q}}\backslash M_{1111, \mathbb{A}}}f(k^{-1}m^{-1}\gamma nmk)\rm{exp}(-<2\rho_{P_{1111}}, H_{0}(m)>)dm\ dn\ dk
\end{eqnarray}
and the terms obtained by Lemma \ref{8.1} and Lemma \ref{l82}, 
\begin{eqnarray}\label{1213}
    &+\hat{\alpha}_3(T)\sum_{k}a_{P_{31}}c_{P_{31}}\int_K\int_{A_{31, \infty}^+M_{31, \mathbb{Q}}\backslash M_{31, \mathbb{A}}} \sum_{\gamma\in M_{31}^{\mathfrak{o}_{1111}^k}}\int_{N_{31, \mathbb{A}}}\\\notag
    &f(k^{-1}m^{-1}\gamma nmk)\rm{exp}(-<2\rho_{P_{31}}, H_0(m)>)dn\ dm\ dk
\end{eqnarray}

\begin{eqnarray}\label{1215}
    &+\hat{\alpha}_1(T)\sum_{k}a_{P_{13}}c_{P_{13}}\int_K\int_{A_{13, \infty}^+M_{13, \mathbb{Q}}\backslash M_{13, \mathbb{A}}} \sum_{\gamma\in M_{13}^{\mathfrak{o}_{1111}^k}}\int_{N_{13, \mathbb{A}}}\\\notag
    &f(k^{-1}m^{-1}\gamma nmk)\rm{exp}(-<2\rho_{P_{13}}, H_0(m)>)dn\ dm\ dk
\end{eqnarray}

\begin{eqnarray}\label{1217}
    &+\hat{\alpha}_2(T)\sum_ka_{P_{22}}c_{P_{22}}\int_K\int_{A_{22, \infty}^+M_{22, \mathbb{Q}}\backslash M_{22, \mathbb{A}}} \sum_{\gamma\in M_{22}^{\mathfrak{o}_{1111}^k}}\int_{N_{22, \mathbb{A}}}\\\notag
    &f(k^{-1}m^{-1}\gamma nmk)\rm{exp}(-<2\rho_{P_{22}}, H_0(m)>)dn\ dm\ dk
\end{eqnarray}

\begin{eqnarray}\label{1219}
    &+\sum_k\sum_{P\in\mathfrak{P}_{211}}\frac{<\lambda_0, T>^2}{\Pi_{\eta\in\Phi_P}<\lambda_0, \eta>}\frac{c_{P}}{2}\cdot a_{P}\sum_{\gamma\in M^{\mathfrak{o}_{1111}^k}}\\\notag
    &\int_{K}\int_{N_{\mathbb{A}}}\int_{A_{\infty}^+M_{\mathbb{Q}}\backslash M_{\mathbb{A}}}f(k^{-1}m^{-1}\gamma nmk)\rm{exp}(-<2\rho_{P}, H_0(m)>)dn\ dm\ dk
\end{eqnarray}

\begin{eqnarray}
    &+\sum_{\substack{\rm{ramified}\ \mathfrak{o}_{1111}^{k}\\\mathfrak{o}\neq \mathfrak{o}_{1111}^4}}c_{P_{\{\mathfrak{o}\}}}a_{P_{\{\mathfrak{o}\}}}\sum_{\gamma\in M_{t,\{\mathfrak{o}\}}^\mathfrak{o}}\tilde{\tau}(\gamma,M)\int_K\int_{N_{\{\mathfrak{o}\},\mathbb{A}}}\int_{M(\gamma)_{\{\mathfrak{o}\},\mathbb{A}}\backslash M_{\{\mathfrak{o}\},\mathbb{A}}}\\\notag
    &f(k^{-1}n^{-1}m^{-1}\gamma mnk)v_{M_{\{\mathfrak{o}\}}}(m)dm\ dn\ dk.
\end{eqnarray}
\subsection{The third parabolic term}
In this section, we shall prove that the integral of the first parabolic term associated to $\mathfrak{o}_{1111}^k$ can be canceled by the integrals of $K''_{P_{1111}}(f,x,T)$.

The integral of of $-K_{P_{1111}}''(f, x, T)$ is 
\begin{eqnarray*}
    -\frac{1}{192(\pi i)^3}\sum_{\alpha, \beta\in I_{P_{1111}}}\int_{i\mathfrak{a}_{G}\backslash i\mathfrak{a}_{1111}}\int_{Z_\infty^+G_\mathbb{Q}\backslash G_\mathbb{A}}E''^{T}_{P_{1111}}(\Phi_\alpha, \lambda, x)\overline{E''^{T}_{P_{1111}}(\Phi_\beta, \lambda, x)}dx\ d\lambda. 
\end{eqnarray*}
\begin{lemma}
    For $\alpha, \beta\in I_{P_{1111}}$ and $\lambda$ a nonzero imaginary number in $i\mathfrak{a}_{G}\backslash i\mathfrak{a}_{1111}$, the integral \[\int_{Z_\infty^+G_\mathbb{Q}\backslash G_\mathbb{A}}E''^{T}_{P_{1111}}(\Phi_\alpha, \lambda, x)\overline{E''^{T}_{P_{1111}}(\Phi_\beta, \lambda, x)}dx\] is
    \begin{eqnarray}
        &\frac{a_{P_{1111}}}{6}\sum_{s\in\Omega(\mathfrak{a}_{1111}, \mathfrak{a}_{1111})}\frac{<\lambda_0, s^{-1}T>^3}{\Pi_{\eta\in\Phi_P}<\lambda_0, s^{-1}\eta>}(\Phi_\alpha, \Phi_\beta)\label{1221}\\
        &+\frac{a_{P_{1111}}}{6}\sum_{t\in\Omega(\mathfrak{a}_{1111}, \mathfrak{a}_{1111})}(M_{P_{1111}}(t^{-1}, t\lambda)D_\lambda M_{P_{1111}}(t, \lambda)\Phi_\alpha, \Phi_\beta)\label{1222}\\
        &+\frac{a_{P_{1111}}}{6}\sum_{s\neq t\in\Omega(\mathfrak{a}_{1111}, \mathfrak{a}_{1111})}\frac{\rm{exp}(<t\lambda-s\lambda, T>(M_{P_{1111}}(t, \lambda)\Phi_\alpha), M_{P_{1111}}(s, \lambda)\Phi_\beta)}{\Pi_{\eta\in \Phi_P}<t\lambda-s\lambda, \eta>}. \label{1223}
    \end{eqnarray}
\end{lemma}
This prove is similar to Lemma \ref{lemma11.1}.

The term corresponding to (\ref{1222}) is
\begin{eqnarray}\label{1224}
&-\frac{a_{P_{1111}}}{1152(\pi i)^3}\sum_{s\in\Omega(\mathfrak{a}_{1111}, \mathfrak{a}_{1111})}\sum_{\chi}\int_{i\mathfrak{a}_{G}\backslash i\mathfrak{a}_{1111}}\\\notag
&\rm{tr}\{M(s^{-1}, s\lambda)\cdot(D_\lambda M_{P_{1111}}(s, \lambda))\cdot\pi_{P_{1111}, \chi}(\lambda, f)\}d\lambda. 
\end{eqnarray}
This term is finite.

Then we substitute (\ref{1221}) into $K_{P_{1111}}''(f, x, T)$, it equals
\begin{eqnarray*}
    \frac{a_{P_{1111}}}{6}\sum_{s\in\Omega(\mathfrak{a}_{1111}, \mathfrak{a})}\frac{<\lambda_0, s^{-1}T>^3}{\Pi_{\eta\in\Phi_P}<\lambda_0, s^{-1}\eta>}\int_{i\mathfrak{a}_{G}\backslash \mathfrak{a}_{1111}}\rm{tr}\,\pi_{P_{1111}}(\lambda, f)d\lambda. 
\end{eqnarray*}
we can write it as
\begin{eqnarray*}
    c_{P_{1111}}\cdot\frac{a_{P_{1111}}}{192(\pi i)^3}\sum_{s\in\Omega(\mathfrak{a}_{1111}, \mathfrak{a})}\frac{<\lambda_0, s^{-1}T>^3}{\Pi_{\eta\in\Phi_P}<\lambda_0, \eta>}\\
    \int_{i\mathfrak{a}_{G}\backslash \mathfrak{a}_{1111}}\int_{A_{1111, \infty}^+M_{1111, \mathbb{Q}}\backslash M_{1111, \mathbb{A}}}P_{P_{1111}}(\lambda, f, mk, mk)dm\ dk\ d\lambda, 
\end{eqnarray*}by the continity of $P_{P_{1111}}$.

Apply the Fourier inversion formula, we obtain
\begin{eqnarray*}
    \frac{c_{P_{1111}}}{24}\cdot a_{P_{1111}}\sum_{s\in\Omega(\mathfrak{a}_{1111}, \mathfrak{a}_{1111})}\frac{<\lambda_0, s^{-1}T>^3}{\Pi_{\eta\in\Phi_P}<\lambda_0, s^{-1}\eta>}\\
    \int_K\int_{A_{1111, \infty}^+M_{1111, \mathbb{Q}}\backslash M_{1111, \mathbb{A}}}\sum_{\gamma\in M_{1111, \mathbb{Q}}}\int_{N_{1111, \mathbb{A}}}f(k^{-1}m^{-1}\gamma n mk)\rm{exp}(-<2\rho_{P_{1111}}, H_{0}(m)>)dn\ dm\ dk. 
\end{eqnarray*}
Hence (\ref{124}), (\ref{1212}) can be canceled.

Now consider the term (\ref{1223}). We put it into the function $-K''_{P_{1111}}(f, x, T)$.

We write it as 
\begin{eqnarray*}
    \frac{a_{P_{1111}}}{192(\pi i)^3}\sum_{\chi}\int_{i\mathfrak{a}_{G}\backslash \mathfrak{a}_{1111}}\sum_{s\neq t\in\Omega(\mathfrak{a}_{1111}, \mathfrak{a}_{1111})}\frac{\rm{exp}(<t\lambda-s\lambda, T>)(M_{P_{1111}}(t, \lambda)\Phi_\alpha, M_{P_{1111}}(s, \lambda)\Phi_\beta)}{\Pi_{\eta\in \Phi_P}<t\lambda-s\lambda, \eta>}d\lambda. 
    \end{eqnarray*}
For every term above, the sum over $\beta$ is finite.

Thus this term is 
\begin{eqnarray}\label{1225}
\frac{a_{P_{1111}}}{192(\pi i)^3}\sum_{s\neq t\in\Omega(\mathfrak{a}_{1111}, \mathfrak{a}_{1111})}\iota_{P_{1111}}(s, t). 
\end{eqnarray}
Of course, the calculation is the same as we said in the last section, and its result has no $T$.

\begin{lemma}
    The sum \[I_{\rm{unram}}^{\mathfrak{o}_{1111}^0}(f,x,T)+\sum_k I_{\rm{ram}}^{\mathfrak{o}_{1111}^k}(f,x,T)\] equals the sum of (\ref{125}), (\ref{1211}),  (\ref{1224}), (\ref{1225}). 
\end{lemma}

Now, all the second parabolic terms associated to different parabolic subgroups contain $P_{1111}$ can be canceled by the third parabolic terms.

If we use the notation of Arthur wrote in \cite{A4}, we can write \begin{eqnarray*}
    J_{\rm{geo}}(f)=J_{\rm{spec}}(f)
\end{eqnarray*}
as 
\begin{eqnarray*}
    J_{\rm{geo}}^{\rm{d}}(f)+J_{\rm{geo}}^{\rm{c}}(f)=J_{\rm{spec}}^{\rm{d}}(f)+J_{\rm{spec}}^{\rm{c}}(f),
\end{eqnarray*}
where $\rm{d}$ means divergent which is associated to $T$, and $\rm{c}$ means convergant.

We have 
\begin{theorem}\label{thm1}
    For any $f\in C_c^\infty(Z_\infty^+\backslash G_\mathbb{A})$,\[J_{\rm{geo}}^{\rm{c}}(f)=J_{\rm{spec}}^{\rm{c}}(f).\]
\end{theorem}
\section{summary}
So far, we finished the calculation of $\rm{tr}\ \rm{R}_0(f)$, it equals the integral over $Z_\infty^+G_\mathbb{Q}\backslash G_\mathbb{A}$ of \[K(x, x)-K_1(x, x). \]We have proved that the first parabolic terms approaches $0$ as $T$ approaches $\infty$, the sum of the second and third parabolic terms are what remain. 

\begin{theorem}\label{thm2}
    For ramified orbits, the integrals of the kernel over $Z_\infty^+G_\mathbb{Q}\backslash G_\mathbb{A}$ is the sum
    \begin{eqnarray*}
        &\rm{lim}_{\lambda\rightarrow 0}\int_{Z_\infty^+G_\mathbb{Q}\backslash G_\mathbb{A}}D_\lambda\{\lambda\mu_{\mathfrak{o}_{1111}^4}(\lambda,f,x)\}dx
    \end{eqnarray*}
    \begin{eqnarray*}
        &+\rm{lim}_{\lambda\rightarrow 0}\int_{Z_\infty^+G_\mathbb{Q}\backslash G_\mathbb{A}}D_\lambda\{\lambda\mu_{\mathfrak{o}_{22}^2}(\lambda,f,x)\}dx
    \end{eqnarray*}
    \begin{eqnarray*}
        &+\sum_{\substack{\rm{ramified}\ \mathfrak{o},\\\mathfrak{o}\neq\mathfrak{o}_{1111}^4,\mathfrak{o}_{22}^2}}c_{P_{\{\mathfrak{o}\}}}a_{P_{\{\mathfrak{o}\}}}\sum_{\gamma\in M_{r,\{\mathfrak{o}\}}^\mathfrak{o}}\tilde{\tau}(\gamma,M)\int_K\int_{N_{\{\mathfrak{o}\},\mathbb{A}}}\int_{M(\gamma)_{\{\mathfrak{o}\},\mathbb{A}}\backslash M_{\{\mathfrak{o}\},\mathbb{A}}}\\\notag
        &f(k^{-1}n^{-1}m^{-1}\gamma mnk)v_{M_{\{\mathfrak{o}\}}}(m)dm\ dn\ dk.
    \end{eqnarray*}
\end{theorem}

Finally, we have 
\begin{theorem}\label{result}
    For any $f\in C_c^\infty(Z_\infty^+\backslash G_\mathbb{A})$, the trace of $\rm{R}_0(f)$ is the sum
\begin{eqnarray*}
    &\sum_{\gamma\in G_e}\tilde{\tau}(\gamma, G)\int_{G(\gamma)_\mathbb{A}\backslash G_\mathbb{A}}f(x^{-1}\gamma x)dx
\end{eqnarray*}
 the term of G-elliptic
\begin{eqnarray*}
    &-c_{P_{31}}a_{P_{31}}\sum_{\gamma\in\{M_{t,31}^{\mathfrak{o}^0_{31}}\}}\tilde{\tau}(\gamma, M)\int_{K}\int_{N_{31, \mathbb{A}}}\int_{M(\gamma)_{31, \mathbb{A}}\backslash M_{31, \mathbb{A}}}f(k^{-1}n^{-1}m^{-1}\gamma mnk)\\
    &\cdot \hat{\alpha}_1(H_{0}(w_{(14)}n))dm\ dn\ dk
\end{eqnarray*}
\begin{eqnarray*}
    &-c_{P_{31}}a_{P_{31}}\sum_{\gamma\in\{M_{t,31}^{\mathfrak{o}^{31}_{1111}}\}}\tilde{\tau}(\gamma, M)\int_{K}\int_{N_{31, \mathbb{A}}}\int_{M(\gamma)_{31, \mathbb{A}}\backslash M_{31, \mathbb{A}}}f(k^{-1}n^{-1}m^{-1}\gamma mnk)\\
    &\cdot \hat{\alpha}_1(H_{0}(w_{(14)}n))dm\ dn\ dk 
\end{eqnarray*}
\begin{eqnarray*}
    &+\frac{1}{2\pi i}\sum_{\chi}\int_{i\mathfrak{a}_{G}\backslash i\mathfrak{a}_{31}}\rm{tr}\{M_{P_{31}}((14), (14)\lambda)\cdot(\frac{d}{d\lambda}M_{P_{31}}((14), \lambda))\cdot\pi_{P_{31}, \chi}(\lambda, f)\}d\lambda\\
    &+\frac{1}{2\pi i}\sum_{\chi}\int_{i\mathfrak{a}_{G}\backslash i\mathfrak{a}_{13}}\rm{tr}\{M_{P_{13}}((14), (14)\lambda)\cdot(\frac{d}{d\lambda}M_{P_{13}}((14), \lambda))\cdot\pi_{P_{13}, \chi}(\lambda, f)\}d\lambda\\
\end{eqnarray*}
the terms from $P_{31}$, we write the geometric terms which are the first two terms of it $J_{M_{31}}(\gamma, f)$
\begin{eqnarray*}
    &-\frac{c_{P_{22}}}{2}\sum_{\gamma\in\{M_{t,22}^{\mathfrak{o}_{22}}\}}\tilde{\tau}(\gamma, M)\int_{K}\int_{N_{22, \mathbb{A}}}\int_{M(\gamma)_{22, \mathbb{A}}\backslash M_{22, \mathbb{A}}}f(k^{-1}n^{-1}m^{-1}\gamma mnk)\\
    &\cdot \hat{\alpha}_2(H_0(w_{(14)(23)}n))dm\ dn\ dk
\end{eqnarray*}
\begin{eqnarray*}
    &-\frac{c_{P_{22}}}{2}\sum_{\gamma\in\{M_{t,22}^{\mathfrak{o}_{1111}^{22}}\}}\tilde{\tau}(\gamma, M)\int_{K}\int_{N_{22, \mathbb{A}}}\int_{M(\gamma)_{22, \mathbb{A}}\backslash M_{22, \mathbb{A}}}f(k^{-1}n^{-1}m^{-1}\gamma mnk)\\
    &\cdot \hat{\alpha}_2(H_0(w_{(14)(23)}n))dm\ dn\ dk
\end{eqnarray*}
\begin{eqnarray*}
    &-\frac{c_{P_{22}}}{2}\sum_{\gamma\in\{M_{t,22}^{\mathfrak{o}_{211}^2}\}}\tilde{\tau}(\gamma, M)\int_{K}\int_{N_{22, \mathbb{A}}}\int_{M(\gamma)_{22, \mathbb{A}}\backslash M_{22, \mathbb{A}}}f(k^{-1}n^{-1}m^{-1}\gamma mnk)\\
    &\cdot \hat{\alpha}_2(H_0(w_{(14)(23)}n))dm\ dn\ dk
\end{eqnarray*}
\begin{eqnarray*}
    &-\rm{lim}_{\lambda\rightarrow 0}\int_{Z_\infty^+G_\mathbb{Q}\backslash G_\mathbb{A}}D_\lambda\{\lambda\mu_{\mathfrak{o}_{22}^2}(\lambda,f,x)\}dx
\end{eqnarray*}
\begin{eqnarray*}
    &+\frac{1}{4\pi i}\sum_{\chi}\int_{i\mathfrak{a}_{G}\backslash i\mathfrak{a}_{22}}\rm{tr}\{M_{P_{22}}((14)(23), (14)(23)\lambda)\cdot(\frac{d}{d\lambda}M_{P_{22}}((14)(23), \lambda))\cdot\pi_{P_{22}, \chi}(\lambda, f)\}d\lambda
\end{eqnarray*}
\begin{eqnarray*}
    &-\frac{1}{4}\rm{tr}\{M_{P_{22}}((14)(23), 0)\pi_{P_{22}}(0, f)\}
\end{eqnarray*}
the terms from $P_{22}$, write the first four terms $J_{M_{22}}(\gamma, f)$
\begin{eqnarray*}
    &+\sum_{P\in\mathfrak{P}_{211}}\sum_{s\in\Omega(\mathfrak{a}_{211}, \mathfrak{a})}\frac{c_{P_{211}}}{4}\cdot a_{P_{211}}\sum_{\gamma\in\{M_{t,211}^{\mathfrak{o}_{211}^0}\}}\tilde{\tau}(\gamma, M)\\
    &\int_{K}\int_{N_{211,\mathbb{A}}}\int_{M(\gamma)_{211,\mathbb{A}}\backslash M_{211,\mathbb{A}}}f(k^{-1}n^{-1}m^{-1}\gamma mnk)\frac{<\lambda_0, -H_0(w_smnk)>^2}{\Pi_{\eta\in\Phi_P}<\lambda_0, s^{-1}\eta>}dm\ dn\ dk
\end{eqnarray*}
\begin{eqnarray*}
    &+\sum_{P\in\mathfrak{P}_{211}}\sum_{s\in\Omega(\mathfrak{a}_{211}, \mathfrak{a})}\frac{c_{P_{211}}}{4}\cdot a_{P_{211}}\sum_{\gamma\in\{M_{t,211}^{\mathfrak{o}_{1111}^{211}}\}}\tilde{\tau}(\gamma, M)\\
    &\int_{K}\int_{N_{211,\mathbb{A}}}\int_{M(\gamma)_{211,\mathbb{A}}\backslash M_{211,\mathbb{A}}}f(k^{-1}n^{-1}m^{-1}\gamma mnk)\frac{<\lambda_0, -H_0(w_smnk)>^2}{\Pi_{\eta\in\Phi_P}<\lambda_0, s^{-1}\eta>}dm\ dn\ dk
\end{eqnarray*}
\begin{eqnarray*}
    &-\frac{a_{P_{211}}}{48(\pi i)^2}\sum_{\chi}\int_{i\mathfrak{a}_{G}\backslash i\mathfrak{a}_{211}}\rm{tr}\{M_{P_{211}}((34), (34)\lambda)\cdot(D_\lambda M_{P_{211}}((34), \lambda))\cdot\pi_{P_{211}, \chi}(\lambda, f)\}d\lambda
\end{eqnarray*}
\begin{eqnarray*}
    &-\frac{a_{P_{211}}}{48(\pi i)^2}\sum_{\chi}\int_{i\mathfrak{a}_{G}\backslash i\mathfrak{a}_{211}}\rm{tr}\{M_{P_{211}}((143), (134)\lambda)\cdot(D_\lambda M_{P_{211}}((134), \lambda))\cdot\pi_{P_{211}, \chi}(\lambda, f)\}d\lambda
\end{eqnarray*}
\begin{eqnarray*}
    &-\frac{a_{P_{211}}}{48(\pi i)^2}\sum_{\chi}\int_{i\mathfrak{a}_{G}\backslash i\mathfrak{a}_{211}}\rm{tr}\{M_{P_{211}}((14)(23), (14)(23)\lambda)\cdot(D_\lambda M_{P_{211}}((14)(23), \lambda))\cdot\pi_{P_{211}, \chi}(\lambda, f)\}d\lambda
\end{eqnarray*}
\begin{eqnarray*}
    &-\sum_{P\in\mathfrak{P}_{211}}\frac{a_{P}}{16(\pi i)^2}\sum_{s\neq t\in\Omega(\mathfrak{a}_{\mathfrak{o}}, \mathfrak{a})}\iota_{P_{211}}(s, t)\\
\end{eqnarray*}
the terms from $P_{211}$, we write the first three terms $J_{M_{211}}(\gamma, f)$
\begin{eqnarray*}
    &+\sum_{s\in \Omega(\mathfrak{a}_{1111}, \mathfrak{a}_{1111})}\frac{c_{P_{1111}}}{144}\cdot a_{P_{1111}}\sum_{\gamma\in\{M_{t,1111}^{\mathfrak{o}^0_{1111}}\}}\tilde{\tau}(\gamma, M)\int_{K}\int_{N_{1111, \mathbb{A}}}\int_{M(\gamma)_{1111, \mathbb{A}}\backslash M_{1111, \mathbb{A}}}\\
    &f(k^{-1}n^{-1}m^{-1}\gamma mnk)\frac{<\lambda_0, -s^{-1}H_{0}(w_smnk)>^3}{\Pi_{\eta\in\Phi_P}<\lambda_0, s^{-1}\eta>}dm\ dn\ dk\\
    &+\rm{lim}_{\lambda\rightarrow 0}\int_{Z_\infty^+G_\mathbb{Q}\backslash G_\mathbb{A}}D_\lambda\{\lambda\mu_{\mathfrak{o}_{1111}^4}(\lambda,f,x)\}dx\\
    &-\frac{a_{P_{1111}}}{96}\sum_{s\neq t\in\Omega(\mathfrak{a}_{1111}, \mathfrak{a}_{1111})}\rm{tr}\{M_{P_{1111}}(t^{-1}, 0)M_{P_{1111}}(s, 0)\pi_{P_{1111}}(0, f)\}\\
    &-\frac{a_{P_{1111}}}{192(\pi i)^3}\sum_{s\neq t\in\Omega(\mathfrak{a}_{1111}, \mathfrak{a}_{1111})}\iota_{P_{1111}}(s, t)\\
\end{eqnarray*}
the terms from $P_{1111}$, we write the first two terms $J_{M_{1111}}(\gamma, f)$. 
\end{theorem}

\end{document}